\documentclass[12pt,twoside,a4paper]{amsart}
\usepackage{amssymb}
\usepackage{amsmath,amsfonts}
\usepackage[showonlyrefs]{mathtools}
\usepackage{bm}
\usepackage{tikz-cd}
\usepackage[
	colorlinks	= true,
	allcolors 	= blue
	]{hyperref}

\usepackage{amsthm}
\theoremstyle	{plain}
\newtheorem		{theorem}					{Theorem}	[section]
\newtheorem		{lemma}			[theorem]	{Lemma}
\newtheorem		{corollary}		[theorem]	{Corollary}
\newtheorem		{proposition}	[theorem]	{Proposition}

\theoremstyle	{definition}
\newtheorem		{definition}	[theorem]	{Definition}
\newtheorem		{remark}		[theorem]	{Remark}
\newtheorem*	{remark*}					{Remark}

\numberwithin{equation}{section}

\newcommand{\cC}{{\mathcal C}}

\newcommand{\cH}{{\mathcal H}}

\newcommand{\cK}{{\mathcal K}}
\newcommand{\cL}{{\mathcal L}}

\def\sA{{\mathfrak A}}      
      
   \def\sH{{\mathfrak H}}   
      \def\sL{{\mathfrak L}}
   \def\sN{{\mathfrak N}}

\def\gh{{\mathfrak h}}

\newcommand {\ff}{\mathfrak f}
\newcommand {\fg}{\mathfrak g}

\def\bR{{\mathbf R}}
\def\bA{{\mathbf A}}

      \def\dC{{\mathbb C}}

   \def\dN{{\mathbb N}}   
      \def\dR{{\mathbb R}}

   \def\cB{{\mathcal B}}   \def\cC{{\mathcal C}}
      \def\cF{{\mathcal F}}
\def\cG{{\mathcal G}}   \def\cH{{\mathcal H}}   
\def\cJ{{\mathcal J}}   \def\cK{{\mathcal K}}   \def\cL{{\mathcal L}}
   \def\cN{{\mathcal N}}   
\def\cP{{\mathcal P}}   \def\cQ{{\mathcal Q}}   \def\cR{{\mathcal R}}
      
   \def\cW{{\mathcal W}}

\newcommand{\ptp}{p\times p}

\newcommand{\w}[1]{\widetilde{#1}}
\newcommand{\wt}{\widetilde}
\newcommand{\wh}{\widehat}

\DeclareMathOperator{\sgn}{sgn}

\DeclareMathOperator{\ran}{ran}
\DeclareMathOperator{\rank}{rank}
\DeclareMathOperator{\dom}{dom}
\DeclareMathOperator{\mul}{mul}

\DeclareMathOperator{\gr}{gr}

\def\ov{\overline}

\renewcommand{\Re}{\operatorname{Re}}
\renewcommand{\Im}{\operatorname{Im}}

\newcommand{\dsp}{\displaystyle}

\begin{document}

\title[${\sL}$-resolvents of symmetric linear relations in Pontryagin spaces]
{${\sL}$-resolvents of symmetric linear relations in Pontryagin spaces}

\author[V.~Derkach]{Volodymyr Derkach}
\address{ %Boston, USA}
	Department of Mathematics,
	Vasyl Stus Donetsk National University,
	Vinnytsia, Ukraine}
\email{derkach.v@gmail.com}

%%%%%%%%%%%%%%%%%%%%%%%%%%%%%%%%%%%%%%%%%%%%%%%%%%%%%%%%%%%%%%%%%%%%%%%%%%%5

\date{\today}


\begin{abstract}
Let $A$ be
a closed symmetric operator with the deficiency index $(p,p)$, $p<\infty$,
acting in a Hilbert space $\sH$
and let $\sL$ be a subspace of $\sH$.
%The operator-valued function
%$P_\sL(\wt A-\lambda I)\upharpoonright_\sL,$
% where
%$\wt A$ is a selfadjoint extension of $A$ acting in a (possibly larger) space
%and $P_\sL$ is the orthogonal projection onto $\sL$
%is called the  $\sL$-resolvent of $A$.
The set of $\sL$-resolvents of a densely defined symmetric operator 
in a Hilbert space with a proper  gauge $\sL(\subset\sH)$ was described by Kre\u{\i}n and Saakyan.
The Kre\u{\i}n--Saakyan  theory of $\sL$-resolvent matrices was extended by
Shmul'yan and Tsekanovskii to the case of improper  gauge $\sL(\not\subset\sH)$
and by Langer and Textorius to the case of symmetric linear relations in Hilbert spaces.
In the present paper we find connections between the theory of boundary triples
and the Kre\u{\i}n--Saakyan  theory of $\sL$-resolvent matrices
for symmetric linear relations with
improper  gauges in Pontryagin spaces. We extend the known formula for the  $\sL$-resolvent matrix
in terms of boundary operators to this class of relations.
The results are applied
to the minimal linear relation generated by a canonical system.

\end{abstract}


\maketitle

\tableofcontents

\section{Introduction} \label{sec:intro}
Let  $\sH$  be  a Hilbert space and let $A$ be
a closed symmetric linear operator with the deficiency index $(p,p)$, $p<\infty$,
acting in a Hilbert space $\sH$
is said to have a proper gauge $\sL(\subset\sH)$ if the set $\rho(A,\sL)$ of regular type points $\lambda\in\dC$ of $A$ such that the space $\sH$ admits the direct decomposition
\begin{equation}\label{eq:Prop_Gauge}
  \sH=\ran(A-\lambda I)\dotplus \sL
\end{equation}
is non-empty.
The operator-valued function
\begin{equation}\label{eq:Lresolv}
  r(\lambda)=P_\sL(\wt A-\lambda I)\upharpoonright_\sL,\quad \lambda\in\rho(\wt A),
\end{equation}
 where
$\wt A$ is a selfadjoint extension of $A$ acting in a (possibly larger) space
$\wt\sH(\supseteq\sH)$, $P_\sL$ is the orthogonal projection onto $\sL$
 and $\rho(\wt A)$
is the set of regular points of $\wt A$,
is called the  $\sL$-resolvent of $A$.
The set of all $\sL$-resolvents of $A$ was described by M. G. Kre\u{\i}n in \cite{Kr44}
and \cite{Kr46} in the case $p<\infty$
and by S. Saakyan in the case $p=\infty$ \cite{Sa65}
by a formula which in the case $p=1$ takes the form
\begin{equation}\label{eq:LresolvDescr}
  r(\lambda)=(w_{11}({\lambda})\tau({\lambda})+w_{12}({\lambda}))
(w_{21}({\lambda})\tau({\lambda})+w_{22}({\lambda}))^{-1},\quad \lambda\in\rho(\wt A)\cap\rho(A,\sL),
\end{equation}
where $\tau$ ranges over the set $\cN^{p\times p}$ of Nevanlinna families, see Definition~\ref{def:Nk-class}.
The block matrix $W(\lambda)=[w_{ij}(\lambda)]_{i,j=1^2}$
generating this linear fractional transform is called the
$\sL$-resolvent matrix of $A$.
In particular, the
$\sL$-resolvent matrix turns out to be $J_p$-contractive in the upper halfplane $\dC_+$,
i.e.,
\begin{equation}\label{eq:J_p}
J_p-W(\lambda)J_p W(\omega)^*\ge 0\quad\text{for}\quad \lambda\in\rho(A,\sL),
\quad\text{where}\quad
J_p=\begin{bmatrix}
    O_{p} & -i I_{p}\\
    i I_{p} & O_{p}
    \end{bmatrix}.
\end{equation}
The theory of $\sL$-resolvent matrices of symmetric operators $A$ with
proper gauges was developed in \cite{Kr49}, \cite{Sa65} and \cite{KS70}.
An effective tool in the study of extension theory of symmetric operators
is the notion of the boundary triple introduced by J. Calkin~\cite{Cal39} and developed
by A. Kochubei \cite{Koc75}, V. Bruk \cite{Br76} and M. and L. Gorbachuks~\cite{GG91}.
In \cite{DM91}  relations between the Kre\u{\i}n's representation theory
and the theory of boundary triples
were investigated.

However, some problems of the analysis require  consideration of $\sL$-resolvents of symmetric operators $A$ with improper gauges $\sL$ consisting of generalized elements
from a space $\sH_-$ of ``distributions'', see Section~\ref{sec:4.1}.
The theory of $\sL$-resolvent matrices and a description of $\sL$-resolvents of a symmetric operator $A$ with
improper gauge was developed by Yu. Shmul'yan and E. Tsekanovskii, \cite{ShTs77}.
This theory was extended by H.~Langer and B.~Textorius to the case of a symmetric linear relation $A$
and applied to the problem of description of $\sL$-resolvents and spectral functions
of operators associated with a canonical system, see~\cite{LaTe84}, \cite{LaTe85}.

 $\sL$-resolvents  of a Pontryagin space symmetric operator $A$ with deficiency index $(1,1)$ were studied by M. Kaltenback and H. Woracek,~\cite{KW98}. V. Derkach and H. Dym employed  the theory of de Branges--Pontryagin spaces  to describe
 $\sL$-resolvents  of a Pontryagin space symmetric operator $A$ with deficiency index $(p,p)$, see~\cite{DD21}.
 In \cite{DD24} this  description was  used for  parametrization of solutions of indefinite truncated matrix moment  problem.

In the present paper we  consider a symmetric linear relation
in a Pontryagin space with deficiency index $(p,p)$
that has an improper  gauge $\sL$.
Using boundary triple's approach we calculate in Theorem~\ref{thm:ResM}
the $\sL$-resolvent matrix $W_\sL(\lambda)$ of $A$.
%all the objects of the Kre\u{\i}n's representation theory.
Under an additional assumption it is shown, that the $\sL$-resolvent matrix $W(\lambda)$ of $A$  belong to the class
$\cW_{\kappa}(J_p)$ of $2p\times 2p$-matrix-valued functions
such that the kernel
\begin{equation}\label{kerK0}
{\mathsf K}_\omega(\lambda):=
\frac{ J_p-W(\lambda)J_p W(\omega)^*}{-i(\lambda-\ov\omega)}, \quad\lambda,\omega\in \rho(A,{\sL}),\quad \lambda\ne\overline{\omega}
\end{equation}
has $\kappa$ negative squares on $\rho(A,{\sL})$.
The $\sL$-resolvent matrix $W(\lambda)$ allows to describe the set of so-called $\sL$-regular
$\sL$-resolvents of $A$, see Theorem~\ref{prop:Gresolv}.
The set of  $\sL$-regular
$\sL$-resolvents of $A$ is parametrized there by formula~\eqref{eq:Lresolv}
where parameter $\tau$ ranges over a class  of generalized Nevanlinna families.
Necessary facts about the class $\cN_\kappa^{\ptp}$ (resp. $\wt\cN_\kappa^{\ptp}$) of generalized Nevanlinna functions (resp., families) introduced by
M. G. Kre\u{\i}n and H. Langer in~\cite{KL71} are presented in Section~\ref{sec:2}.

The formula~\eqref{eq:Formula_W} for the $\sL$-resolvent matrix $W(\lambda)$ obtained in Theorem~\ref{thm:ResM} seems to be new even for the definite case, $\kappa=0$.
We illustrate this formula on the example of
a minimal linear relation $A$ generated by the canonical system~\eqref{eq:can,eq-n}
considered in~\cite{Orc67}.
A description of generalized resolvents of this relation was given in \cite{LaTe82}.
We present another calculation of the $\sL$-resolvent matrix $W(\lambda)$
based on the formula~\eqref{eq:Formula_W}.
We are going to apply this formula for indefinite canonical system elsewhere.

\section{Preliminaries}\label{sec:2}
\subsection{Linear relations in Pontryagin spaces}
\label{subsec:pre:lr}
Let us recall some definitions from ~\cite{Bog68}, \cite{AI86}, \cite{Arens}, \cite{Ben72}.
An indefinite inner product space
$(\sH,[ \cdot,\cdot]_{\sH})$, see~\cite{Bog68}, is called a Kre\u{\i}n space
if it can be decomposed into a direct orthogonal sum
\begin{equation}\label{eq:1.10}
\sH=\sH_+[+]\sH_-
\end{equation}
of two subspaces $\sH_+$ and $\sH_-$ such that $(\sH_{\pm}, \pm [\cdot,\cdot]_\sH)$
are Hilbert spaces.
 The operator $J=P_+-P_-$, where $P_{\pm}$ are orthogonal projections in $\sH$ onto $\sH_\pm$, is called {\it the fundamental symmetry} of $\sH$. It induces a Hilbert inner product in $\sH$ by the formula
\begin{equation}\label{eq:1.11}
   (f,g)_\sH^2=[J f,g]_{\cK}, \quad f,g\in\sH.
\end{equation}
and turns $\sH$ into a Hilbert space
$(\sH,(\cdot,\cdot)_{\sH})$.
  A Kre\u{\i}n space $(\sH,[ \cdot,\cdot]_{\sH})$ with a finite negative index $\kappa_-(\sH):=\dim\sH_-=\kappa$ is called a {\it Pontryagin space} with negative
index $\kappa$  (or shortly, $\pi_\kappa$-space).

 A linear relation $T$ in a Pontryagin space
 $(\sH,[ \cdot,\cdot]_{\sH})$
is a linear subspace of $\sH \times \sH$.
%Let us recall some basic definitions and
%properties associated with linear relations,
%see~\cite{Arens,Ben72}.
The \emph{domain}, the \emph{range}, the \emph{kernel}, and the \emph{multivalued part} of a linear relation $T$ are defined as follows:
\begin{align}
	\dom{T} &\coloneqq \left\{ f \colon \begin{bmatrix} f \\ g \end{bmatrix} \in T \right\}, &
	\ran{T} &\coloneqq \left\{ g \colon \begin{bmatrix} f \\ g \end{bmatrix} \in T \right\},\\
	\ker{T} &\coloneqq \left\{ f \colon \begin{bmatrix} f \\ 0 \end{bmatrix} \in T \right\}, &
	\mul{T} &\coloneqq \left\{ g \colon \begin{bmatrix} 0 \\ g \end{bmatrix} \in T \right\}.
\end{align}
The \emph{adjoint} linear relation $T^{[*]}$ is defined by
\begin{equation}
	T^{[*]} \coloneqq \left\{
	\begin{bmatrix} u \\ f \end{bmatrix}
	\in \sH \times \sH \colon \langle f,v\rangle_{\sH} = \langle u,g\rangle_{\sH}\ \text{for any}\
	\begin{bmatrix} v \\ g \end{bmatrix}
	\in T \right\}.
\end{equation}
A linear relation $T$ in $\sH$ is called \emph{closed} if $T$ is closed as a subspace of $\sH \times \sH$.
The set of all closed linear operators (relations) is denoted by $\mathcal{C}(\sH)$ ($\wt{\mathcal{C}}(\sH)$).
Identifying a linear operator $T \in \mathcal{C}(\sH)$ with its graph one can consider $\mathcal{C}(\sH)$ as a part of $\wt{\mathcal{C}}(\sH)$.

Let $T$ be a closed linear relation, $\lambda \in \mathbb{C}$, then
	\begin{equation}
		T - \lambda I \coloneqq \left\{ \begin{bmatrix} f \\ g-\lambda f \end{bmatrix} \colon
		\begin{bmatrix} f \\ g \end{bmatrix}
		\in T \right\}.
	\end{equation}
	A point $\lambda \in \mathbb{C}$ such that $\ker{\left( T - \lambda I \right)} = \{0\}$ and $\ran{\left( T - \lambda I \right)} = \sH$ is called a \emph{regular point} of the linear relation $T$. Let $\rho(T)$ be the set of regular points.
	The \emph{point spectrum} $\sigma_p(T)$ of the linear relation $T$
	is defined by
	\begin{equation}\label{eq:Point_s}
		\sigma_p(T) \coloneqq
		\{\lambda\in\mathbb{C} \colon \ker(T-\lambda I)\ne\{0\}\},
	\end{equation}

A linear relation $A$ is called \emph{symmetric} in a Pontryagin space
 $(\sH,[ \cdot,\cdot]_{\sH})$ if $A \subseteq A^{[*]}$.
A point $\lambda \in \mathbb{C}$ is called a \emph{point of regular type} (and is written as $\lambda \in \widehat{\rho}(A)$) for a closed symmetric linear relation $A$, if $\lambda \notin \sigma_p(A)$ and the subspace $\ran(A-\lambda I)$ is closed in $H$.
For $\lambda \in \widehat{\rho}(A)$ let us set
$\sN_{\lambda}(A^{[*]}) \coloneqq \ker (A^{[*]} -\lambda I)$ and
\begin{equation}
	\widehat{\sN}_{\lambda}(A^{[*]})  \coloneqq
	\left\{
	\wh f_{\lambda} =
	\begin{bmatrix} f_{\lambda} \\ \lambda f_{\lambda} \end{bmatrix} \colon f_{\lambda} \in \sN_{\lambda}(A^{[*]})
	\right\}.
\end{equation}
As is known, see  \cite[Theorem 6.1]{IKL},  the numbers $\dim{\sN}_{\lambda}$ take a constant value $n_+(A)$ for all $\lambda\in\wh\rho(A)\cap\dC_+$, and $n_-(A)$ for all $\lambda\in\wh\rho(A)\cap\dC_-$, where
\begin{equation}\label{Cpm}
 \dC_\pm=\{\lambda\in\dC:\, \pm\text{Im }\lambda>0 \}.
\end{equation}
The numbers $n_\pm(A)$ are called the \textit{defect numbers} of $A$.

\subsection{$\cN_\kappa$-families}
Recall that a Hermitian kernel  ${\mathsf
K}_\omega(\lambda):\Omega\times\Omega\to\dC^{m\times m}$ is said to
have $\kappa$ {\bf negative squares}
if for every positive integer $n$ and every choice of $\omega_j\in\Omega$
and $u_j\in\dC^m$
$(j=1,\dots,n)$ the matrix
\[
\left[ u_k^*\mathsf{
K}_{\omega_j}(\omega_k)u_j\right]_{j,k=1}^n
\]
has at most $\kappa$ negative eigenvalues and for some choice of
$\omega_1,\ldots,\omega_n\in\Omega$ and $u_1,\ldots,u_n\in\dC^m$ exactly
$\kappa$ negative
eigenvalues. In this case we write
\[
\mbox{sq}_-{\mathsf K}=\kappa.
\]
For a mvf $ f(\lambda)$
we  shall  use the symbol
${\mathfrak h}_f$ for the domain of holomorphy of $f$ in $\dC$
% and $\dC_+$, respectively. Let us set
and set $f^\#(\lambda)=f(\overline{\lambda})^*$, for
$\overline{\lambda}\in {\mathfrak h}_f$.

\begin{definition} \label{def:Nk-class}
A $\ptp$ mvf $Q(\lambda)$ meromorphic on
     $\dC_+\cup\dC_-$  is said to belong to the class $\cN_\kappa^{\ptp}$ if:
\begin{enumerate}
  \item [(a)] the kernel
  \[
  {\mathsf
N}^Q_\omega(\lambda)
:=\left\{\begin{array}{ll}
\frac{\dsp Q(\lambda)-Q(\omega)^*}
{\dsp\lambda-\ov{\omega}}
&\quad \textrm{for }\lambda,\,\omega\in\gh_Q, \ \lambda\ne\ov{\omega}\\
Q'(\lambda)
&\quad \textrm{for }\lambda,\,\omega\in\gh_Q, \ \lambda=\ov{\omega}
\end{array}\right.
  \]
  has $\kappa$ negative squares on $\gh_Q$;
  \item [(b)] $Q^\#(\lambda)=Q(\lambda)$ for all $\lambda\in\gh_Q$.
\end{enumerate}
\end{definition}
The class $\cN_\kappa^{\ptp}$ was introduced by
M. G. Kre\u{\i}n and H. Langer in~\cite{KL71}. In the case $\kappa=0$ it coincides with the class $\cR^{\ptp}:=\cN_0^{\ptp}$ of mvf's meromorphic on
     $\dC_+\cup\dC_-$ which satisfy the assumption (b) in Definition~\ref{def:Nk-class} and have non-negative imaginary part $\text{Im }Q(\lambda)$ in $\dC_+$, see~\cite{KaKr74}.
      A mvf $M\in \cR^{\ptp}$ is said to belong to the class $\cR_u^{\ptp}$
     if  $\det\Im{M({i})}\ne0$.
\begin{definition} \label{def:Nk-family}
%Let $\Omega$ be the set $\dC\setminus\dR$, excluding a finite number $\le \kappa$ of points symmetric with respect to $\dR$.
    A family $\tau(\lambda)=\ran\begin{bmatrix}
    \varphi(\lambda)\\\psi(\lambda)
    \end{bmatrix}$ where $ \varphi$ and $\psi$ are $\ptp$ mvf's holomorphic on
     $\dC_+\cup\dC_-$  will be called an $\cN_\kappa^{\ptp}$-family if:
\begin{enumerate}
  \item [(i)] the kernel
  \[
  {\mathsf
N}_\omega^{\varphi\psi}(\lambda)
:=\left\{\begin{array}{ll}
\frac{\dsp\varphi(\omega)^*\psi(\lambda)-\psi(\omega)^*\varphi(\lambda)}
{\dsp\lambda-\ov{\omega}}
&\quad \textrm{for }\lambda,\,\omega\in\dC_+\cup\dC_-, \ \lambda\ne\ov{\omega}\\
\varphi^\#(\lambda)\psi(\lambda)-\psi(\lambda)^\#\varphi(\lambda)
&\quad \textrm{for }\lambda,\,\omega\in\dC_+\cup\dC_-, \ \lambda=\ov{\omega}
\end{array}\right.
  \]
  has $\kappa$ negative squares on $\dC_+\cup\dC_-$;
  \item [(ii)] $\varphi^\#(\lambda)\psi(\lambda)
      -\psi^\#(\lambda)\varphi(\lambda)$ for all $\lambda\in\dC_+\cup\dC_-$;
  \item[(iii)] $\ker \varphi(\lambda)\cap\ker \psi(\lambda)=\{0\}$
  for all $\lambda\in\dC_+\cup\dC_-$.
\end{enumerate}
The set of $\cN_\kappa^{\ptp}$-families is denoted by $\wt\cN_\kappa^{\ptp}$.
\end{definition}

Two $\cN_\kappa^{\ptp}$-families
$\ran\text{\rm col} \{\varphi_1(\lambda),\psi_1(\lambda)\}$ and
$\ran\text{\rm col} \{\varphi_2(\lambda),\psi_2(\lambda)\}$
coincide, if
$\varphi_2 (\lambda)=\varphi_1(\lambda)\chi(\lambda)$
and $\psi_2(\lambda)=\psi_1(\lambda)\chi(\lambda)$
for some mvf $\chi(\lambda)$,
which is holomorphic and invertible on $\dC_+\cup\dC_-$.

Every $\cN_\kappa^{\ptp}$-family $\tau=
\ran\text{\rm col} \{\varphi(\lambda), \psi(\lambda)\}$
%    \end{bmatrix}^T$
    can be treated as a family of linear relations in $\dC^p$.
In the case when $\det\varphi(\lambda)\not\equiv 0$
it has at most $\kappa$ zeros and
the $\cN_\kappa^{\ptp}$-family $\tau$ coincides on $\gh_{\varphi^{-1}}$ with
the graph of $\cN_\kappa^{\ptp}$-function
$Q(\lambda)=\psi(\lambda)\varphi(\lambda)^{-1}$.
Since every mvf from  $\cN_\kappa^{\ptp}$ admits such a representation
(see~\cite{BDHS11}),
the set $\cN_\kappa^{\ptp}$ can be identified with a subset of
$\wt \cN_\kappa^{\ptp}$.

\begin{definition} \label{def:Nk-pair}
    A pair $\begin{bmatrix}
    C(\lambda)& D(\lambda)
    \end{bmatrix}$ where $ C$ and $D$ are $\ptp$ mvf's holomorphic on
     $\dC_+\cup\dC_-$  will be called an $\cN_\kappa^{\ptp}$-pair if:
\begin{enumerate}
  \item [(i)] the kernel
  \[
  {\mathsf
N}_\omega^{CD}(\lambda)
:=\left\{\begin{array}{ll}
\frac{\dsp C(\lambda)D(\omega)^*-D(\lambda)C(\omega)^*}
{\dsp\lambda-\ov{\omega}}
&\quad \textrm{for }\lambda,\,\omega\in\dC_+\cup\dC_-, \ \lambda\ne\ov{\omega}\\
C'(\lambda)D(\omega)^*-D'(\lambda)C(\omega)^*
&\quad \textrm{for }\lambda,\,\omega\in\dC_+\cup\dC_-, \ \lambda=\ov{\omega}
\end{array}\right.
  \]
  has $\kappa$ negative squares on $\dC_+\cup\dC_-$;
  \item [(ii)] $C(\lambda)D^\#(\lambda)=D(\lambda)C^\#(\lambda)$ for all $\lambda\in\dC_+\cup\dC_-$;
  \item[(iii)] $\rank \begin{bmatrix}
    C(\lambda)& D(\lambda)
    \end{bmatrix}=p$
  for all $\lambda\in\dC_+\cup\dC_-$.
\end{enumerate}
In particular, the pair $\begin{bmatrix}
    C(\lambda)& D(\lambda)
    \end{bmatrix}$ is called a Nevanlinna pair if the assumptions (ii) and (iii) hold
    and the kernel ${\mathsf
N}_\omega^{CD}(\lambda)$ is non-negative on $\dC_+\cup\dC_-$.
\end{definition}
\begin{lemma}\label{lem:Nkpairs_fam}
There is a one to one correspondence between  $\cN_\kappa^{\ptp}$-families
$\tau(\lambda)=\ran\begin{bmatrix}
    \varphi(\lambda)\\\psi(\lambda)
    \end{bmatrix}$
    and  $\cN_\kappa^{\ptp}$-pairs $\begin{bmatrix}
    C(\lambda)& D(\lambda)
    \end{bmatrix}$
established by the formulas
\begin{equation}\label{eq:Nkpairs_fam}
  C(\lambda)=\psi^\#(\lambda),\quad D(\lambda)=\varphi^\#(\lambda), \quad
  \lambda\in\dC_+\cup\dC_-.
\end{equation}
The $\cN_\kappa^{\ptp}$-family $\tau$  admits the following representation
\begin{equation}\label{LR_50}
\begin{split}
\tau(\lambda)&=\ker{\begin{bmatrix} {C(z)} & -{D}(z)\end{bmatrix}}\\
&=\left\{\begin{bmatrix}
      u \\
      u'
    \end{bmatrix}\in \dC^d\times\dC^d: {C(z)}u- {D}(z)u'=0\right\},
\quad \lambda\in\dC_+\cup\dC_-,
\end{split}
\end{equation}
\end{lemma}
\begin{proof}
With this definition of the pair $\begin{bmatrix}
    C(\lambda)& D(\lambda)
    \end{bmatrix}$
the conditions (i)-(iii) in Definition~\ref{def:Nk-family} are equivalent
to  conditions (i)-(iii) in Definition~\ref{def:Nk-pair}.
  The formula~\eqref{LR_50} follows from the formula (ii) in Definition~\ref{def:Nk-pair}.
\end{proof}

In the case when $\kappa=0$ and $\varphi$ and $\psi$ are constant matrices,
 conditions (ii)-(iii) in Definition~\ref{def:Nk-pair} characterize all selfadjoint
 linear relations in $\dC^p$.
 \begin{lemma}[\cite{RB69}]\label{lem:SALR}
Every selfadjoint
 linear relation in $\dC^p$ admits the representation  $\tau=\ker{\begin{bmatrix} C &D \end{bmatrix}}$
 where $C,\, D\in {\dC}^{\ptp}$ and
\begin{equation}\label{eq:rank_and sym}
{C}{D}^*={D}{C}^* \quad\mbox{and}\quad
\rank \begin{bmatrix}
    C& D
    \end{bmatrix}=p.
\end{equation}
\end{lemma}
\subsection{Boundary triples and Weyl functions}
\label{subsec:pre:triples}
Let $A$ be a closed symmetric linear relation in a $\pi_\kappa$-space $(\sH,[ \cdot,\cdot]_{\sH})$ with equal defect numbers $n_\pm(A)=p<\infty$. We will need the following
\begin{lemma}\label{lem:11_A*}
If $\wt A$ is a selfadjoint extension of $A$ and $\lambda\in\rho(\wt A)$, then
  \begin{equation}\label{eq:11_A*wtA}
 A^{[*]}=\wt A\dotplus\wh\sN_{\lambda}.
  \end{equation}
\end{lemma}
In the case of a densely defined operator the notion of the boundary triple was introduced in \cite{Koc75,GG91} under the name ``space of boundary values''.
This notion  was adapted to the case of a
non-densely defined symmetric operator in~\cite{Der90}, \cite{M92}, \cite{DM95}
and to the case of linear relations in~\cite{D99}. In the
next definition we follow~\cite{DM95} and~\cite{D99}.

\begin{definition} \label{def:btriple}
Let $A$ be a closed symmetric linear relation in a $\pi_\kappa$-space $(\sH,[ \cdot,\cdot]_{\sH})$ with equal defect numbers $n_{\pm}(A)=p<\infty$.
% and let $\dC^p$ be an auxiliary Hilbert space such that $\dim \dC^p=n$.
%and let
%$\Gamma_j$ be linear mappings from $A^{[*]}$ to $\dC^p$, $j=1,2$.
	A tuple $\Pi = (\dC^p,\Gamma_0,\Gamma_1)$, where $\Gamma_0$ and $\Gamma_1$ are linear mappings from $A^{[*]}$ to $\dC^p$,
is called a \emph{boundary triple} for the linear relation $A^{[*]}$, if:
	\begin{enumerate}
	\item [(i)]
	for all
	$\wh f= \begin{bmatrix} f \\ f' \end{bmatrix}$,
	$\wh g = \begin{bmatrix} g \\ g' \end{bmatrix} \in A^{[*]}$
	the following Green's identity holds
	\begin{equation} \label{eq:1.9}
		[ f',g ]_{\sH} - [ f,g' ]_{\sH}
=(\Gamma_1\wh f, \Gamma_0\wh g)_{\dC^p}- (\Gamma_0\wh f, \Gamma_1\wh g)_{\dC^p};
	\end{equation}
	\item [(ii)]
	the mapping
	$\Gamma=\begin{bmatrix}\Gamma_0 \\ \Gamma_1\end{bmatrix} \colon	A^{[*]} \rightarrow \dC^2$
	is surjective.
	\end{enumerate}
\end{definition}

The following linear relations
\begin{equation} \label{e q:A0A1}
	A_0 \coloneqq \ker \Gamma_0, \qquad A_1 \coloneqq \ker \Gamma_1
\end{equation}
are selfadjoint extensions of the symmetric linear relation $A$.

The definition of the Weyl function for a
 symmet\-ric relation $A$ can be adapted as follows.
\begin{definition} \label{def:11W00}
Let $A$ be a closed symmetric linear relation with $n_+(A)=n_-(A)\le \infty$  and let ${\Pi}=({\dC^p},{\Gamma}_0,{\Gamma}_1)$ be a boundary triple for $A^{[*]}$. The Weyl function $M({\cdot})$ and the $\gamma$-field $\gamma({\cdot})$ of $A$ corresponding to
the boundary triple~$\Pi$ are defined by
   \begin{equation}\label{eq:11.M}
 M({z})\Gamma_0\wh f_{{z}}=\Gamma_1\wh f_{{z}},\quad \wh f_{z}\in \wh\sN_{{z}},
 \quad
 {z}\in\rho(A_0);
\end{equation}
and
\begin{equation}\label{eq:11_gamma1}
\wh\gamma({z})=(\Gamma_0\upharpoonright{\wh\sN_{{z}}})^{-1}, \quad
\gamma({z})=\pi_1\wh\gamma({z}), \quad
 {z}\in\rho(A_0),
\end{equation}
where $\pi_1$ is the projection onto the first component in $\sH\times\sH$.
\end{definition}
By Lemma~\ref{lem:11_A*}, the operator $\Gamma_0\upharpoonright{\wh\sN_{{z}}}:{\wh\sN_{{z}}}\to\dC^p$ is boundedly invertible and the operator-functions
%$\wh\gamma({z})$ takes values in $\cB(\dC^p,\wh\sN_{z})$ for ${z}\in\rho(A_0)$. Hence $\gamma({z})$ takes values in $\cB(\dC^p,\sN_{z})$ for ${z}\in\rho(A_0)$.
%Moreover, the vector functions
$\wh\gamma({z})$ and  $\gamma({z})$ admit the representations
 \begin{equation}\label{eq:11.wh_gamma2}
  \wh\gamma({z})=\wh\gamma({z})+
  ({z}-\zeta)\begin{bmatrix}
                (A_0 - z)^{-1}\gamma(\zeta) \\
                \gamma(\zeta)+\zeta(A_0 - z)^{-1}\gamma(\zeta)
              \end{bmatrix},\quad  {z},\zeta\in\rho(A_0),
\end{equation}
\begin{equation}\label{eq:11_gamma2}
  \gamma({z})=\gamma(\zeta)+({z}-\zeta)(A_0 - z)^{-1}\gamma(\zeta),\quad  {z},\zeta\in\rho(A_0),
\end{equation}
and so they are holomorphic on $\rho(A_0)$ with values in $\cB(\dC^p,\wh\sN_{z})$ and $\cB(\dC^p,\sN_{z})$, respectively, see~\cite{D99}.
Moreover, the following statement holds.
  \begin{theorem}\label{P:11.5}
Let $A$ be a  symmetric linear relation  with $n_{\pm}(A)=p<\infty$, let
$(\dC^p,\Gamma_0,\Gamma_1)$  be a boundary triple for
$A^{[*]}$, and let $M(\cdot)$ be the corresponding Weyl function. Then
\begin{enumerate}
    \item [\rm(i)] $M(\cdot)$ is a well-defined $\mathcal B({\mathcal H})$-valued function  holomorphic on $\rho(A_0)$.
        % with values in $\mathcal B({\mathcal H})$;
    \item[\rm(ii)] For all ${z},\zeta\in\rho(A_0)$ the following identity holds:
  \begin{equation}\label{Eq:11.8M}
{\mathsf N}_\omega^{M}(\lambda)
=\frac{M(\lambda)-M(\omega)^*}{\lambda-\ov\omega}
=\gamma(\omega)^{[*]}\gamma(\lambda).
%,\quad {z},\zeta\in\rho(A_0);
  \end{equation}
\item[\rm(iii)] If the subspace $\overline{\textup{span}}\left\{{\sN}_{\lambda}:\,\,
    \lambda\in\rho(A_0)\right\}$ is infinite-dimensional then for every $n\in\dN$ there exist $m\in\dN$
and a set of points $\lambda_j\in\rho(A_0)$ and vectors $u_j\in\dC^p$ $(j=1,\dots,m)$ such that the matrix
\begin{equation}\label{eq:Gram_M}
  \left[u_k^*{\mathsf N}_{\lambda_k}^{M}(\lambda_j)u_j\right]_{j,k=1}^m
\end{equation}
has at least $n$ positive eigenvalues.
  \end{enumerate}
\end{theorem}
\begin{proof}
  The proof of (i) and (ii) is based on the formula~\eqref{eq:11_A*wtA} and the identity~\eqref{eq:1.9}.

(iii) For $n\in \dN$ let us   choose  points $\lambda_j\in\rho(A_0)$ and vectors $f_j\in\sN_{\lambda_j}$ such that the linear space
\[
{\textup{span}}\left\{\sum_{j=1}^m c_jf_j:\, c_j\in\dC \right\}
\]
has an $n$--dimensional positive subspace in the $\pi_\kappa$-space $\sH$. Therefore, the Gram matrix
$ \left[[f_j,f_k]_\sH\right]_{j,k=1}^m$
of
the vectors $f_j$, $j=1,\dots,m$, has at least $n$ positive eigenvalues.
Representing vectors $f_j$ as $f_j=\gamma(\lambda_j)u_j$ with $u_j\in\dC^p$
and using~\eqref{Eq:11.8M}
one proves (iii).
\end{proof}
It follows from~\eqref{Eq:11.8M} that the Weyl function $M(\cdot)$ is a $Q$-function of the linear relation $A$ in the sense of~\cite{KL71}. Recall that a symmetric linear operator $A$ is called simple if
\begin{equation}\label{eq:simple_A}
      \sH=\overline{\textup{span}}\left\{{\sN}_{\lambda}:\,\,
    \lambda\in\wh\rho(A)\right\}.
\end{equation}

If $A$ is a simple symmetric operator in a $\pi_\kappa$-space
$(\sH,[ \cdot,\cdot]_{\sH})$, then its Weyl function $M({\cdot})$ defined by
\eqref{eq:11.M} belongs to the generalized Nevanlinna class $\cN_\kappa^{\ptp}$
 and, in particular, if $\kappa=0$,
then $M$ belongs to the Herglotz-Nevanlinna class $\cR^{\ptp}:=\cN_0^{\ptp}$. Moreover,
  %$M(\cdot)$ is a uniformly strict $\cR[{\mathcal{H}}]$-function, i.e.,
 $M\in \cR_u^{\ptp}$, i.e., $M\in \cR^{\ptp}$ and  $\det\Im{M({i})}\ne0$.
  %$($see  Definition~{\rm\ref{def:6SubclRcH}}$)$.

\subsection{Generalized resolvents of a symmetric linear relation $A$}
\begin{definition}\label{def:genres}
Let $\wt A$ be a selfadjoint extension of the symmetric linear relation $A$ in a possibly larger
Pontryagin space $(\wt\sH,[ \cdot,\cdot]_{\wt\sH})$ with negative
index $\wt\kappa$, $\wt \sH(\supseteq \sH)$. The extension $\wt A$ of
$A$ is called \textit{minimal}, if $\wt{{\sH}}=\sH_{\wt A}$
\begin{equation}\label{eq:MinGres}
    \sH_{\wt A}:=\overline{\textup{span}}\left\{\sH+(\wt{A}-\omega I_{\wt\sH})^{-1}\sH:\,\,
    \omega\in\rho(\wt A)\right\}.
\end{equation}
Let $P_\sH$ be the orthogonal projection onto $\sH$ in $\wt\sH$.
The operator-valued function
\begin{equation}\label{eq:genres}
  \omega \mapsto {\mathbf R}_\omega:=P_\sH(\wt A-\omega I_{\wt\sH})^{-1}|_\sH,\quad \omega\in\rho(\wt A)
\end{equation}
is called the {\bf generalized resolvent} of  $A$.
The representation~\eqref{eq:genres} of the generalized resolvent
${\mathbf R}_\omega$
%of the symmetric linear relation $A$
is called \textit{minimal} if \eqref{eq:MinGres} holds,
and the selfadjoint relation $\wt A$ is called the representing relation of the generalized resolvent $ {\mathbf R}_\omega$.
A generalized resolvent $ {\mathbf R}_\omega$ is said to have
 index $\wt \kappa$, if
$\mbox{ind}_-\wt\sH=\wt\kappa$ for a minimal representation \eqref{eq:genres}.
\end{definition}
%%%%%%%%%%%%%%%%%%%%%%%%%%%%%%%%%%%%%%%%%%%5

Every generalized resolvent $ {\mathbf R}_\omega$  of $S$ admits a minimal representation \eqref{eq:genres}; it   is unique up to a unitary equivalence, see \cite{KL71}, \cite{LaTe82}, \cite{D99}.
Moreover, the index $\wt\kappa=\mbox{ind}_-\wt\sH$
of such a minimal  representation does not depend on the choice of the representing extension $\wt S$ in~\eqref{eq:genres}.
%%%%%%%%%%%%%%%%%%%%%%%%%%%%%%%%%%%%%%%%%%%%%%%%%%%%%%%%%%%%

%Every generalized resolvent $ {\mathbf R}_\omega$  of $A$ admits a $\sH-$minimal representation~\eqref{eq:genres}, which is unique up to a unitary equivalence~\eqref{eq:genres}, see \cite{KL71}, \cite{LaTe82}.
       \begin{theorem}[\cite{KL71,D99}]\label{krein}
 Let $A$ be a closed symmetric linear relation in a $\pi_\kappa$-space  with defect numbers $n_{\pm}(A)=p<\infty$,
let $\Pi=({\dC^p},\Gamma_0,\Gamma_1)$ be a boundary triple for
$A^*$,  let $M(\cdot)$ and $\gamma(\cdot)$ be the
corresponding Weyl function and the $\gamma$-field,  let $A_0=\ker\Gamma_0$, $R_{\lambda}^0=(A_0-{{\lambda}}I_{\sH})^{-1}$
and let $\wt \kappa\in\dN_0$, $\wt \kappa\ge \kappa$. Then
\begin{enumerate}
\item[\rm(i)]
The formula %{\CB for generalized resolvents}
\begin{equation} \label{gres0}
{{\mathbf R}_{{\lambda}}}=R_{\lambda}^0-\gamma({\lambda})(M({\lambda})
     +\tau({\lambda}))^{-1}\gamma(\ov{{\lambda}})^{[*]},\quad
     {\lambda}\in\dC_+\cup\dC_-
\end{equation}
establishes a bijective correspondence:   ${\mathbf R}_{\lambda} \longleftrightarrow \tau$   between the class
of all generalized resolvents   ${\mathbf R}_{\lambda}$  of index $\wt\kappa$ and  the  class
of all $\cN_{\wt\kappa-\kappa}^{\ptp}$-families $\tau$ %\in \wt\cN_{\wt\kappa-\kappa}^{\ptp}$.
\item[\rm(ii)]
%If, in addition,   $\mul A_0=\{0\}$, then ${{\mathbf R}_z}\in \Omega_A$ if and only if $\tau\in \wt \cR(\dC^p)$ and
%\begin{equation}\label{eq:adm_Tau}
%%{\mathbf R}_{\lambda}\in \Omega_A \Longleftrightarrow
%\mbox{\rm s-}\lim_{y\to\infty}\frac{(\tau(iy)+M(iy))^{-1}}{y}=O_{\dC^p},
%\end{equation}
%
%\item[\rm(iii)]
${{\mathbf R}_{\lambda}}=(A_{-\tau({\lambda})}-{\lambda}I_{\sH})^{-1}$, that is, for every $h\in\sH$, $f={\mathbf R}_{\lambda}h$ is the solution of the boundary value problem with the  eigenvalue dependent boundary condition
\begin{equation}\label{eq:11.Shtr}
\begin{bmatrix}
  f & h
\end{bmatrix}^\top\in A^{[*]}-\lambda I_\sH,\quad
\begin{bmatrix}
 \Gamma_0\wh f &\, \Gamma_1\wh f\,
\end{bmatrix}^\top\in -\tau({\lambda}),
\end{equation}
where $\wh f=\begin{bmatrix}
  f & h +\lambda f
\end{bmatrix}^\top\in A^{[*]}$, and $\lambda\in\dC_+\cup\dC_-$.
\end{enumerate}
      \end{theorem}

      \begin{corollary}[\cite{KL71,D99}]\label{cor:krein}
In the assumptions of Theorem~{\rm\ref{krein}}
the formula
\begin{equation} \label{gres0CD}
{{\mathbf R}_{{\lambda}}}=R_{\lambda}^0-
\gamma({\lambda})(C({\lambda})+D({\lambda})M({\lambda}))^{-1}D({\lambda})\gamma(\ov{{\lambda}})^{[*]},
\quad
     {\lambda}\in\rho(A_0)\cap\rho(\wt A).
\end{equation}
establishes a bijective correspondence  between the class
of all generalized resolvents   ${\mathbf R}_{\lambda}$  of index $\wt\kappa$ and  the  class
of all $\cN_{\wt\kappa-\kappa}^{\ptp}$-pairs $\begin{bmatrix} C(\lambda) & D(\lambda)\end{bmatrix}$.
In particular, the formula
\eqref{gres0} gives a description of all canonical resolvents of $A$ when
$C$, $D$ range over the set of $\ptp$-matrices such that
\eqref{eq:rank_and sym} holds.

For every $h\in\sH$, $f={\mathbf R}_{\lambda}h$ is the solution of the boundary value problem
%with the  eigenvalue dependent boundary condition
\begin{equation}\label{eq:11.ShtrCD}
\begin{bmatrix}
  f & h
\end{bmatrix}^\top\in A^{[*]}-\lambda I_\sH,\quad
C(\lambda) \Gamma_0\wh f-D(\lambda) \Gamma_1\wh f=0,
\end{equation}
where $\wh f=\begin{bmatrix}
  f & h +\lambda f
\end{bmatrix}^\top\in A^{[*]}$, and ${\lambda}\in\rho(A_0)\cap\rho(\wt A)$.
      \end{corollary}

\section{Rigged Pontryagin spaces and generalized resolvents}
In this section we present the construction of the rigging $\sH_+\subset\sH\subset\sH_-$ of a Pontryagin space associated with a symmetric linear relation $A$ and look at the properties of
extended generalized resolvents of $A$ and extended boundary triples.
\subsection{Rigged Pontryagin spaces. }
\begin{lemma}\label{lem:RiggedPontrSp}
Let $A$ is a closed symmetric linear relation with equal
defect  numbers $n_{\pm}(A)=p<\infty$ in a $\pi_\kappa$-space
$(\sH,[ \cdot,\cdot]_{\sH})$. % such that $\mul A$ is a non-degenerate subspace of $\sH$.
Then there exists a pair $\sH_+$ and $\sH_-$ of Hilbert spaces with the duality $\langle\mathfrak{f},  h\rangle_{-,+}$ for $\mathfrak{f}\in\sH_-$ and $h\in\sH_+$ such that
\begin{enumerate}
  \item [(i)] $\sH_+=\dom A^{[*]}(\subset \sH)$;
  \item [(ii)] $\sH\subset \sH_-$ and for all $f\in \sH$ and $h\in\sH_+$ we have
  $\langle{f},  h\rangle_{-,+}=[f,h]_\sH$.
\end{enumerate}
Moreover, the norm in $\sH_+$ can be defined by
\begin{equation}\label{eq:A+Norm}
  \| f\|^2_{\sH_+}=\|f\|_\sH^2+\|P_{\sH_0}Jf'\|_\sH^2\quad
\mbox{for all $\begin{bmatrix}
                        f\\
                        f'
\end{bmatrix}\in A^{[*]}$}
\end{equation}
where $J$ is a fundamental symmetry in $(\sH,[ \cdot,\cdot]_{\sH})$ and
$P_{\sH_0}$ is the orthogonal projection  onto $\sH_{0}:=\overline{\dom A}$
in the Hilbert space $(\sH,[ J\cdot,\cdot]_{\sH})$.
\end{lemma}
\begin{proof}
Let us set $S:=JA$ for the symmetric linear relation in the Hilbert space $(\sH,(\cdot,\cdot)_{\sH})$  with the inner product
$(\cdot,\cdot)_{\sH})=(\sH,[J \cdot,\cdot]_{\sH})$.
The proof is splited into 2 steps.

%where $\sH$ is a Pontryagin space and $\sH_+$, $\sH_-$ are Hilbert spaces,
%is called a   {\bf rigged Pontryagin space} $\sH$, cf.,  \cite{Ber65} in the case of a Hilbert space $\sH$.

{\bf Step 1: Construction of a rigging for the  linear relation  $S$.}
If  $\mul S\ne\{0\}$, then the  linear relation  $S$ admits the representation
\begin{equation}\label{LR_22}
S=\gr(S_{\rm op})\widehat\oplus\,S_{\rm mul}, %\quad\text{where}\quad
%S_{\rm mul}=\begin{bmatrix}0\\ \mul S\end{bmatrix}.
\end{equation}
where $S_{\rm op}$ is  a single-valued symmetric  operator in the Hilbert space $\sH_{\rm op}:=\sH\ominus\mul S$,
while $S_{\rm mul}:=\begin{bmatrix}0\\ \mul S\end{bmatrix}$ is a purely multi-valued linear relation in the subspace $\sH_{\rm mul}:=\mul S$,  called {\it the multi-valued part} of~$S$.

Let us consider $S_{\rm op}$ as an operator from the Hilbert space $\sH_{0}:=\overline{\dom S}$ to the Hilbert space  $\sH_{\rm op}$ and denote it by $S_0$ and let $S_0^*\in\cC(\sH,\sH_{0})$ be the adjoint to the operator  $S_0\in\cC(\sH_{0},\sH)$. Then $S_0^*$ is a single-valued operator.
Denote by $\sH_{0,+}$ the linear space $\sH_{0,+}:=\dom S_0^*$ endowed with the inner product
 \begin{equation}\label{E:3.1a}
   ( f,g)_{0,+}:=(f,g)_{\sH}+(S_0^* f,S_0^*g)_\sH,\quad
   f,g\in\sH_{0,+}=\dom S_0^*.
\end{equation}
Since the operator $S_0^*$ is closed, $\sH_{0,+}$ is a Hilbert space with the graph norm
 \begin{equation}\label{eq:S0+norm}
\|\wh f\|^2_{0,+}=\|f\|_\sH^2+\|S_0^*f\|_\sH^2\quad
f\in\sH_{0,+}=\dom S_0^*
\end{equation}
such that $\sH_{0,+}\subset\sH_{\rm op}$. If $\sH_{0}\subsetneq\sH$ then $S_{\rm op}^*$
is a linear relation
\[
S_{\rm op}^*=\mbox{gr }S_0^*\widehat\oplus\,\begin{bmatrix}
                                     0 \\
                                     \sH_1
                                   \end{bmatrix},\quad\text{where}\quad
\sH_1=\sH_{\rm op}\ominus \sH_0.
\]

By \cite[Section~1.1.1]{Ber65},
there exist a dual Hilbert space $\sH_{-}$ of  bounded conjugate linear functionals on $\dC^p_{0,+}$
and an isometric operator $V_{0}$ from $\sH_{0,-}$ onto $\sH_{0,+}$ such that $\sH\subset\sH_{-}$ and
\begin{equation}\label{eq:f(h)0}
[f,h]_\sH=(V_0f,h)_{\sH_{+}}\quad\text{for all }\quad f\in\sH_{\rm op},\quad h\in\sH_{0,+}.
\end{equation}
The embedding ${\imath}:\sH_{\rm op}\hookrightarrow \sH_{0,-}$ is realized by the identification of
any vector $f\in\sH_{\rm op}$ with the functional
\begin{equation}\label{eq:f(h)}
{\imath}f:h\in\sH_{0,+}\mapsto {\imath}f(h)=(f,h)_\sH.
\end{equation}
and the space $\sH_{0,-}$ can  be realized as the completion of $\sH_{\rm op}$ with respect to the ``negative norm''
\[
\|{\imath}f\|_{\sH_{-}}=
\sup_{h\in\sH_{0,+}\setminus\{0\}}\frac{|[f,h]_\sH|}{\|h\|_{\sH_{0,+}}}.
\]
For $f\in\sH$ we will identify ${\imath}f$ with $f$. This gives the inclusion $\sH_{\rm op}\subset\sH_{0,-}$ and we will use the notation
\begin{equation}\label{eq:B-+DualityM}
  \langle\mathfrak{f},  h\rangle^{(0)}_{-,+}:=\mathfrak{f}(h)\quad \textrm{for \ $\mathfrak{f}\in\sH_{0,-}$ \ and  \quad $h\in\sH_{0,+}$},
\end{equation}
for the duality between $\sH_{0,-}$ and $\sH_{0,+}$.  In view of \eqref{eq:f(h)}
  the expression $\langle\mathfrak{f},  h\rangle_{-,+}$ can be viewed as an extension  of the inner product in the Hilbert space $\sH_{\rm op}$.

The triple
\[
\sH_{0,+}\subset\sH_{\rm op}\subset\sH_{0,-}
\]
is called a   {\bf rigged Hilbert space} $\sH_{\rm op}$, cf.,  \cite{Ber65}.
Setting
\begin{equation}\label{eq:H0pm_decom}
  \sH_+=\sH_{0,+}\oplus\mul S,\quad \sH_-=\sH_{0,-}\oplus\mul S,\quad
\wh V:=V\oplus I_{\mul S}:\,\wh\sH_-\to\wh\sH_+
\end{equation}
we obtain a rigged Hilbert space $\sH$
\begin{equation}\label{eq:H_Rigging}
  \sH_+\subset\sH\subset\sH_-
\end{equation}
with the duality  between $\sH_-$ and $\sH_+$ given by
\[
\langle\mathfrak{f},  h\rangle^{(S)}_{-,+}:=\langle\mathfrak{f}_0,  h_0\rangle^{(0)}_{-,+}+
(f_1,h_1)_\sH
\]
for
\[
\mathfrak{f}=\mathfrak{f}_0+f_1,\quad
{h}=h_0+h_1,\quad
 \ff_0\in\sH_{0,-},  \quad h_0\in\sH_{0,+},\quad f_1,h_1\in\mul S.
 \]
 Then it follows from \eqref{eq:B-+DualityM} and \eqref{eq:f(h)} that
 \begin{equation}\label{eq:S_duality}
\langle\mathfrak{f},  h\rangle^{(S)}_{-,+}=(f,h)_\sH\quad
 \text{ for all $f\in \sH$ and $h\in\sH_+$}.
\end{equation}
Since $S^*=S_{\rm op}^*\widehat\oplus\,\begin{bmatrix}
                                     0 \\
                                     \mul S
                                   \end{bmatrix}$
we obtain from~\eqref{eq:S0+norm}
%for all $\begin{bmatrix}
%                        f\\
%                        f'
%\end{bmatrix}\in S^*$
\begin{equation}\label{eq:S+Norm}
  \| f\|^2_{\sH_+}=\|f\|_\sH^2+\|P_{\sH_0}f'\|_\sH^2\quad
 \mbox{for all $\begin{bmatrix}
                        f\\
                        f'
\end{bmatrix}\in S^*$}.
\end{equation}
%%%%%%%%%%%%%%%%%%%%%%%%%%%%%%%%%%%%%%%%%%555

{\bf Step 2: Construction of a rigging for the  linear relation  $A$.}
Let $\sH_+\subset\sH\subset\sH_-$ be the rigging constructed for  the  linear relation  $S$.
Then (i) holds since $S^*=J A^{[*]}$ and hence $\sH_+=\dom S^*=\dom A^{[*]}$.
The formula~\eqref{eq:A+Norm} follows from~\eqref{eq:S+Norm}.

Let us define  the duality  between $\sH_-$ and $\sH_+$  by
\[
\langle\mathfrak{f},  h\rangle_{-,+}:=\langle\mathfrak{f}, J h\rangle^{(S)}_{-,+}\quad
 \text{ for  $\mathfrak{f}\in \sH_-$ and $h\in\sH_+$}.
\]
Then, by \eqref{eq:S_duality}, we get for all $f\in \sH$ and $h\in\sH_+$
\[
\langle{f},  h\rangle_{-,+}=\langle\mathfrak{f}, J h\rangle^{(S)}_{-,+}
=(f,Jh)_\sH=[f,h]_\sH.
\]
This proves (ii).
\end{proof}

The triple $\sH_+\subset\sH\subset\sH_-$ constructed in Lemma~\ref{lem:RiggedPontrSp} is called a   {\it rigged Pontryagin space}.

Let $\dC^p$ be an auxiliary Hilbert space.
For a linear operator $T\in\cB(\sH_-,\dC^p)$ we denote by $T^{\langle*\rangle}\in\cB(\dC^p,\sH_+)$  its adjoint with respect to the duality $\langle \cdot, \cdot\rangle_{-,+}$, i.e.,
\begin{equation}\label{eq:<*>}
  \langle \ff, T^{\langle*\rangle}\xi\rangle_{-,+}=(T\ff,\xi)_{\dC^p}=\xi^*T\ff\quad
\xi\in\dC^p,\quad \ff\in\sH_-.
\end{equation}
For an operator $T\in\cB(\sH,\dC^p)$ the operator $T^{\langle*\rangle}$ defined by the equality~\eqref{eq:<*>} coincides with $T^{[*]}\in\cB(\dC^p,\sH)$.
%$ T^{\langle*\rangle}=T^{[*]}\in\cB(\dC^p,\sH_+)$.

Similarly, the adjoint to  $G\in\cB(\dC^p,\sH_-)$ with respect to the duality $\langle \cdot, \cdot\rangle_{-,+}$ is defined as the operator $G^{\langle*\rangle}\in\cB(\sH_{+},\dC^p)$ such that
\begin{equation}\label{eq:<*>2}
\langle G\xi, h \rangle_{-,+}=(\xi,G^{\langle*\rangle} h)_{\dC^p}\quad
\xi\in\dC^p,\quad  h\in\sH_+.
\end{equation}
In particular, for the operator $\gamma(\ov{\lambda})\in\cB(\dC^p,\sH_+)$
its adjoint $\gamma(\ov{\lambda})^{\langle*\rangle}\in\cB(\sH_-,\dC^p)$ is defined by
\begin{equation}\label{eq:gamma<*>}
\langle \ff, \gamma(\ov{\lambda}) \xi \rangle_{-,+}
=(\gamma(\ov{\lambda})^{\langle*\rangle}\ff,\xi)_{\dC^p}\quad
\ff\in\sH_-,\quad  \xi\in\dC^p.
\end{equation}
In what follows we also use the notation
\begin{equation}\label{eq+-duality}
  \langle h, \ff  \rangle_{-,+}:=\langle \ff, h \rangle_{-,+}^*,\quad
  \text{for}\quad h\in\sH_+,\quad \ff\in\sH_-.
\end{equation}
\subsection{Extended generalized resolvents}
\begin{lemma}\label{lem:3.2B}
Let $A$ be a  closed symmetric linear relation with equal defect numbers,
%let  $\wt A$ be a selfadjoint extension of $A$ in  an exit space $\wt\sH\supseteq \sH$,
and let ${\mathbf R}_\lambda$ be a generalized resolvent of  $A$ with a minimal representing relation $\wt A$, $\lambda\in\rho(\wt A)$. Then
\begin{enumerate}
  \item [(i)] $\ran {\mathbf R}_\lambda\subset \sH_+$ and
  $\begin{bmatrix}
     {\mathbf R}_\lambda f \\
     (I_\sH+\lambda{\mathbf R}_\lambda)f
   \end{bmatrix}\in A^{[*]}$ for all $f\in\sH$.
  \item [(ii)] ${\mathbf R}_\lambda\in\cB(\sH,\sH_+)$.
  \item [(iii)] The operator ${\mathbf R}_\lambda$ admits a continuation to an operator
  $\wt{\mathbf R}_\lambda\in\cB(\sH_-,\sH)$ called extended generalized resolvent.
   \item [(iv)] For every selfadjoint extension $A_0$ of $A$  the extended resolvent $\wt R_\lambda^0$ satisfies the identity
   \begin{equation}\label{eq:HilbertId}
     \wt R_\lambda^0-\wt R_\mu^0=(\lambda-\mu) R_\lambda^0\wt R_\mu^0,\quad
     \lambda,\mu\in\rho (A_0),
   \end{equation}
   and hence
  $\wt R_\lambda^0-\wt R_\mu^0\in\cB(\sH_-,\sH_+)$.
\end{enumerate}
\end{lemma}
\begin{proof}
  (i) For all $f\in\sH$ and $\begin{bmatrix}
     h\\
     h'
   \end{bmatrix}\in A$ we obtain
 \[
%\begin{split}
   [(I_\sH+\lambda{\mathbf R}_\lambda)f,h]_\sH - [{\mathbf R}_\lambda f,h])_\sH\\
   =
   [f,h]_\sH -[f, {\mathbf R}_{\overline\lambda}(h'-\overline\lambda h)]_\sH =[f,h]_\sH-[f,h]_\sH=0  %\end{split}
\]
 Here we used the inclusion ${\mathbf R}_{\overline\lambda}(h'-\overline\lambda h)=h $.
This proves (i).

(ii) Let $f_n\to 0$ in $\sH$. Then ${\mathbf R}_\lambda f_n\to 0$ and, by
\eqref{eq:A+Norm},
 \[
  \| {\mathbf R}_\lambda f_n\|^2_{\sH_+}=\|{\mathbf R}_\lambda f_n\|_\sH^2+\|P_{\sH_0}J(f_n+\lambda {\mathbf R}_\lambda f_n)\|_\sH^2\to 0.
\]

(iii) Since  ${\mathbf R}_{\overline\lambda}^{\langle*\rangle}\in\cB(\sH_-,\sH)$, (iii) follows from the inclusion ${\mathbf R}_{\lambda}\subset{\mathbf R}_{\overline\lambda}^{\langle*\rangle}$.

(iv) follows from the Hilbert identity
\[
R_{\overline\lambda}^0-R_{\overline\mu}^0
=({\overline\lambda}-{\overline\mu}) R_{\overline\lambda}^0 R_{\overline\mu}^0
\]
by taking the adjoint with respect to the duality $\langle \cdot, \cdot\rangle_{-,+}$
and using the equalities $(R_{\overline\lambda}^0)^{\langle*\rangle}=\wt R_{\lambda}^0$,
$(R_{\overline\mu}^0)^{\langle*\rangle}=\wt R_{\mu}^0$.
\end{proof}
The operator $\bR$ defined by
\begin{equation}\label{eq:cR}
\bR=\frac12(R_{i}^0+R_{-i}^0).
\end{equation}
has the properties
\begin{enumerate}
  \item $\bR=\bR^*$, $\bR\in\cB(\sH,\sH_+)$ and $\wt\bR:=\bR^{\langle*\rangle}\in\cB(\sH_-,\sH)$.
  \item For every  extended generalized resolvent $\wt{\mathbf R}_\lambda\in\cB(\sH_-,\sH)$ we have
  $\wt{\mathbf R}_\lambda-\wt\bR\in\cB(\sH_-,\sH_+)$.
\end{enumerate}
the first claim follows from Lemma~\ref{lem:3.2B}(ii)--(iii) and the second claim follows from the equality
\[
\wt{\mathbf R}_\lambda-\wt\bR=(\wt R_{\lambda}^0-\wt\bR)-\gamma(\lambda)(M(\lambda)
     +\tau(\lambda))^{-1}\gamma(\ov{\lambda})^{\langle*\rangle},
\]
and the relations $\gamma(\lambda)\in\cB(\dC^p,\sH_+)$ and $\gamma(\ov{\lambda})^{\langle*\rangle}\in\cB(\sH_-,\dC^p)$, see~\eqref{eq:gamma<*>}.

The operator $\bR$  is called the {\it regularizing operator} and
\begin{equation}\label{eq:RegExtRes}
  \wh{\mathbf R}_\lambda:=\wt{\mathbf R}_\lambda-\wt\bR
\end{equation}
is called
the {\it regularized extended generalized resolvent}.
\begin{lemma}\label{lem:9.5B}
Let $A$ be a  closed symmetric linear relation with equal defect numbers,
%let  $\wt A$ be a selfadjoint extension of $A$ in $\sH$,
 let ${\mathbf R}_\lambda$ be an extended generalized resolvent of  $A$  with a minimal representing relation $\wt A$ and let
%${\mathbf{A}}$ be the set of pairs
%$\wh f=\begin{bmatrix}
%f \\
%\ff'
%\end{bmatrix}$, $f\in\sH$, $\ff'\in\sH_-$ such that
\begin{equation}\label{eq:bf_S}
\bA=\left\{\wh f=\begin{bmatrix}
f \\
\ff'
\end{bmatrix}\in \begin{bmatrix}
\sH \\
\sH_-
\end{bmatrix}  :\, \langle \ff', h \rangle_{-,+}=[f,h']_\sH\quad
\text{for}\quad  \wh h=\begin{bmatrix}
                                               h \\
                                               h'
                                             \end{bmatrix}\in A^{[*]}.
                                             \right\}
%\langle \ff', h \rangle_{-,+}=[f,h']_\sH\quad
%\text{for all}\quad  \wh h=\begin{bmatrix}
%                                               h \\
%                                               h'
%                                             \end{bmatrix}\in A^{[*]}.
\end{equation}
Then:
\begin{enumerate}
\item[\rm(i)] $\bA$ is a closed linear relation in $\wt\cC(\sH,\sH_-)$
with $\dom \bA=\sH_{0}:=\overline{\dom S}$.
  \medskip

  \item [(ii)] $\bA$ is an extension of $A$ and $A=\bA\cap\sH^2$.
  % and hence
  % $\ran({\mathbf{A}}-\lambda\imath)={\imath}\ran(A-\lambda I)$.
  \medskip

  \item [(iii)]  $\widetilde{\mathbf R}_{\lambda}(f' -{\lambda} f)=f$ for all
  $\begin{bmatrix}
  f \\
  f'
  \end{bmatrix}\in \bA$, $\lambda\in\rho(\wt A)$.
  \medskip

  \item [(iv)] $\ran({\mathbf{A}} -{\lambda} I_\sH)$ is a closed subspace of $\sH_-$ for all ${\lambda}\in\wh\rho(A)$.
  \medskip

  \item [(v)] $\ker({\mathbf{A}} -{\lambda} I_\sH)=\{0\}$ for $\lambda\in\wh\rho(A)$.
      %$\textup{rng }({\mathbf{A}} -{\lambda} \imath)\cap\imath \sH=\imath \textup{rng }({A} -{\lambda} I)$.
  \medskip

   \item [(vi)] The annihilator $\{\ran({\mathbf{A}} -{\lambda} I_\sH)\}^\perp$ of the set $\ran({\mathbf{A}} -{\lambda} I_\sH)$ in $\sH_+$ coincides with $\wh\sN_{\lambda}$.
\end{enumerate}
\end{lemma}
\begin{proof}
  (i) $\&$ (ii) are  immediate from \eqref{eq:bf_S}.
  \medskip

  (iii) For  $g\in\sH$ we get, by Lemma~\ref{lem:3.2B},
  $\begin{bmatrix}
     {\mathbf R}_{\overline{\lambda}} g \\
     (I_\sH+\ov\lambda{\mathbf R}_{\overline{\lambda}}g
   \end{bmatrix}\in A^{[*]}$
   and so   for  $\begin{bmatrix}
  f \\
  \ff'
  \end{bmatrix}\in \bA$ we obtain, by~\eqref{eq:bf_S},
  \[
  \begin{split}
     [\widetilde{\mathbf R}_{\lambda}(\ff' -{\lambda} f),g]_\sH
      &= \langle \ff' -{\lambda} f,{{\mathbf R}_{\overline{\lambda}}}g\rangle_{-,+} \\
      &=[f, g+\overline{\lambda}{\mathbf R}_{\overline{\lambda}}g])_\sH
      -\lambda[f, {\mathbf R}_{\overline{\lambda}}g]_\sH=
      [f,g]_\sH.
      \end{split}
  \]
  This proves (iii).
  \medskip

  (iv) Let
  $\begin{bmatrix}
  f_n \\
  \ff_n'
  \end{bmatrix}\in \bA$, and let
   $\ff_n' -{\lambda} f_n\to \fg$ in $\sH_-$ as $n\to\infty$. Then, by (iii) and Lemma~\ref{lem:3.2B}(iii),
   $f_n=\widetilde{\mathbf R}_{\lambda}\ff_n' -{\lambda} f_n \to f:=\widetilde{\mathbf R}_{\lambda}\fg$  in $\sH$. Hence
   $\ff_n' \to\ff':={\lambda} f+ \fg$ in $\sH_-$ and, by (ii),
   we get
   $\begin{bmatrix}
  f \\
  \ff'
  \end{bmatrix}\in \bA$ and therefore
 $\fg=\ff'-\lambda f\in \ran(\bA-\lambda I_\sH)$.
    \medskip

  (v) follows from (ii).
    \medskip

  (vi)  Since $\ran(A -{\lambda} I_\sH)\subset \ran({\mathbf{A}} -{\lambda} I_\sH)$ for $\lambda\in\wh\rho(A)$, the inclusion  $\{\ran({\mathbf{A}} -{\lambda} I_\sH)\}^\perp\subseteq\wh\sN_{\lambda}$ holds.
  Conversely if $h\in\wh\sN_{\ov\lambda}$, then
 $\wh h =\begin{bmatrix}
                                               h \\
                                               \ov{\lambda}h
                                             \end{bmatrix}\in A^{[*]}$
 and for all $\begin{bmatrix}
  f \\
  \ff'
  \end{bmatrix}\in \bA$ we get, by \eqref{eq:bf_S},
  %and Lemma~\ref{lem:RiggedPontrSp}(ii),
 \[
 0= \langle \ff'-\lambda f, h \rangle_{-,+}=[f,\ov{\lambda}h]_\sH-\lambda[f,h]_\sH=0
 %=[f,h'-\ov\lambda h]_\sH=0.
 \]
  Therefore $ h\in\{\ran({\mathbf{A}} -{\lambda} I_\sH)\}^\perp$.
%
%
%  $\begin{bmatrix}
%  f \\
%  \lambda f
%  \end{bmatrix}\in \bA$ for some $f\in\sH_{0}$. Then, by~\eqref{eq:bf_S},
%  \[
% \lambda[f,h]_\sH= \langle \lambda f, h \rangle_{-,+}=[f,h']_\sH\quad
%\text{for all}\quad  h=\begin{bmatrix}
%                                               h \\
%                                               h'
%                                             \end{bmatrix}\in A^{[*]}.
%  \]
%  and hence $\begin{bmatrix}
%  f \\
%  \lambda f
%  \end{bmatrix}\in (A^*)^*=A$. Since $\lambda\in\dC_+\cup\dC_-$, we get $f=0$.
\end{proof}

Let the linear relation $A^{\langle*\rangle}$ be given by
\begin{equation}\label{eq:bf_A*}
A^{\langle*\rangle}=\left\{\wh f=\begin{bmatrix}
f \\
\ff'
\end{bmatrix}\in \begin{bmatrix}
\sH \\
\sH_-
\end{bmatrix}  :\, \langle \ff', h \rangle_{-,+}=[f,h']_\sH\quad
\text{for}\quad  \wh h=\begin{bmatrix}
                                               h \\
                                               h'
                                             \end{bmatrix}\in A\right\}.
\end{equation}
Similarly, for a selfadjoint extension $A_0$ of $A$ we will set
\begin{equation}\label{eq:bf_A0}
\bA_0=\left\{\wh f=\begin{bmatrix}
f \\
\ff'
\end{bmatrix}\in \begin{bmatrix}
\sH \\
\sH_-
\end{bmatrix}  :\, \langle \ff', h \rangle_{-,+}=[f,h']_\sH\quad
\text{for}\quad  \wh h=\begin{bmatrix}
                                               h \\
                                               h'
                                             \end{bmatrix}\in A_0\right\}.
\end{equation}
Then, by~\eqref{eq:bf_S},~\eqref{eq:bf_A*} and~\eqref{eq:bf_A0},
\begin{equation}\label{eq:bf_A_A0}
\bA\subset \bA_0\subset A^{\langle*\rangle}.
\end{equation}
%%%%%%%%%%%%%%%%%%%%%%%%%%%%%%%%%%%%%%%%%%%%%%%%%%%%%%%%%%%%%%%%%%%%%%%%%%%%
\begin{lemma}\label{lem:bf_A0}
Let $A_0$  be a selfadjoint extension of $A$ in $\sH$, let $\bA_0$ be defined by~\eqref{eq:bf_A0},
and let $\widetilde{ R}^0_{\lambda}$ be the extended resolvent of $A_0$,  ${\lambda}\in\rho(A_0)$.
Then:
\begin{enumerate}
\item [(i)]  $\widetilde{ R}^0_{\lambda}(\ff' -{\lambda} f)=f$ for all
  $\begin{bmatrix}
  f \\
  \ff'
  \end{bmatrix}\in \bA_0$, $\omega\in\rho(\wt A)$.
  \medskip

\item [(ii)]
The linear relation $\bA_0$ admits the representation
\begin{equation}\label{eq:bf_A0R}
\bA_0=\left\{\begin{bmatrix}
\wt R_\lambda^0\ff \\
\ff+\lambda\wt R_\lambda^0\ff
\end{bmatrix}:\, \ff\in\sH_-\right\}.
\end{equation}
  \medskip

\item [(iii)]
$A^{\langle*\rangle}$, $\bA_0$ and $\wh\sN_{\lambda}$ are closed subspaces $\wh\sH:=\sH\times\sH_-$ and
\begin{equation}\label{eq:A+_decom0}
    A^{\langle*\rangle}=\bA_0\dotplus\wh\sN_{\lambda}.
\end{equation}
\end{enumerate}
\end{lemma}
\begin{proof}
(i) is proved similarly to  Lemma~\ref{lem:9.5B}(iii).
\medskip

(ii) Notice first that for every $\omega\in\rho(A_0)$
\begin{equation}\label{eq:A0R}
A_0=\left\{\begin{bmatrix}
 R_\omega^0h \\
h+\omega R_\omega^0h
\end{bmatrix}:\, h\in\sH\right\}.
\end{equation}
Then for all $\ff\in\sH_-$ and $h\in\sH$ we get, using the Hilbert identity
\begin{multline*}
\langle \ff+\lambda\wt R_\lambda^0\ff, R_\omega^0h \rangle_{-,+}
-[\wt R_\lambda^0\ff,h+\omega R_\omega^0h]_\sH\\
=\langle \ff, R_\omega^0h+\ov \lambda R_{\ov\lambda}^0R_\omega^0h \rangle_{-,+}-\langle \ff,R_{\ov\lambda}^0h+\omega R_{\ov\lambda}^0R_\omega^0h\rangle_{-,+}\\
=\langle \ff, R_\omega^0h-R_{\ov\lambda}^0h-(\omega-\ov \lambda) R_{\ov\lambda}^0R_\omega^0h \rangle_{-,+}=0
\end{multline*}
Therefore, $\wh g=\begin{bmatrix}
\wt R_\lambda^0\ff \\
\ff+\lambda\wt R_\lambda^0\ff
\end{bmatrix}\in\bA_0$ for all $\ff\in\sH_-$.

Conversely, let $\wh g=\begin{bmatrix}
g \\
\fg'
\end{bmatrix}\in\bA_0$. By (i), we get for every $\lambda\in\rho(A_0)$
\begin{equation}\label{eq:A0R2}
\wt R_\lambda^0(\fg'-\lambda g)=g.
\end{equation}
Setting $\ff=\fg'-\lambda g$ we obtain $g=\wt R_\lambda^0\ff$ and hence
$
\fg'=\ff+\lambda g=\ff+\wt R_\lambda^0\ff.
$
% from~\eqref{eq:A0R2}
%\[
%\begin{split}
%  0&= \wt R_\omega^0(\ff'-\omega f)-f=
%  \wt R_\omega^0\ff'-\omega R_\omega^0R_\lambda^0 g-R_\lambda^0 g \\
%     & =
%  \wt R_\omega^0\ff'-\lambda R_\omega^0R_\lambda^0 g-R_\omega^0 g
%    = \wt R_\omega^0\left(\ff'-(g+\lambda R_\lambda^0 g)\right).
%\end{split}
%\]
Therefore, the vector $\wh g$ admits the representation $\wh g=\begin{bmatrix}
\wt R_\lambda^0\ff \\
\ff'+\lambda\wt R_\lambda^0\ff'
\end{bmatrix}$ which proves~\eqref{eq:bf_A0R}.

(iii) For
$\wh f=\begin{bmatrix}
f \\
\ff'
\end{bmatrix}\in A^{\langle *\rangle}$ and $\lambda\in\rho(A_0)$ let us set
\[
g=\wh R_\lambda^0(\ff'-\lambda f),\quad
\fg'=(\ff'-\lambda f)+\lambda\wh R_\lambda^0(\ff'-\lambda f).
\]
Then, by Lemma~\ref{lem:bf_A0} (ii), $\wh g=\begin{bmatrix}
g \\
\fg'
\end{bmatrix}\in \bA_0$ and $ f- g\in\sN_\lambda$ since for all $\begin{bmatrix}
h \\
h'
\end{bmatrix}\in A$
\[
[f-\wt R_\lambda^0(\ff'-\lambda f),h'-\ov\lambda h]_\sH
=[f,h'-\ov\lambda h]_\sH-\langle \ff'-\lambda f, h \rangle_{-,+}=0.
\]
Moreover,
\[
\wt f-\wt g=\begin{bmatrix}
f -g\\
\ff'-\fg'
\end{bmatrix}=\begin{bmatrix}
f -\wt R_\lambda^0(\ff'-\lambda f)\\
\lambda f-\lambda \wt R_\lambda^0(\ff'-\lambda f)
\end{bmatrix}\in\wh\sN_\lambda.
\]
This proves the inclusion $A^{\langle*\rangle}=\bA_0\dotplus\wh\sN_{\lambda}$.
The converse inclusion follows from \eqref{eq:bf_A_A0}.
\end{proof}
%%%%%%%%%%%%%%%%%%%%%%%%%%%%%%%%%%%%%%%%%%%%%%%%%%%%%%%%%%%%%%%%%%%%%%%%%%%%

\subsection{Extended boundary triples}
\begin{lemma}\label{lem:bf_A0Pi}
Let $\Pi=(\dC^p,\Gamma_0,\Gamma_1)$ be a boundary triple for $A^{[*]}$,
let the $\gamma$-field $\gamma$ and  $\wh\gamma$ be given by~\eqref{eq:11_gamma1}
and let
$R_\lambda^0=(A_0-\lambda I)^{-1}$, $\lambda\in\rho(A_0).$
Then the mapping $\Gamma$ admits a continuation to a bounded mapping
$\wh\Gamma: A^{\langle *\rangle}\to\dC^{2p}$ given by
\begin{equation}\label{eq:wh_Gamma}
 \wh\Gamma_0 \begin{bmatrix}
\wt R_\lambda^0\ff  \\
\ff +\lambda\wt R_\lambda^0\ff
\end{bmatrix}= 0, \quad
 \wh\Gamma_1 \begin{bmatrix}
\wt R_\lambda^0\ff  \\
\ff +\lambda\wt R_\lambda^0\ff
\end{bmatrix}= \gamma(\ov\lambda)^{\langle *\rangle}\ff , \quad \ff \in\sH_-,
\end{equation}
\[
\wh\Gamma\wh\gamma(\lambda)u=\Gamma\wh\gamma(\lambda)u,\quad u\in\dC^p.
\]
Moreover, $\wh\Gamma$ satisfies the relations
\begin{equation}\label{ker_wh_Gamma}
  \ker\wh\Gamma_0=\bA_0,\quad \ker\wh\Gamma_1\cap\ker\wh\Gamma_0 =\bA.
\end{equation}
\end{lemma}
\begin{proof}
  1) Since $A^{\langle*\rangle}$, $\bA_0$ and $\wh\sN_{\lambda}$ are closed subspaces of
   $\wh\sH:=\sH\times\sH_-$ and the operators $\Gamma_j\upharpoonright\wh\sN_\lambda$, $j=0,1$ are bounded from $\wh\sH$ to $\dC^p$, it is enough to prove that
    the operators $\Gamma_j\upharpoonright\bA_0$, $j=0,1$ are bounded from $\wh\sH$ to $\dC^p$. For $j=0$ this statement is clear, so let us prove that
     $\Gamma_1\upharpoonright\bA_0$ is bounded as an operator from $\wh\sH$ to $\dC^p$.

Assume that
\[
\wt R_\lambda^0\ff_n\stackrel{\sH}{\longrightarrow}g,\quad
\ff_n +\lambda\wt R_\lambda^0\ff_n\stackrel{\sH_-}{\longrightarrow}\fg'\quad\text{as}\quad n\to\infty
\]
with $\ff_n, \fg'\in\sH_-$, $g\in\sH$. Then
$\ff_n\stackrel{\sH_-}{\longrightarrow}\ff:=\fg'-\lambda g$ and
$\gamma(\ov\lambda)^{\langle *\rangle}\ff_n\stackrel{\dC^p}{\longrightarrow}
\gamma(\ov\lambda)^{\langle *\rangle}\ff$.  This proves that
$\Gamma_1\upharpoonright\bA_0$ is bounded as an operator from $\wh\sH$ to $\dC^p$.

2) The first equality in~\eqref{ker_wh_Gamma} is clear. Assume that
$\wh\Gamma_0\fg=\wh\Gamma_1\fg=0$ for some $\wh g\in A^{\langle*\rangle}$.
Then $\wh g\in \bA_0$ and, by Lemma~\ref{lem:bf_A0} (ii),
\[
\wh g=\begin{bmatrix}
\wt R_\lambda^0\ff  \\
\ff +\lambda\wt R_\lambda^0\ff
\end{bmatrix}\quad\text{for some}\quad \ff\in\sH_-,\quad \lambda\in\rho(A_0).
\]
By~\eqref{eq:wh_Gamma}, we get the equality
$\wh\Gamma_1\wh g=\gamma(\ov\lambda)^{\langle *\rangle}\ff =0$
which, by Lemma~\ref{lem:9.5B}(vi), implies that $\ff\in\ran(\bA-\lambda I_\sH)$.
In view of  Lemma~\ref{lem:bf_A0}(i), we get
$\wh g
%=\begin{bmatrix}
%\wt R_\lambda^0\ff  \\
%\ff +\lambda\wt R_\lambda^0\ff
%\end{bmatrix}
\in\bA$.
\end{proof}
The triple $(\dC^p,\wh\Gamma_0,\wh\Gamma_1)$ will be called an {\it extended boundary triple for} $A^{\langle*\rangle}$.
\begin{corollary}\label{cor:ExtGreen}
For all $\wh f=\begin{bmatrix}
f \\
\ff'
\end{bmatrix}\in A^{\langle *\rangle}$ and
$\wh g=\begin{bmatrix}
g \\
g'
\end{bmatrix}\in A^{[*]}$  the following identity holds
\begin{equation}\label{eq:ExtGreen}
  \langle \ff', g \rangle_{-,+}-[ f, g']_\sH
  =(\Gamma_0\wh g)^*(\wh\Gamma_1\wh f)-(\Gamma_1\wh g)^*(\wh\Gamma_0\wh f).
%  =(\wh\Gamma_1\wh f,\wh\Gamma_0\wh g)_{\dC^p}-(\wh\Gamma_0\wh f,\wh\Gamma_1\wh g)_{\dC^p}.
\end{equation}
\end{corollary}
%%%%%%%%%%%%%%%%%%%%%%%%%%%%%%%%%%%%%%%%%%%%%%%%%%%%%%%%%%%%%%%%%%%%%%%%%%%%%%%%%%%%
In the following lemma we present an extended version of the abstract Green formula to the set $(A^{\langle*\rangle}+\wh{\sL})\cap \sH^2$, where
$\wh{\sL}=\begin{bmatrix}
       0 \\
       {\sL}
     \end{bmatrix}$ and ${\sL}$ is a subspace of $\sH_-$ such that $\dim\sL=p$.
Let us fix a linear homeomorphism ${L}\in\cB({\dC^p},{\sL})$.
\begin{lemma}\label{lem:ExtGreenForm}
Let  $\wh\Pi=  (\dC^p,\wh\Gamma_0,\wh\Gamma_1) $ be an extended boundary triple for
$A^{\langle*\rangle}$,
let $\bR$ be the regularizer given by~\eqref{eq:cR},
let ${\sL}\subset\sH_-$ be a closed subspace of $\sH_-$, let
$\wh f=\begin{bmatrix}
                 f \\
                 \ff'
               \end{bmatrix}$,
               $\wh g=\begin{bmatrix}
                 g \\
                 \fg'
               \end{bmatrix}\in A^{\langle*\rangle}$
and let $u,v\in{\dC^p}$ be such that $\ff'+Lu, \fg'+Lv\in \sH$.
%for $A$ such that$\rho(A,{\sL})\cap\rho({A}_0)\ne\emptyset$.
Then
\begin{enumerate}
  \item [(i)] $f+\wt\bR Lu\in\sH_+$ and $g+\wt\bR Lv\in\sH_+$;
  \item [(ii)] the following equality holds:
\begin{multline}\label{eq:ExtGreen2}
  [\ff'+Lu,g]_\sH-[f, \fg'+Lv]_\sH
  =(\wh\Gamma_0\wh g)^*(\wh\Gamma_1\wh f)-(\wh\Gamma_1\wh g)^*(\wh\Gamma_0\wh f)\\
  %(\wh\Gamma_1\wh f,\wh\Gamma_0\wh g)_{\dC^p}-(\wh\Gamma_0\wh f,\wh\Gamma_1\wh g)_{\dC^p}\\
  +\left\langle f+\wt\bR Lu, Lv\right\rangle_{+,-}-\left\langle Lu,g+\wt\bR Lv\right\rangle_{-,+}.
\end{multline}
\end{enumerate}
\end{lemma}
\begin{proof}
(i) By \eqref{eq:A+_decom0}, for $\lambda\in \rho({A}_0)$ the vector $\wh g$ admits the representation
%there exist vectors
%$\wh g_0=\begin{bmatrix}
%                 g_0 \\
%                 g_0'
%               \end{bmatrix}\in\bA_0$ and
%$\wh g_\lambda=\begin{bmatrix}
%                 g_\lambda \\
%                \lambda g_\lambda
%               \end{bmatrix}\in\wh\sN_\lambda$ such that
\[
\begin{bmatrix}
                 g \\
                 \fg'
               \end{bmatrix}=
\begin{bmatrix}
                 g_0 \\
                 g_0'
               \end{bmatrix}
               +\begin{bmatrix}
                 g_\lambda \\
                \lambda g_\lambda
               \end{bmatrix}\quad
\text{for some $\wh g_0=\begin{bmatrix}
                 g_0 \\
                 g_0'
               \end{bmatrix}\in\bA_0$ and
$\wh g_\lambda=\begin{bmatrix}
                 g_\lambda \\
                \lambda g_\lambda
               \end{bmatrix}\in\wh\sN_\lambda$}.
\]
Then
\[
h:=\fg'-\lambda g +Lv=g_0'-\lambda g_0 +Lv\in\sH,
\]
By Lemma~\ref{lem:bf_A0} (i), $\wt R_\lambda^0h=g_0+\wt R_\lambda^0Lv$ and hence
\[
g+\wt\bR Lv=g_0+\wt\bR Lv+g_\lambda=\wt R_\lambda^0h-(\wt R_\lambda^0-\wt\bR) Lv
+g_\lambda\in\sH_+.
\]

\noindent
(ii)
If $\wh g=\wh g_\lambda\in\wh\sN_\lambda$, the equality~\eqref{eq:ExtGreen2} is reduced to~\eqref{eq:ExtGreen}. Therefore, it is enough to prove~\eqref{eq:ExtGreen2} for
$\wh g\in\bA_0$. Let us choose $\wh g_n=\begin{bmatrix}
                 g_n \\
                 g_n'
               \end{bmatrix}\in A_0$ such that $g_n\stackrel{\sH}{\longrightarrow}g$, $g_n'\stackrel{\sH_-}{\longrightarrow}\fg'$.
Then, by Lemma~\ref{lem:bf_A0Pi}, $\Gamma_j\wh g_n\to \wh\Gamma_j\wh g$ for $j=0,1$ and,
 by~\eqref{cor:ExtGreen}, %the right hand part RHP of~\eqref{eq:ExtGreen2} can be rewritten as
 we get
  \begin{equation}\label{eq:ExtGreen3}
  (\Gamma_0\wh g_n)^*(\wh\Gamma_1\wh f)-(\Gamma_1\wh g_n)^*(\wh\Gamma_0\wh f)
%  (\wh\Gamma_1\wh f,\wh\Gamma_0\wh g_n)_{\dC^p}-(\wh\Gamma_0\wh f,\wh\Gamma_1\wh g_n)_{\dC^p}
  =\langle\ff',g_n\rangle_{-,+}-[f, g_n']_\sH=C_n+D_n
  \end{equation}
  where
  \begin{equation}\label{eq:ExtGreenCn}
C_n=[\ff'+Lu,g_n]_\sH-\langle f+\wt\bR Lu,  g_n'+Lv\rangle_{-,+}
  \end{equation}
and
  \begin{equation}\label{eq:ExtGreenDn}
  \begin{split}
D_n&=\langle f+\wt\bR Lu, Lv\rangle_{+,-}+[\wt\bR Lu,g_n']_\sH-\langle  Lu, g_n\rangle_{-,+}\\
&=\langle f+\wt\bR Lu, Lv\rangle_{+,-}+\langle Lu,\bR g_n'- g_n\rangle_{-,+}
%\to \langle f+\bR Lu, Lv\rangle_{+,-}+\langle Lu, \bR \fg'-g\rangle_{-,+}
\end{split}
  \end{equation}
Since %$\begin{bmatrix}
%                 g_n \\
%                 g_n'
%               \end{bmatrix}\in A_0$ we get
$\begin{bmatrix}
                 g_n \\
                 g_n'-\lambda g_n
               \end{bmatrix}\in A_0-\lambda I$
and $\begin{bmatrix}
                 g \\
                 \fg'-\lambda g
               \end{bmatrix}\in \bA_0-\lambda I$
we get, by Lemma~\ref{lem:bf_A0}(i),
\[
R_\lambda^0( g_n'-\lambda g_n)=g_n,\quad
\wt R_\lambda^0( \fg'-\lambda g)=g,
\]
and hence, by Lemma~\ref{lem:3.2B}(ii),
\[
     \wt R_\lambda^0 g_n'-g_n=\lambda  \wt R_\lambda^0 g_n\stackrel{\sH_+}{\longrightarrow}\lambda \wt R_\lambda^0 g
     = \wt R_\lambda^0 \fg'-g\in\sH_+ ,\quad\text{for all $\lambda\in\rho(A_0)$}.
\]
Therefore,
\[
\bR g_n'-g_n \stackrel{\sH_+}{\longrightarrow}\wt\bR \fg'-g\quad\text{as}\quad n\to\infty.
\]
and so
\[
D_n\to\langle f+\wt\bR Lu, Lv\rangle_{+,-}+\langle Lu, \wt\bR \fg'-g\rangle_{-,+}
\quad\text{as}\quad n\to\infty.
\]
Passing to the limit in \eqref{eq:ExtGreenCn} and \eqref{eq:ExtGreen3} as $n\to\infty$ we obtain
\[
C_n\to
[\ff'+Lu,g]_\sH-[f, \fg'+Lv]_\sH-[\wt\bR Lu,  \fg'+Lv]_\sH
\]
and thus
  \begin{equation}\label{eq:ExtGreen4}
  \begin{split}
  &(\wh\Gamma_1\wh f,\wh\Gamma_0\wh g)_{\dC^p}-(\wh\Gamma_0\wh f,\wh\Gamma_1\wh g)_{\dC^p}
  =[\ff'+Lu,g]_\sH-[f, \fg'+Lv]_\sH\\
  &
 +\langle f+\wt\bR Lu, Lv\rangle_{+,-}+\langle Lu, \wt\bR \fg'-g\rangle_{-,+}
 -
\langle Lu, \wt\bR \fg'+\wt\bR Lv\rangle_{-,+}.
\end{split}
  \end{equation}
The latter equality is equivalent to the equality~\eqref{eq:ExtGreen2}.
\end{proof}
%%%%%%%%%%%%%%%%%%%%%%%%%%%%%%%%%%%%%%%%%%%%%%%%%%%%%%%%%%%%%%%%%%%

\section{${\sL}$-resolvents}\label{sec:4.1}
Here we introduce the ${\sL}$-resolvents  of a symmetric linear relation $A$ in a Pontryagin space %a $\pi_\kappa$-space
 $(\sH,[ \cdot,\cdot]_{\sH})$ and present their description
 %of all ${\sL}$-resolvents of $A$
 via the so-called  ${\sL}$-resolvent matrix.
 In the case of  a Hilbert space symmetric operator $A$ with $n_\pm(A)=1$ these notions
were  originally introduced by   M.G. Kre\u{\i}n~\cite{Kr44}.
\subsection{Gauge ${\sL}$ and mvf's $\cP(\lambda)$, $\cQ(\lambda)$.}
The notion of a ${\sL}$-regular point of a Hilbert space symmetric operator $A$ with $n_\pm(A)=1$ with respect to a proper gauge ${\sL}$, i.e. closed subspace  of $\sH$,
was introduced by M.G. Kre\u{\i}n~\cite{Kr49}.
 This notion was generalized to the case of improper gauge ${\sL}$ that is not contained in $\sH$ by Yu.L. Shmul'yan~\cite{Sh71}, see also \cite{ShTs77}, \cite{LaTe82}.
 In the case when the gauge ${\sL}$ is a proper subspace of a $\pi_\kappa$-space $\sH$ the set of ${\sL}$-resolvents of a symmetric linear operator $A$ in $\sH$ was described in~\cite{D97}, the case of improper gauge  was considered in \cite{KW98} and \cite{DD19}, see also~\cite{DD24}. Here
we are mainly interested in the case when the  gauge ${\sL}$ is  improper, i.e ${\sL}\not\subset\sH$, but ${\sL}\subset\sH_-$ where $\sH_-(\supset \sH)$ is a space of distributions and calculate the $\sL$-resolvent matrix of $A$ in terms of a boundary triple.
\begin{definition}\label{def:L_reg}
Let  $\sH_+\subset\sH\subset\sH_-$ be a rigged Pontryagin space associated with
a symmetric linear relation $A$,
and let ${\sL}$ be a closed subspace of $\sH_-$.
%let ${\sL}$ be an auxiliary Hilbert space, and let $G\in\cB({\sL},{\sL})$ be invertible, i.e. $G^{-1}\in\cB({\sL},{\sL})$.
\begin{enumerate}
\item[\rm(i)]  A point $\lambda\in\wh\rho(A)$ is called ${\sL}$-{\it regular} for the operator $A$, if
\begin{equation}\label{eq:Lres2}
    \sH_-=\ran(\bA-{\lambda}I_\sH)\dotplus{\sL}.
\end{equation}
The set of all ${\sL}$-regular points of the operator $A$ is denoted by $\rho(A,{\sL})$.
\item[\rm(ii)] For $\lambda\in\rho(A,{\sL})$ denote by $\Pi_{{\sL}}^\lambda$ the projection in $\sH_-$ onto ${\sL}$ parallel to
$\ran(\bA-{\lambda}I_\sH)$.
\end{enumerate}
\end{definition}
A  closed subspace ${\sL}\subset\sH_-$ with the property that $\rho(A,{\sL})\ne\emptyset$ will be called a gauge of $A$.

\begin{lemma}\label{lem:PQ_prop}
Let  $\sH_+\subset\sH\subset\sH_-$ be a rigged Pontryagin space associated with
a symmetric linear relation $A$,
let ${\sL}$ be a subspace of $A$ of dimension $p$,
%let ${\sL}$ be an auxiliary Hilbert space,
let ${L}\in\cB({\dC^p},{\sL})$ be invertible,
%i.e. ${L}^{-1}\in\cB({\sL},{\dC^p})$,
and let
$\lambda\in\wh\rho(A)$. Then
%\begin{equation}\label{eq:11rhoAL}
\begin{enumerate}
\item[\rm(i)]
$\lambda\in\rho(A,{\sL})\,\Longleftrightarrow  \,$ ${L}^{\langle*\rangle}\!\upharpoonright\!\sN_{\overline\lambda}$ is invertible, i.e. $({L}^{\langle*\rangle}\!\upharpoonright\!\sN_{\overline\lambda})^{-1}\in\cB({\dC^p},\sN_{\overline\lambda})$.
\medskip
\item[\rm(ii)]
The operator function
\begin{equation}\label{eq:P_lambda}
   \cP(\lambda)={L}^{-1}\Pi^\lambda_{\sL}:\sH_-\to\dC^p,\quad \lambda\in\rho(A,\sL),
\end{equation}
 takes values in $\cB(\sH_-,{\dC^p})$ and
$\cP(\lambda)^{\langle*\rangle}=({L}^{\langle*\rangle}\!\upharpoonright\!\sN_{\overline\lambda})^{-1}
\in\cB({\dC^p},\sH_+)$.
\medskip
 \item[\rm(iii)] For  $\lambda\in\rho(A,{\sL})$ we have
\begin{equation}\label{eq:PQ_pr1}
  \cP(\lambda){L}=I_p,\quad
  {L}^{\langle*\rangle} \cP(\lambda)^{\langle*\rangle}=I_{p}.
  \end{equation}

\item[\rm(iv)]
If ${\mathbf R}_{\lambda}$ is a generalized resolvent of $A$ and $\wh{\mathbf R}_{\lambda}$  is a regularized extended generalized resolvent % with a minimal representing relation $\wt A$ such that
holomorphic at $\lambda\in\rho(A,{\sL})$ then
the operator function %$\cQ(\cdot)$ be defined by
\begin{equation}\label{eq:Q_lambda}
   %\cP(\lambda)=L^{-1}P(\lambda),\quad
   \cQ(\lambda)={L}^{\langle*\rangle}
   ({\mathbf R}_{\lambda}- \wh{\mathbf R}_{\lambda}\Pi_{{\sL}}^\lambda)
   :\sH\to\dC^p,\quad \lambda\in\rho(A,\sL),
\end{equation}
 takes values in $\cB(\sH,{\dC^p})$ and does not depend on the choice of generalized resolvent ${\mathbf R}_{\lambda}$ and
 $\cQ(\lambda)^{\langle*\rangle}=\cQ(\lambda)^{[*]}\in\cB({\dC^p},\sH)$.
\medskip
\item[\rm(v)] For  $\lambda\in\rho(A,{\sL})$  the operator $\cQ(\lambda)-{L}^{\langle*\rangle}\bR\in\cB(\sH,{\dC^p})$ admits a continuation $\wt\cQ(\lambda)={L}^{\langle*\rangle}\wh{\mathbf R}_{\lambda}(I-{L}\cP(\lambda))\in\cB(\sH_-,{\dC^p})$ and
\begin{equation}\label{eq:PQ_pr2}
  \wt\cQ(\lambda){L}=O_{p},\quad {L}^{\langle*\rangle}\wt\cQ(\lambda)^{\langle*\rangle}={L}^{\langle*\rangle} (\cQ(\lambda)^{\langle*\rangle}-\wt\bR {L})=O_{p}.
\end{equation}
\end{enumerate}
\end{lemma}
\begin{proof}
  (i)--(iii) For $\lambda\in\rho(A,{\sL})$ we have $\cP(\lambda){L}=I_{p}$ and hence
  ${L}^{\langle*\rangle} \cP(\lambda)^{\langle*\rangle}=I_{p}$. Since $\cP(\lambda)^{\langle*\rangle}{\dC^p}\subseteq \sN_{\overline\lambda}$, this implies that
  actually $\cP(\lambda)^{\langle*\rangle}{\dC^p}=\sN_{\overline\lambda}$ and
  the operator ${L}^{\langle*\rangle}\!\upharpoonright\!\sN_{\overline\lambda}\in\cB(\sN_{\overline\lambda},{\dC^p})$ %is surjective and hence also
  is invertible.
\medskip

(iv) The definition~\eqref{eq:Q_lambda} of $Q(\lambda)$ is correct since for $f\in\sH$ we have
${\mathbf R}_{\lambda}f\in\sH_+$ and
\[
({\mathbf R}_{\lambda}- \wh{\mathbf R}_{\lambda}\Pi_{{\sL}}^\lambda)f=
 \wh{\mathbf R}_{\lambda}(I-\Pi_{{\sL}}^\lambda)f+\bR f\in\sH_+.
\]
If ${\mathbf R}_{\lambda}^{(1)}$ is another generalized resolvent of $A$ and
$\wh{\mathbf R}_{\lambda}^{(1)}$  is the corresponding regularized extended generalized resolvent % with a minimal
holomorphic at $\lambda\in\rho(A,{\sL})$ and if
\[
f=(h'-\lambda h)+\fg\quad\text{for some}\quad \begin{bmatrix}
                                                h \\
                                                h'
                                              \end{bmatrix}\in\bA
\quad\text{for some}\quad \fg\in{\sL},
\]
then, by Lemma~\ref{lem:9.5B}(iii), we get
\[
\begin{split}
 {L}^{\langle*\rangle}({\mathbf R}_{\lambda}- \wh{\mathbf R}_{\lambda}\Pi_{{\sL}}^\lambda)f
&-{L}^{\langle*\rangle}({\mathbf R}_{\lambda}^{(1)}-\wh{\mathbf R}_{\lambda}^{(1)}\Pi_{{\sL}}^\lambda)f
  ={L}^{\langle*\rangle}(\wh{\mathbf R}_{\lambda}-\wh{\mathbf R}_{\lambda}^{(1)})(I_{\sH}-\Pi_{{\sL}}^\lambda)f\\
  &= {L}^{\langle*\rangle}(\wh{\mathbf R}_{\lambda}-\wh{\mathbf R}_{\lambda}^{(1)})(h'-\lambda h)
={L}^{\langle*\rangle}(h-h)=0.
\end{split}
\]

(v) Since $\wt\cQ(\lambda)={L}^{\langle*\rangle}\wh{\mathbf R}_{\lambda}(I-\Pi_\sL^\lambda)\in\cB(\sH_-,{\dC^p})$ it holds that
$\wt\cQ(\lambda){L}=O_{p}$ for $\lambda\in\rho(A,{\sL})$.
The second equality in \eqref{eq:PQ_pr2} follows from the first and the equality
$\wt\cQ(\lambda)=\cQ(\lambda)-L^{\langle*\rangle}\wt\bR$.
% and $ \cQ(\lambda)$, and  those in~\eqref{eq:PQ_pr2} follow from  \eqref{eq:PQ_pr1}.
%Next, for $\lambda\in\rho(A,{\dC^p})$, $ f_{\ov \lambda}\in\sN_{\ov \lambda}$, and $ h\in\sH$ we get
%\[
%  (\cP(\lambda)^{\langle*\rangle}P_{\dC^p} f_{\ov\lambda},h]_\sH=(P_{\dC^p} f_{\ov\lambda},\cP({\lambda}) h)_{\dC^p}
%  =(f_{\ov\lambda},\cP({\lambda}) h]_\sH=(f_{\ov\lambda},h)_{\sH}.
%  \]
%  This proves that $\cP(\lambda)^{\langle*\rangle}(P_{\dC^p} \upharpoonright\sN_{\ov\lambda})=I_{\sN_{\ov\lambda}}$ and hence
%  $(P_{\dC^p}\!\upharpoonright\!\sN_{\overline z})^{-1}=\cP(\lambda)^{\langle*\rangle}$.
\end{proof}
In particular, by Lemma~\ref{lem:PQ_prop}(iii), we obtain the following statement.
\begin{corollary}
For all  $\lambda\in\rho(A,{\sL})\cap\rho(A_0)$  the ovf $\cQ(\lambda)$ can be calculated by
\begin{equation}\label{eq:Q0}
  \cQ(\lambda)={L}^{\langle*\rangle}
   ( R^0_{\lambda}- \wh R^0_{\lambda}\Pi_{{\sL}}^\lambda),
\end{equation}
where $R^0_{\lambda}$ is the resolvent of the selfadjoint extension $A_0$ of $A$.
\end{corollary}
\begin{lemma}
For  ${\lambda}\in\rho(A,{\sL})$ let us set
\begin{equation}\label{eq:whP_whQ}
    \wh\cP(\lambda):=\begin{bmatrix}
                       \cP(\lambda) &
                       \lambda\cP(\lambda)
                     \end{bmatrix}\quad\text{and}\quad
    \wh\cQ(\lambda):=\begin{bmatrix}
                       \cQ(\lambda) &
                       \lambda\cQ(\lambda)+{L}^{\langle*\rangle}
                     \end{bmatrix}.
\end{equation}
Then
\begin{equation}\label{eq:PQ*}
    \wh\cP(\lambda)^{\langle*\rangle}=\begin{bmatrix}
                       \cP(\lambda)^{\langle*\rangle} \\
                       \ov\lambda\cP(\lambda)^{\langle*\rangle}
                     \end{bmatrix}\quad\text{and}\quad
    \wh\cQ(\lambda)^{\langle*\rangle}=\begin{bmatrix}
                       \cQ(\lambda)^{\langle*\rangle} \\
                       \ov\lambda\cQ(\lambda)^{\langle*\rangle}+L
                     \end{bmatrix}, %\quad u\in{\dC^p}.
\end{equation}
and
\begin{enumerate}
  \item [(i)]  $\wh\cP(\lambda)^{\langle*\rangle}u$ and $\wh\cQ(\lambda)^{\langle*\rangle}u\in A^{\langle*\rangle}$ for all $u\in{\dC^p}$.
  \item [(ii)] For ${\lambda}\in\rho_s(A,{\sL}):=\rho(A,{\sL})\cap\overline{\rho(A,{\sL})}$ the following   direct sum decomposition holds
\begin{equation}\label{eq:A+_decom}
    A^{\langle*\rangle}=\bA\dotplus\wh\cP(\lambda)^{\langle*\rangle}{\dC^p}\dotplus\wh\cQ(\lambda)^{\langle*\rangle}{\dC^p}.
\end{equation}
\end{enumerate}
\end{lemma}
\begin{proof}
(i)
The inclusion
$\wh\cP(\lambda)^{\langle*\rangle}{\dC^p}\subseteq A^{\langle*\rangle}$ is clear.
Next, for $\begin{bmatrix}
                                               h \\
                                               h'
                                             \end{bmatrix}\in A$ and $u\in{\dC^p}$ we obtain,
by~\eqref{eq:Q_lambda}, \eqref{eq:Q0} and Lemma~\ref{lem:9.5B}(iii),
\[
\begin{split}
  \left\langle h,\ov\lambda \cQ(\lambda)^{\langle*\rangle}u+{L}u\right\rangle_{+,-}- [h',\cQ(\lambda)^{\langle*\rangle}u]_\sH
  & =u^*\left\{{L}^{\langle*\rangle}h-\cQ(\lambda)(h'-\lambda h)\right\} \\
     & =u^*\left\{{L}^{\langle*\rangle}h-{L}^{\langle*\rangle}R_\lambda^0(h'-\lambda h)\right\} =0.
\end{split}
\]
%\[
%\begin{split}
%  \left\langle\ov\lambda \cQ(\lambda)^{\langle*\rangle}u+{L}u,h\right\rangle_{-.+}- (\cQ(\lambda)^{\langle*\rangle}u,h')_\sH
%  & =(u,{L}^{\langle*\rangle}h)_{\dC^p}-(u,\cQ(\lambda)(h'-\lambda h))_{\dC^p} \\
%     & =(u,{L}^{\langle*\rangle}h)_{\dC^p}-(u,{L}^{\langle*\rangle}R_\lambda^0(h'-\lambda h))_{\dC^p} =0
%\end{split}
%\]
Hence, $\wh\cQ(\lambda)^{\langle*\rangle}{\dC^p}\subseteq A^{\langle*\rangle}$.
\medskip

(ii)
The inclusion $A\dotplus\wh\cP(\lambda)^{\langle*\rangle}{\dC^p}\dotplus \wh\cQ(\lambda)^{\langle*\rangle}{\dC^p}\subseteq A^{\langle*\rangle}$ follows from (i) and~\eqref{eq:bf_A_A0}.

Let us prove the converse inclusion. Let $\begin{bmatrix}
                                            f &  \ff'
                                          \end{bmatrix}^\top\in A^{\langle*\rangle}$.
Since $\ov{\lambda}\in\rho(A,{\sL})$, there exist
$\begin{bmatrix}
                                            g &  \fg'
                                          \end{bmatrix}^\top\in \bA$ and $u\in{\dC^p}$ such that
\begin{equation}\label{eq:11.10}
  \ff'-\overline{\lambda}f=\fg'-\overline{\lambda}g+{L}u.
\end{equation}
It follows from \eqref{eq:11.10} and the relations %~\eqref{eq:P_hat} that
\[
\begin{bmatrix}
f \\  \ff'-\overline{\lambda}f
\end{bmatrix},\
\begin{bmatrix}
g \\  \fg'-\overline{\lambda}g
\end{bmatrix},\
\begin{bmatrix}
\cQ(\lambda)^{\langle*\rangle}u  \\  {L}u
\end{bmatrix}\in A^{\langle*\rangle}-\overline{\lambda}I
\]
that
\begin{equation}\label{eq:11.11}
 \begin{bmatrix}
 f-g-\cQ(\lambda)^{\langle*\rangle}u \\
        0
 \end{bmatrix}\in A^{\langle*\rangle}-\overline{\lambda}I.
\end{equation}
Hence $ f-g-\cQ(\lambda)^{\langle*\rangle}u\in\ker(A^{\langle*\rangle}-\overline{\lambda}I)
=\sN_{\overline{\lambda}}$.
By the assumption ${\lambda}\in\rho(A,{\sL})$ and Lemma~\ref{lem:PQ_prop}(ii), the equality $\cP(\lambda)^{\langle*\rangle}{\dC^p}=\sN_{\overline{\lambda}}$ holds, and hence
there is a vector $v\in{\dC^p}$ such that
\begin{equation}\label{eq:11.10v}
f-g-\cQ(\lambda)^{\langle*\rangle}u=\cP(\lambda)^{\langle*\rangle}v.
\end{equation}
Therefore,
\[
\begin{bmatrix}
f \\
\ff'
\end{bmatrix}
=\begin{bmatrix}
g+\cP(\lambda)^{\langle*\rangle}v+\cQ(\lambda)^{\langle*\rangle}u\\
\fg'+\overline{\lambda}\cP(\lambda)^{\langle*\rangle}v+\overline{\lambda}\cQ(\lambda)^{\langle*\rangle}u+Lu
\end{bmatrix}
\in \bA\dotplus\wh\cP(\lambda)^{\langle*\rangle}{\dC^p}\dotplus\wh\cQ(\lambda)^{\langle*\rangle}{\dC^p}.
\]
Moreover, the decomposition \eqref{eq:A+_decom} is direct, since $u$, $v$, $g$ and $\fg'$ are uniquely determined
by~\eqref{eq:11.10} and~\eqref{eq:11.10v}:
\begin{equation}\label{eq:11.12}
\begin{split}
  u&=\cP(\overline{\lambda})(\ff'-\ov{\lambda} f),\quad
  g=R_{\overline{\lambda}}^0(I-\Pi_{{\sL}}^{\overline{\lambda}})(\ff'-\ov{\lambda} f),\\
 v &={L}^{\langle*\rangle}(f-g),\qquad
 \fg'=\ff'-\overline{\lambda}(f-g)-{L}u.
 \end{split}
\end{equation}
%If $\begin{bmatrix}
%      0 & \ff'
%    \end{bmatrix}^\top\in A^{\langle*\rangle}$, then $h:=f'\in\sH\ominus\dom A$ and \eqref{eq:11.12} is reduced to~\eqref{eq:11.5.9}.
\end{proof}

\subsection{The ${\sL}$-preresolvent matrix}
Let $\widetilde A$ be a self-adjoint extension of $A$ in a possibly  larger Pontryagin space $\wt\sH(\supseteq\sH)$. Consider $A$ as a linear relation $A'$ in $\wt\sH$. Then the adjoint linear relation of $A$ in  $\wt\sH$ is equal to
\[
(A')^{[*]}=\left\{\begin{bmatrix}
                  f+h \\
                  f'+h'
                \end{bmatrix}:\,
                \begin{bmatrix}
                  f\\
                  f'
                \end{bmatrix}\in A^{[*]},\, h,h'\in \wt\sH[-]\sH\right\}.
\]
Let us consider $\sH^\perp:=\wt\sH[-]\sH$ as a subspace of the Hilbert space $\wt\sH$.
Consider $\wt\sH_+:=\dom (A')^{[*]}=\sH_+\oplus \sH^\perp$ as a  Hilbert space endowed with the  norm
\begin{equation}\label{eq:A+Norm2}
  \| \wt f\|^2_{\wt\sH_+}=\|f\|_{\sH_+}^2+\|h\|_\sH^2\quad
\mbox{for }\quad \wt f=f+h,\quad f\in\sH_+,\quad h\in \sH^\perp.
\end{equation}
Let $\sH_-$ be the dual Hilbert space that was  introduced in Lemma~\ref{lem:RiggedPontrSp}.
Then the Hilbert space $\wt\sH_-:=\sH_-\oplus \sH^\perp$ can be treated as a dual space to $\wt\sH_+$
with respect to the duality
\begin{equation}\label{eq:wtHpm}
\langle\ff+f^\perp, h+h^\perp\rangle^{(\wt\sH)}_{-,+}
:=\langle\ff, h\rangle_{-,+}+(f^\perp, h^\perp)_{\wt\sH}\quad
 \text{ for  $\ff\in \sH_-$, $h\in\sH_+$, $f^\perp,h^\perp\in\sH^\perp$}.
\end{equation}
Denote the resolvent of $\wt A$ by ${R}_\lambda(\wt A):=(\wt A-\lambda I_{\wt\sH})^{-1}$ ($\lambda\in\rho(\wt A)$) and let the extended  resolvent  $\wt{R}_\lambda(\wt A)$ of $\wt A$ be defined by
\begin{equation}\label{eq:9.14BGc}
  [\widetilde{R}_\lambda(\wt A)\ff,\wt h]_{\wt \sH}=
  \langle\ff, {R}_{\ov{\lambda}}(\wt A)\wt h\rangle^{(\wt\sH)}_{-,+}
  \quad \textrm{for}\ \wt h\in\wt\sH,\ \wt\ff\in\wt\sH_-,\  \lambda\in\rho(\widetilde A).
\end{equation}
  \begin{lemma}\label{lem:ExtRes_wtA}
  Let  $\widetilde{R}_\lambda(\wt A)$ be the extended  resolvent and
  let $\wt{\mathbf R}_{\lambda}$ be the extended generalized resolvent of $\wt A$. Then
 $\widetilde{R}_\lambda(\wt A)\in\cB(\wt\sH_-,\wt\sH)$
 and $P_{\sH}\wt{R}_\lambda(\wt A)\upharpoonright\sH_-=\wt{\mathbf R}_{\lambda} $
 for all $\lambda\in\rho(\widetilde A)$.
\end{lemma}
\begin{proof}
  By Lemma~\ref{lem:9.5B}(iii) and \eqref{eq:wtHpm}, we get for $\ff\in\sH_-$, $h\in\sH$
\[
[\wt{R}_\lambda(\wt A)\ff,h]_{\wt\sH}=\langle\ff, {R}_{\ov{\lambda}}(\wt A)h\rangle^{(\wt\sH)}_{-,+}
=\langle\ff, {\mathbf R}_{\overline{\lambda}} h\rangle_{-,+}
=[\wt{\mathbf R}_{\lambda}\ff,h]_{\sH}.
\]
\end{proof}

%%%%%%%%%%%%%%%%%%%%%%%%%%%%%%%%%%%%%%%%%%%%%%%%%%%%%%%%%%%%%%%%%%%%%%%%%%%%%%%%%%%%%%%%%%%%
\begin{definition}\label{def:11.Pl-res,matr}
Let ${\mathbf R}_{\lambda}$ be a generalized resolvent of $A$ of index $\wt\kappa$, let $\wh{\mathbf R}_{\lambda}$ be a regularized extended generalized resolvent
and let  ${\sL}\subset\sH$ be a gauge for $A$.  The
 operator function
\begin{equation}\label{eq:G_res}
r(\lambda):={L}^{\langle*\rangle}\wh{\mathbf R}_{\lambda}{L}={L}^{\langle*\rangle}(\wt{\mathbf R}_{\lambda}-\bR){L}
\end{equation}
is called a ${\sL}${\bf-resolvent} of $A$. Let
\begin{equation}\label{eq:9.10K}
  \sH':=\overline{\textup{span}}\left\{
  \widetilde{R}_{\omega}(\wt A){\sL}:\,\omega\in\rho(\wt A)\right\}.
\end{equation}
The ${\sL}$-resolvent ${L}^{\langle*\rangle}\wh{\mathbf R}_{\lambda}{L}$ is said to be of index $\kappa'$, if $\textup{ind}_-(\sH')=\kappa'(\le\wt\kappa)$.
A selfadjoint extension $\wt A$ of $A$ and the ${\sL}$-resolvent ${L}^{\langle*\rangle}\wh{\mathbf R}_{\lambda}{L}$
are called ${\sL}${\bf-regular} if $\kappa'=\wt\kappa$, i.e.,
 the negative index of the space $\sH'$
 coincides with the negative index of the space $\wt \sH$ in the minimal representation of the generalized resolvent ${\mathbf R}_\lambda$.
 %%%%%%%%%%%%%%%%%%%%%%%%%%%%%%%%%%%%%%%%%%%%%%%%%%%%%%%%%%%%%%%%%%%%%%%%%%%%%%%%%%%%%%
%In what follows we will use the notation ${\sL}\cR_{\wt\kappa}(A)$  for the set of ${\sL}$-regular ${\sL}$-resolvents of $A$.
  \end{definition}
\begin{lemma}
  \label{lem:9.6}
%Let ${\mathbf R}_{\lambda}$ be a generalized resolvent of $A$ of index $\wt\kappa$,
Let $\wh{\mathbf R}_{\lambda}$ be a regularized extended generalized resolvent  of $A$ of index $\wt\kappa$ with the representing relation $\wt A$
%  Let $A$ is a closed symmetric linear relation with equal
%defect  numbers $n_{\pm}(A)=p<\infty$ in a $\pi_\kappa$-space
%$(\sH,[ \cdot,\cdot]_{\sH})$,
%let $\wt A$ be a selfadjoint extension of $A$
in a $\pi_{\wt\kappa}$-space
$(\wt\sH,[ \cdot,\cdot]_{\wt\sH})$,  let $\sH'$ be defined by~\eqref{eq:9.10K}, $\mbox{ind}_-(\sH')=\kappa'$,
 and let $r(\lambda):={L}^{\langle*\rangle}\wh{\mathbf R}_{\lambda}{L}$ be given by~\eqref{eq:G_res}.
Then:
\begin{enumerate}
  \item [\rm(i)] $r\in\cN_{\kappa'}^{\ptp}$  $(0\le\kappa'\le\wt\kappa)$.
  \item [\rm(ii)] $r\in\cN_{\wt\kappa}^{\ptp}$  if and only if  the
   representing relation $\wt A$ is ${\sL}$-regular.
\end{enumerate}
\end{lemma}
\begin{proof}
If  $\lambda,\omega\in\rho(\wt A)$ and $u,v\in\dC^p$, then \begin{equation}\label{eq:wtF1}
  \begin{split}
u^*\frac{r(\lambda)-r(\omega)^*}{\lambda-\ov{\omega}}v
% = \xi^*\frac{c(\lambda)-c(\ov{\omega})}{\rho_\omega(\lambda)}\eta\\
       &= \frac{1}{\lambda-\ov{\omega}}
       \left[ (\wt {R}_\lambda(\wt A) -
       \wt {R}_{\ov\omega}(\wt A)){L}v,{L}u\right]_{\wt\sH}\\
 &=[\wt{R}_\lambda(\wt A) {L}v, \wt {R}_{{\omega}}(\wt A){L}u]_{\wt\sH},
    \end{split}
  \end{equation}
where $\wt {R}_\lambda(\wt A)$ is the extended resolvent of $A$.  Therefore, (i) holds.

The ${\sL}-$resolvent $r(\lambda)$ is  ${\sL}$-regular if and only if $\mbox{ind}_-(\sH')=\wt\kappa$ which, in view of \eqref{eq:wtF1},
  is equivalent to $r\in\cN_{\wt\kappa}^{\ptp}$. This proves (ii).
\end{proof}
\begin{lemma}\label{lem:PreresM}
Let $A$ is a closed symmetric linear relation with equal
defect  numbers $n_{\pm}(A)=p<\infty$ in a $\pi_\kappa$-space
$(\sH,[ \cdot,\cdot]_{\sH})$,
let $\sL$ be a gauge for $A$,
let $\Pi=(\dC^p,\Gamma_0,\Gamma_1)$ be a boundary triple for
$A^{*}$, let $M(\cdot)$ and $\gamma(\cdot)$  be the corresponding Weyl
function and  $\gamma$-field, respectively,
let $A_0 = %A^* \!\upharpoonright\!
\ker\Gamma_0$, let $\lambda\in\rho(A_0)$
and let $\wh R_\lambda^0=\wt R_\lambda^0-\bR$ be the  regularized extended  resolvent of $A_0$.
Then  the $\dC^{2p\times 2p}$--valued
 function ${\mathfrak  A}_{\Pi{\sL}}({\cdot})$  defined by
\begin{equation}\label{eq:Lres2A}
 {\mathfrak A}_{\Pi{\sL}}(\lambda) = ({\mathfrak a}_{ij}(\lambda))_{i,j=1}^2  :=
    \begin{bmatrix} M(\lambda) & \gamma(\ov\lambda)^{\langle*\rangle}{L}\\
    {L}^{\langle*\rangle}\gamma(\lambda)& {L}^{\langle*\rangle}\wh R_\lambda^0 {L}
    \end{bmatrix},
\end{equation}
has the following properties:
\begin{enumerate}%\def\labelenumi{\rm (\roman{enumi})}
\item[\rm(i)] %${\mathfrak A}_{\Pi{\sL}}({\cdot})\in \cR[\dC^p\oplus{\sL}]$
 %%belongs}  to the class $\cR[\dC^p\oplus{\sL}]$
 For all ${\lambda},\omega\in\rho(A_0)$ %the following identity holds$:$
\[%begin{equation}\label{eq:A_fact}
    {\mathsf{N}}^{{\sA}_{\Pi{\sL}}}_\omega({\lambda})
    := \frac{{\mathfrak A}_{\Pi{\sL}}({\lambda})
    - {\mathfrak A}_{\Pi{\sL}} (\omega)^*}{{\lambda} - \overline{\omega}}
    = T(\omega)^*T({\lambda}),\quad %\textup{where}\quad
T({\lambda})=\begin{bmatrix}
             \gamma({\lambda}) &
             \wt R_\lambda^0 {L}
           \end{bmatrix}.
\]%end{equation}
and ${\mathfrak A}_{\Pi{\sL}}\in \cN_{\kappa_1}^{p\times p}$ for some $\kappa_1\le\kappa$.
\item[\rm(ii)]
${\mathfrak A}_{\Pi{\sL}}\in \cN_\kappa^{p\times p}$ if and only if
%$\kappa_-(\sH_\mathfrak A)=\kappa$, where
\begin{equation}\label{eq:A_N_k}
\kappa_-(\sH_\mathfrak A)=\kappa,\quad \text{ where }\quad
\sH_\mathfrak A:=\overline{\textup{span}}\left\{T({\lambda}){\sL}:\,\lambda\in\rho(A_0)\right\}.
\end{equation}
This happens, in particular, if either $A$ is a simple symmetric operator or the extension $A_0$  is ${\sL}$-regular.
%%%%%%%%%%%%%%%%%%%%%%%%%%%%%%%%%%%%%%%%%%%%%%%%%%%%%%
 \medskip
\item[\rm(iii)]
${\lambda}\in\rho( A, {\sL})\Longleftrightarrow a_{21}({\ov\lambda})^{-1}\in\dC^{p\times p}\Longleftrightarrow a_{12}({\lambda})^{-1}\in\dC^{p\times p}$.
  \medskip
\item[\rm(iv)]   If $\kappa=0$, then ${\mathfrak A}_{\Pi{\sL}}\in \cR^{p\times p}=\cN_0^{p\times p}$.
\end{enumerate}
\end{lemma}
\begin{proof}
(i) It follows from  \eqref{eq:Lres2A}, \eqref{eq:11_gamma2} and \eqref{Eq:11.8M} that for all ${\lambda},\omega\in\rho(A_0)$
\[
\begin{split}
{\mathsf{N}}_{\mathfrak A}({\lambda},\omega)  &=\frac{1}{\lambda -\ov {\omega}}
\begin{bmatrix}
M({\lambda})-M(\ov{\omega}) &
(\gamma (\ov{\lambda})^*-\gamma (\omega)^*){L}\\
{L}^{\langle*\rangle}({\gamma}({\lambda})-{\gamma}(\ov{\omega})) &
\,\,{L}^{\langle*\rangle}(\wt R_\lambda^0-\wt R_{\ov\omega}^{0})
{L}
\end{bmatrix}\\
& =\begin{bmatrix}
\gamma (\omega)^*\gamma({\lambda}) &
\gamma (\omega)^*\wt R_\lambda^0{L}\\
{L}^{\langle*\rangle}R_{\ov\omega}^{0}\gamma({\lambda})&
\,\,{L}^{\langle*\rangle}R_{\ov\omega}^{0}\wt R_\lambda^0{L}
\end{bmatrix}
=T(\omega)^*T({\lambda}).
\end{split}
\]
This implies that for  $\lambda_i\in\rho(A_0)$, $u_i\in{\dC^p}$, $\xi_i\in\dC$, $i\in\{1,\dots,n\}$
 \begin{equation}\label{eq:form_NG}
 \sum_{i,j=1}^n u_j^*{\mathsf{N}}_{{\sA}_{\Pi{\sL}}}({\lambda_i},\lambda_j)u_i
 \xi_i\ov{\xi_j}
 = \sum_{i,j=1}^n\left[T(\lambda_i)u_i,T(\lambda_j)u_j\right]_\sH\xi_i\ov{\xi_j},
 \end{equation}
 and hence the negative index $\kappa_-({\mathsf{N}}_{{\sA}_{\Pi{\sL}}})$ of the kernel ${\mathsf{N}}_{{\sA}_{\Pi{\sL}}}({\lambda},\omega)$ coincides with $\kappa_1:=\kappa_-(\sH_\mathfrak A)$.
\medskip

 (ii) The first statement in (ii) follows from~\eqref{eq:form_NG}.
 Clearly,  each of the conditions: either~\eqref{eq:simple_A} or
 %$\kappa_-(\sH_0)=\kappa$, where
 \[
 \kappa_-(\sH_0)=\kappa,\quad \text{ where }\quad  \sH_0:=\overline{\textup{span}}\left\{\widetilde{ R}^0_\omega{\sL}:\,\omega\in\rho(A_0)\right\},
 \]
 implies that $\kappa_-(\sH_\mathfrak A)=\kappa$ and so ${\mathfrak A}_{\Pi{\sL}}\in \cN_\kappa^{p\times p}$.
\medskip

(iii)
If  ${\lambda}\in\rho(A,{\sL})$, then, by Lemma~\ref{lem:PQ_prop}(i), the operator ${L}^{\langle*\rangle}\upharpoonright\!{\sN_{\overline{\lambda}}}$ is an isomorphism from $\sN_{\overline{\lambda}}$ onto ${\dC^p}$.  Therefore, the operator $a_{21}({\overline{\lambda}})={L}^{\langle*\rangle}\gamma(\overline{\lambda})
=\left({L}^{\langle*\rangle}\!\upharpoonright\!{\sN_{\overline{\lambda}}}\right) \gamma(\overline{\lambda})$ is an isomorphism in $\dC^p$. This implies that $a_{12}({\lambda})^{-1}\in\dC^{p\times p}$, since $a_{12}({\lambda})=a_{21}({\overline{\lambda}})^*$.

Conversely, if $a_{12}({\lambda})^{-1}\in\dC^{p\times p}$, then $a_{21}({\overline{\lambda}})=a_{12}({\lambda})^*={L}^{\langle*\rangle}\gamma(\overline{\lambda})$  is an isomorphism in $\dC^p$. Hence,
${L}^{\langle*\rangle}\!\upharpoonright\!\sN_{\overline\lambda}$ is invertible
%${L}^{\langle*\rangle}{\sN_{\overline{\lambda}}}={\dC^p}$
and, by Lemma~\ref{lem:PQ_prop}(i),
 ${\lambda}\in\rho(A,{\sL})$.
\end{proof}
\begin{definition} \label{def:LpreresM}
The $\dC^{2p\times 2p}$--valued function ${\mathfrak
A}_{\Pi{\sL}}({\lambda})$  defined by~\eqref{eq:Lres2A}
is called the ${\sL}$-preresolvent matrix of $A$ corresponding to the boundary triple $\Pi$, or, shortly,  the $\Pi{\sL}$-preresolvent matrix of $A$.

\end{definition}
\index{$\Pi{\sL}$-preresolvent matrix}
%%%%%%%%%%%%%%%%%%%%%%%%%%%%%%%%%%%%%%%%%%%%%%%%%%%%%%%%%%%%%%%
  The following lemma provides a description of ${\sL}$-regular ${\sL}$-resolvents of $A$.
\begin{lemma}\label{lem:Gres_F1}
Let the assumptions of Lemma~\ref{lem:PreresM}  hold. Then
%Let $\Pi=\{\dC^p,\Gamma_0,\Gamma_1\}$ be a boundary triple for
%$A^{*}$, let $M(\cdot)$ and $\gamma(\cdot)$  be the corresponding Weyl
%function and  $\gamma$-field, respectively,
%let $A_0 = %A^* \!\upharpoonright\!
%\ker\Gamma_0$, let $\lambda\in\rho(A_0)$
%and let $\wh R_\lambda^0=\wt R_\lambda^0-\cR$ be the  regularized extended  resolvent of $A_0$.
%Then
\begin{enumerate}%\def\labelenumi{\rm (\roman{enumi})}
\item[\rm(i)] The formula
%%%%%%%%%%%%%%%%%%%%%%%%%%%%%%%%%%%%%%%%%%%%%%%%%%%%%%
% The set of ${\sL}$-resolvents of $A$ is  described by the formula
\begin{equation}\label{eq:11.Lres1}
    {L}^{\langle*\rangle}\wh{\mathbf R}_{\lambda}{L}={\mathfrak a}_{22}({\lambda})-
    {\mathfrak a}_{21}({\lambda})(\tau({\lambda})+{\mathfrak a}_{11}({\lambda}))^{-1}
    {\mathfrak a}_{12}({\lambda}),\quad z\in\dC\setminus\dR,
\end{equation}
establishes a bijective correspondence between the set %${\sL}\cR_{\wt\kappa}(A)$
of ${\sL}$-regular ${\sL}$-resolvents of $A$ of index $\wt\kappa$ and the set of
 $\tau\in\wt\cN_{\wt\kappa-\kappa}^{\ptp}$ such that
 ${L}^{\langle*\rangle}\wh{\mathbf R}_{\lambda}{L}\in \cN_{\wt\kappa}^{p\times p}$.
  \medskip
\item[\rm(ii)] Condition \eqref{eq:A_N_k} is necessary for the existence of an ${\sL}$-regular
${\sL}$-resolvent of $A$.
\end{enumerate}
\end{lemma}
\begin{proof}
  If ${L}^{\langle*\rangle}\wh{\mathbf R}_{\lambda}{L}$ is
  a  ${\sL}$-regular ${\sL}$-resolvents of $A$ of index $\wt\kappa$
  then the space $\sH'$ in~\eqref{eq:9.10K} has negative index $\wt\kappa$,
  and hence the space $\wt\sH$ in~\eqref{eq:MinGres} has negative index $\wt\kappa$.
  Therefore, the generalized resolvent $ {\mathbf R}_\omega$ has index $\wt \kappa$
  and, by Theorem~\ref{krein}, the extended generalized resolvent $ \wt{\mathbf R}_\lambda$ admits the representation
\begin{equation}\label{eq:11.Lres2}
\wt{\mathbf R}_{{\lambda}}=\wt R_{\lambda}^0-\gamma({\lambda})(M({\lambda})
     +\tau({\lambda}))^{-1}\gamma(\ov{{\lambda}})^{\langle*\rangle},\quad
     {\lambda}\in\dC_+\cup\dC_-
\end{equation}
for some $\tau\in\wt\cN_{\wt\kappa-\kappa}^{\ptp}$.
Subtracting the regularizer $\bR$ and  applying from both sides the operators
$ {L}^{\langle*\rangle}$ and ${L}$ we obtain \eqref{eq:11.Lres1}.
Since ${L}^{\langle*\rangle}\wh{\mathbf R}_{\lambda}{L}$ is
  an  ${\sL}$-regular ${\sL}$-resolvents of $A$ it belongs to the class
  $\cN_{\wt\kappa}^{p\times p}$, see Lemma~\ref{lem:9.6}.

Conversely, assume that $\tau\in\wt\cN_{\wt\kappa-\kappa}^{\ptp}$ and the mvf
 ${L}^{\langle*\rangle}\wh{\mathbf R}_{\lambda}{L}$ belongs to $\cN_{\wt\kappa}^{p\times p}$.
 Then, by Lemma~\ref{lem:9.6}, the ${\sL}$-resolvent ${L}^{\langle*\rangle}\wh{\mathbf R}_{\lambda}{L}$
 of $A$ is ${\sL}$-regular   of index $\wt\kappa$.
\medskip

 (ii) Let ${L}^{\langle*\rangle}\wh{\mathbf R}_{\lambda}{L}$ be a ${\sL}$-regular ${\sL}$-resolvent of $A$. Then
for the set $\sH'$ defined by \eqref{eq:9.10K} we have $\textup{ind}_-(\sH')=\wt\kappa$,
and hence $\sH'$  is a non-degenerate subspace of $\wt\sH$.
For the subspace
\[
P_\sH  \sH':=\overline{\textup{span}}\left\{\wt{\mathbf R}_{\lambda} {\sL}:\,\lambda\in\rho(\wt A)\right\}
\]
we get $\textup{ind}_-(P_\sH  \sH')=\kappa$. It follows from~\eqref{eq:11.Lres2} that
$P_\sH  \sH'\subset \sH_\mathfrak A$. Hence $\textup{ind}_-(\sH_\mathfrak A)=\kappa$.
\end{proof}
\subsection{Right ${\sL}$-resolvent matrix}
\label{sec:11.5}
For a  $2\times 2$ block matrix-function $W(\lambda)=\left[w_{i,j}\right]_{i,j=1}^2$ with blocks $w_{i,j}(\lambda)$ of size
${p\times p}$
%and  $\tau\in \cN_{\kappa}^{\ptp}$ we define the {\bf right linear-fractional transformation}
we define a transformation $T_W$
in the set $\wt\cC(\dC^p)$ of linear relations $\tau$ in $\dC^p$ via
\begin{equation}
\label{eq:App2:LFT}
 T_W[\tau]
  =\Big\{\,\begin{pmatrix}w_{22}h+w_{21}h' \\w_{12}h+w_{11}h'\end{pmatrix}:\,
  \begin{pmatrix} h \\ h'\end{pmatrix}\in \tau\,\Big\},
\end{equation}
see Yu. Shmul'yan~\cite{Sh80}.
Clearly, $T_W[\tau]$ is contained in the linear relation
\begin{equation}\label{eq:2.2.0}
  (w_{11}\tau+w_{12})
  (w_{21}\tau+w_{22})^{-1},
\end{equation}
however, the converse inclusion may fail to hold, see an example in~\cite{DM95}.

  In this section, for any boundary triple  $\Pi=  \{\cH,\Gamma_0,\Gamma_1\}  $ %for  $A^{*}$
and a gauge  $\sL\subset\sH$, we associate a $\Pi\sL$-resolvent matrix $W(\lambda)$ of  a symmetric operator $A$ and present a formula which describes $\sL$-resolvents
of $A$ with the help of the linear-fractional transformation $T_W$.
%and express it explicitly  via the operator functions $\cP(\cdot)$ and $\cQ(\cdot)$
%and the boundary mapping~$\Gamma$.
%
%for every $\tau\in\dC^{p\times p}$ such that the indicated inverse exists.
%Moreover we will also need to apply this transformation to linear relations $\tau$ in $\dC^p$
%such that $0\in\rho(w_{21}\tau+w_{22})$.
%

%\begin{proposition}
%\label{LFTprop}
%{\CR  Let $W\in\cW_{\kappa_1}(J_p)$. Then for every $\tau\in\wt\cN_{\kappa_2}^{p\times p}$
%the linear-fractional transformation $\wt\tau=T_{W}[\tau]$
%belongs to the class $\cN_{\kappa'}^{p\times p}$
%with $\kappa'\le\kappa_1+\kappa_2$
%}
\begin{definition}\label{def:11.Pl-res,matr2}
Let $\Pi=  (\dC^p,\Gamma_0,\Gamma_1) $ be a boundary triple for
$A^{*}$ and let ${\sL}\subset\sH_-$ be a gauge for $A$ such that
$\rho(A,{\sL})\cap\rho({A}_0)\ne\emptyset$,  let
${\mathfrak A}_{\Pi{\sL}}({\lambda})=({\mathfrak a}_{ij}({\lambda}))_{i,j=1}^2$
be  the $\Pi{\sL}$-preresolvent matrix of $A$, see~\eqref{eq:Lres2A}. The $\dC^{2p\times 2p}$--valued
function $W_{\Pi{\sL}}(\lambda)$ defined on $\rho(A,{\sL})\cap\rho({A}_0)$ by
%${\mathfrak h}_W$ by~
\begin{equation}\label{eq:LresM}
    W_{\Pi\sL}({z})=\begin{bmatrix}
    {\mathfrak a}_{22}({z}){\mathfrak a}_{12}({z})^{-1} &     {\mathfrak a}_{22}({z}){\mathfrak a}_{12}({z})^{-1}    {\mathfrak a}_{11}({z})-{\mathfrak a}_{21}({z})\\
        {\mathfrak a}_{12}({z})^{-1}                 &   {\mathfrak a}_{12}({z})^{-1}    {\mathfrak a}_{11}({z})
    \end{bmatrix}, %\quad {z}\in{\mathfrak h}_W,
\end{equation}
is called the {\bf ${\sL}$-resolvent matrix of $A$} corresponding to the boundary triple $\Pi$ or, briefly, the $\Pi{\sL}$-resolvent matrix of $A$.
  \end{definition}

For an $\cN_{\kappa}^{\ptp}$-family $\tau(\lambda)=\ran\begin{bmatrix}
    \varphi(\lambda)\\\psi(\lambda)
    \end{bmatrix}$ (see Definition~\ref{def:Nk-family})
    let us set
\begin{equation}\label{eq:Lambda}
    \Lambda_{\varphi,\psi}=\{\lambda\in \rho(A,{\sL})\cap\rho(A_0):\,
{\det}\left(w_{21}({z})\psi({z})+w_{22}({z})\varphi({z})\right)\ne0\}.
\end{equation}
If $\Lambda_{\varphi,\psi}\ne\emptyset$, then the linear fractional transform in~\eqref{eq:App2:LFT} for
${\lambda}\in\Lambda_{\varphi,\psi}$ takes the form
\begin{equation}\label{eq:Lres_phipsi}
%G^{\langle*\rangle}\wh{\mathbf R}_{z}G
 T_{W}[\tau({\lambda})]=
(w_{11}({\lambda})\psi({\lambda})+w_{12}({\lambda})\varphi({\lambda}))
(w_{21}({\lambda})\psi({\lambda})+w_{22}({\lambda})\varphi({\lambda}))^{-1}.
% \quad
%{\lambda}\in\Lambda_{\varphi,\psi}.
\end{equation}
Otherwise, it is understood in the sense of Shmul'yan, see \eqref{eq:App2:LFT}.

In the following theorem we show that the $\Pi{\sL}$-preresolvent matrix
$W=W_{\Pi\sL}$ belongs to the class $\cW_{\kappa_1}(J_p)$ (see Definition~\ref{def:11.cP}), where $J_p$ is given by~\eqref{eq:J_p}.
\begin{definition}\label{def:11.cP}
A matrix-valued  function $W(z)$ holomorphic on a domain $\Omega\subset\dC_+$ %such that $i\in\Omega$
is said to belong
to the class $\cW_{\kappa_1}(J_p)$
if the kernel
\begin{equation}\label{kerK}
{\mathsf K}_\omega(\lambda):=
\frac{ J_p-W(\lambda)J_p W(\omega)^*}{-i(\lambda-\ov\omega)}, \quad\lambda,\omega\in \rho(A,{\sL}),\quad \lambda\ne\overline{\omega}
\end{equation}
has $\kappa_1$ negative squares on $\Omega$.
\end{definition}

%\begin{definition}\label{def:11.cP}
%A matrix-valued  function $W(z)$ holomorphic on a domain $\Omega\subset\dC_+$ %such that $i\in\Omega$
%is said to belong
%to the class $\cW_{\kappa_1}(J_p)$
%if the kernel
%\begin{equation}\label{kerK}
%{\mathsf K}_\omega(\lambda):=
%\frac{ J_p-W(\lambda)J_p W(\omega)^*}{-i(\lambda-\ov\omega)}, \quad\lambda,\omega\in \rho(A,{\sL}),\quad \lambda\ne\overline{\omega}
%\end{equation}
%has $\kappa_1$ negative squares on $\Omega$.
%\end{definition}
%
%Let us define the matrix  $J_p\in\dC^{2p\times 2p}$, and the operator-valued %$\cB(\sH,\dC^{2p})$-valued
%functions $\cG(\lambda)$ and $\wh\cG(\lambda)$, ${\lambda}\in\rho(A,{\sL})$, by
%\begin{equation}\label{eq:11.cG}
%J_p=\begin{bmatrix}
%    O_{p} & -i I_{p}\\
%    i I_{p} & O_{p}
%    \end{bmatrix},\quad
%    {\cG}({\lambda})=\begin{bmatrix}
%   -\cQ({\lambda})\\
%      \cP({\lambda})
%    \end{bmatrix},\quad
%    \wh{\cG}({\lambda})=\begin{bmatrix}
%   -\wh\cQ({\lambda})\\
%      \wh\cP({\lambda})
%    \end{bmatrix}.
%    %,\quad{\lambda}\in\rho(A,{\sL}).
%\end{equation}

\begin{theorem}\label{thm:Gres_F1}
Let the assumptions of Lemma~\ref{lem:PreresM}  hold and let the mvf's
$\sA(\lambda):={\mathfrak A}_{\Pi{\sL}}({\lambda})
=({\mathfrak a}_{ij}({\lambda}))_{i,j=1}^2$ and $W(\lambda):=W_{\Pi\sL}(\lambda)$
be given  by~\eqref{eq:Lres2A} and~\eqref{eq:LresM}. Then
\begin{enumerate}\def\labelenumi{\rm (\roman{enumi})}
\item $W(\lambda)$ and $\sA(\lambda)$ are related by
\begin{equation}\label{eq:W_A}
  W(\lambda)=\begin{bmatrix}
   0 & {\mathfrak a}_{12}({\lambda})\\
  -I_p & {\mathfrak a}_{22}({\lambda})
    \end{bmatrix}^{-1}
    \begin{bmatrix}
   I_p & {\mathfrak a}_{11}({\lambda})\\
  0  & {\mathfrak a}_{21}({\lambda})
    \end{bmatrix}.
\end{equation}
 \item %[(ii)]
 The kernel ${\mathsf K}_\omega(\lambda)$
 is related to the kernel ${\mathsf{N}}^{\sA}_\omega({\lambda})$ by
\begin{equation}\label{eq:KW_NA}
{\mathsf K}_\omega(\lambda)=\begin{bmatrix}
   0 & {\mathfrak a}_{12}({\lambda})\\
  -I_p & {\mathfrak a}_{22}({\lambda})
    \end{bmatrix}^{-1}
    {\mathsf{N}}^{\sA}_\omega({\lambda})
    \begin{bmatrix}
   0 & {\mathfrak a}_{12}(\omega)\\
  -I_p & {\mathfrak a}_{22}(\omega)
    \end{bmatrix}^{-*}.
\end{equation}
\item%[\rm(iii)]
$W\in \cW_{\kappa_1}(J_p)$ for some $\kappa_1\le\kappa$ if and only if ${\mathfrak A}_{\Pi{\sL}}\in \cN_{\kappa_1}^{p\times p}$
\item
For every $\tau\in\wt\cN_{\kappa_2}^{p\times p}$ of the form
$\tau(\lambda)=\ran\begin{bmatrix}
    \varphi(\lambda)\\\psi(\lambda)
    \end{bmatrix}$
the linear-fractional transformation $\wt\tau=T_{W}[\tau]$
in~\eqref{eq:Lres_phipsi}
belongs to the class $\wt\cN_{\kappa'}^{p\times p}$
with $\kappa'\le\kappa_1+\kappa_2$.
\end{enumerate}
\end{theorem}
\begin{proof}
  The items (i) and (ii) are checked by straightforward calculations.

  (iii) The formula \eqref{eq:KW_NA} ensures that the kernels
  ${\mathsf K}_\omega(\lambda)$ and ${\mathsf{N}}^{\sA}_\omega({\lambda})$
  have the same numbers of negative squares on $\rho(A,\sL)\cap \rho(A_0)$.

  (iv) Let  us set
\begin{equation}
\label{Npair3}
\wt F({\lambda}):=\begin{pmatrix} \wt\psi({\lambda}) \\ \wt\phi({\lambda})\end{pmatrix}
=W({\lambda}) F({\lambda}),\quad
\text{where}\quad
F({\lambda}):=\begin{pmatrix}\psi({\lambda}) \\ \phi({\lambda})\end{pmatrix}.
\end{equation}
Then $\wt\tau=T_{W}[\tau]=\ran\begin{pmatrix} \wt\phi({\lambda}) \\ \wt\psi({\lambda})\end{pmatrix}$ and
it follows from~\eqref{eq:App2:LFT} that the kernel
${\sf N}^{\wt\phi \wt\psi}_\omega ({\lambda})$, see Definition~\ref{def:Nk-family}, admits the representation
\begin{equation}
\label{Nkern3}
\begin{split}
{\sf N}^{\wt\phi \wt\psi}_\omega ({\lambda})
&=-i\frac{\wt F(\omega)^*J_{\sL}\wt F({\lambda})}{{\lambda}-\bar\omega}
=-i\frac{F(\omega)^*W(\omega)^*J_{\cH}W({\lambda})F({\lambda})}{{\lambda}-\bar\omega}\\
&=F(\omega)^*{\sf K}_{W} ({\lambda},\omega)F({\lambda})
+F(\omega)^*{\sf N}_{\phi \psi} ({\lambda},\omega)F({\lambda}).
\end{split}
\end{equation}
Therefore the kernel ${\sf N}_{\wt\phi \wt\psi} ({\lambda},\omega)$ has $\kappa'$ negative squares with $\kappa'\le\kappa_1+\kappa_2$.
Now the properties (2)--(3) in Definition~\ref{def:Nk-family} for $\lambda\in\rho_s(A,\sL)$ are implied by the identities \eqref{Nkern3} and \eqref{Npair3} and \eqref{eq:11.W*W}.
Therefore, $\wt\tau\in \wt\cN_{\kappa'}^{p\times p}$.
\end{proof}
%%%%%%%%%%%%%%%%%%%%%%%%%%%%%%%%%%%%%%%%%%%%%%%%%%%%%%%%%%%%%%%%%%%%%%%%%
%Let $\gh_W$ be the set of holomorphy of $W$.

 In the following theorem we provide an explicit formula for the    $\Pi{\sL}$-resolvent matrix $W_{\Pi{\sL}}({\lambda})$ of $A$
in the sense of Definition \ref{def:11.Pl-res,matr2} that expresses it via
the boundary mappings $\Gamma_j$ and the family $\cG({\cdot})$  given by  %\eqref{eq:11.cG}
\begin{equation}\label{eq:11.cG}
%J_p=\begin{bmatrix}
%    O_{p} & -i I_{p}\\
%    i I_{p} & O_{p}
%    \end{bmatrix},\quad
    {\cG}({\lambda})=\begin{bmatrix}
   -\cQ({\lambda})\\
      \cP({\lambda})
    \end{bmatrix},\quad
    \wh{\cG}({\lambda})=\begin{bmatrix}
   -\wh\cQ({\lambda})\\
      \wh\cP({\lambda})
    \end{bmatrix}.
    %,\quad{\lambda}\in\rho(A,{\sL}).
\end{equation}
\begin{theorem}\label{thm:ResM}
Let $(\dC^p,\wh\Gamma_0,\wh\Gamma_1)$ be an extended boundary triple for $A^{\langle*\rangle}$, let ${\sL}$ be a subspace of $\sH_-$
such that $\rho(A,{\sL})\ne\emptyset$ and let the operator-valued function   $\wh W({\cdot})$ be defined by
\begin{equation}\label{eq:Formula_W}
    \wh W({\lambda}):=\left(\wh\Gamma(\wh{\cG}(\lambda)^{\langle*\rangle})\right)^*
 =\begin{bmatrix}
   -\wh\Gamma_0\wh\cQ(\lambda)^{\langle*\rangle}  &  \wh \Gamma_0\wh\cP(\lambda)^{\langle*\rangle}\\
   -\wh\Gamma_1\wh\cQ(\lambda)^{\langle*\rangle} &   \wh\Gamma_1\wh\cP(\lambda)^{\langle*\rangle}
       \end{bmatrix}^*,\quad
 {\lambda}\in\rho(A,{\sL}).%\quad\textup{where}\quad
\end{equation}
Then:
\begin{enumerate}\def\labelenumi{\rm (\roman{enumi})}
%  \item %[(i)]
%  $\wh W(\cdot)$   satisfies the Kre\u{\i}n-Saakyan identity   %generalized Christoffel identity
%\begin{equation}\label{eq:Christ_Id}
%    J_p-\wh W({\lambda})J_p\wh W(\zeta)^*=-i({\lambda}-\ov\zeta){\sL}({\lambda}){\sL}(\zeta)^{*},\qquad
% {\lambda},\zeta\in\rho(A,{\sL}).
%\end{equation}
  \item %[(ii)]
 The kernel
\begin{equation}\label{kerK1}
{\mathsf K}_\omega(\lambda):=
\frac{ J_p-W(\lambda)J_p W(\omega)^*}{-i(\lambda-\ov\omega)}, \quad\lambda,\omega\in \rho(A,{\sL}),\quad \lambda\ne\overline{\omega}
\end{equation}
admits the factorization
\begin{equation}\label{eq:Christ_Id}
%    J_p-\wh W({\lambda})J_p\wh W(\zeta)^*=-i({\lambda}-\ov\zeta){\sL}({\lambda}){\sL}(\zeta)^{*},
    {\mathsf K}_\omega(\lambda)={\cG}({\lambda}){\cG}(\omega)^{\langle*\rangle}\qquad
 {\lambda},\omega\in\rho(A,{\sL}).
\end{equation}
  \item %(iii)
If the condition \eqref{eq:A_N_k} holds, then $W\in \cW_{\kappa}(J_p)$.
%the kernel ${\mathsf K}_\omega(\lambda)$ has $\kappa$
negative squares.
\medskip
  \item %(iii)
  If, in addition, $\rho_s(A,{\sL}):=\rho(A,{\sL})\cap\overline{\rho(A,{\sL})}\ne
\emptyset$, then $\wh W({\lambda})$ is invertible for $\lambda\in \rho_s(A,{\sL})$ and
coincides  on $\rho_s(A,{\sL})\cap\rho(A_0)$ with the $\Pi{\sL}$-resolvent matrix
 $W_{\Pi{\sL}}({\lambda})$ of $A$.
 \end{enumerate}
\end{theorem}
\begin{proof}
(i) Let us  decompose the kernel ${\mathsf K}_\omega(\lambda)$ into four $p\times p$-blocks: ${\mathsf K}_\omega(\lambda)=\left[{\mathsf K}^{ij}_\omega(\lambda)\right]_{i,j=1}^2$.
Then the identity~\eqref{eq:Christ_Id} is splitted into four identities
\begin{equation}\label{eq:KrSaakId1}
  {\mathsf K}^{11}_\omega(\lambda)=\cQ({\lambda})\cQ(\omega)^{\langle*\rangle},\quad
  {\mathsf K}^{12}_\omega(\lambda)=-\cQ({\lambda})\cP(\omega)^{\langle*\rangle},
\end{equation}
\begin{equation}\label{eq:KrSaakId2}
  {\mathsf K}^{21}_\omega(\lambda)=-\cP({\lambda})\cQ(\omega)^{\langle*\rangle},\quad
  {\mathsf K}^{22}_\omega(\lambda)=\cP({\lambda})\cP(\omega)^{\langle*\rangle}.
\end{equation}

Setting $\lambda,\omega\in{\rho(A,{\sL})}$,
\[
\wh f=\begin{bmatrix}
                          f\\ f'
                        \end{bmatrix}=\wh\cP({\omega})^{\langle*\rangle}u_2,\,\,
                        \wh g=\begin{bmatrix}
                          g\\ g'
                        \end{bmatrix}=\wh\cP(\lambda)^{\langle*\rangle}v_2,\,\quad
u=\begin{bmatrix}
                          0\\ u_2
                        \end{bmatrix},\,\,
                        v=\begin{bmatrix}
                          0 \\ v_2
                        \end{bmatrix}\in\dC^{2p},
\]
 we can rewrite the left hand part of~\eqref{eq:ExtGreen} as
\begin{equation}\label{eq:11.24}
\begin{split}
  [f',g]_\sH-[f,g']_\sH=(\ov\omega-\lambda)[ f,g]_\sH
   =(\ov\omega-\lambda)[\cP(\omega)^{\langle*\rangle}u_2,\cP(\lambda)^{\langle*\rangle}v_2]_\sH.
   \end{split}
\end{equation}
In view of~\eqref{eq:Formula_W} the right hand part of~\eqref{eq:ExtGreen} takes the form
\begin{equation}\label{eq:11.25}
\begin{split}
(\Gamma_0\wh g)^*(\Gamma_1\wh f)-(\Gamma_1\wh g)^*(\Gamma_0\wh f)
%  (\Gamma_1\wh f,\Gamma_0\wh g)_{\dC^p}-(\Gamma_0\wh f,\Gamma_1\wh g)_{\dC^p}
&=i\left(\Gamma\wh\cP(\lambda)^{\langle*\rangle}v_2\right)^* J_p\left(\Gamma\wh\cP(\omega)^{\langle*\rangle}u_2\right)\\
  &=iv^*\wh W({\lambda})J_p\wh W(\omega)^*u
  =(\ov\omega-\lambda)v_2^*{\mathsf K}^{22}_\omega(\lambda)u_2.
   \end{split}
\end{equation}
Comparing of \eqref{eq:11.24} and \eqref{eq:11.25} proves the second identity in~\eqref{eq:KrSaakId2}.

Similarly, setting
%$u=\begin{bmatrix}
%                          0 \\ u_2
%                        \end{bmatrix},$
%                        $v=\begin{bmatrix}
%                          v_1 \\ 0
%                        \end{bmatrix}\in\dC^{2p},$
%\begin{bmatrix}
%                          f\\ f'
%                        \end{bmatrix}
$\wh f=\wh\cP({\omega})^{\langle*\rangle}u_2$,
%                        $\begin{bmatrix}
%                          g\\ g'
%                        \end{bmatrix}
$\wh g=\wh\cQ({\lambda})^{\langle*\rangle}v_1$,
$u_2,v_1\in{\dC^p}$,  $\lambda,\omega\in{\rho(A,{\sL})}$, we can rewrite the left hand part of~\eqref{eq:ExtGreen} as
\begin{equation}\label{eq:W21A}
\begin{split}
  [f',g]_\sH-\langle f,g'\rangle_{+,-}&=(\ov\omega-\lambda)[ f,g]_\sH-\langle f,{L}v_1\rangle_{+,-}\\
   &=(\ov\omega-\lambda)v_1^*
   \cQ(\lambda)J_p\cP(\omega)^*u_2-v_1^*u_2,
   \end{split}
\end{equation}
while the right hand part of~\eqref{eq:ExtGreen} takes the form
\begin{equation}\label{eq:W21B}
\begin{split}
&(\wh\Gamma_0\wh g)^*(\Gamma_1\wh f)-(\wh\Gamma_1\wh g)^*(\Gamma_0\wh f)
=iv_2^*\left( \Gamma\wh\cQ({\lambda})^{\langle*\rangle}\right)^* J_p\left(\Gamma\wh\cP(\omega)^{\langle*\rangle}\right)u_2\\
  &=-i\begin{bmatrix}
                          v_1^* & 0
                        \end{bmatrix}
  \wh W({\lambda})J_p\wh W(\omega)^*\begin{bmatrix}
                         0 \\ u_2
                        \end{bmatrix}
  =-(\ov\omega-\lambda)v_1^*{\mathsf K}^{22}_\omega(\lambda)u_2
  -v_1^*u_2.
   \end{split}
\end{equation}
Comparing of \eqref{eq:W21A} and \eqref{eq:W21B} proves
the second identity  in~\eqref{eq:KrSaakId1}
as well as the first identity in~\eqref{eq:KrSaakId2} since
${\mathsf K}^{21}_\omega(\lambda)=\left({\mathsf K}^{12}_\lambda(\omega)\right)^*$.

To prove the first identity  in~\eqref{eq:KrSaakId1} let us set
$  \lambda,\omega\in{\rho(A,{\sL})}$ and
\[
\wh f=\begin{bmatrix}
                          f\\ f'
                        \end{bmatrix}=-\wh\cQ({\omega})^{\langle*\rangle}u_1,\quad
                        \wh g=\begin{bmatrix}
                          g\\ g'
                        \end{bmatrix}=-\wh\cQ({\lambda})^{\langle*\rangle}v_1,\quad
u=\begin{bmatrix}
                          u_1 \\ 0
                        \end{bmatrix},
                        \quad v=\begin{bmatrix}
                          v_1 \\ 0
                        \end{bmatrix}\in\dC^{2p}.
\]
  Then the left hand part (LHP) of~\eqref{eq:ExtGreen2} equals to
 \[
 \begin{split}
 LHP=
    [f'+{L}u_1,g]_\sH-[f, g'+{L}v_1]_\sH
  =(\ov\omega-\lambda)[\cQ({\omega})^{\langle*\rangle}u_1,\cQ({\lambda})^{\langle*\rangle}v_1]_\sH.
  \end{split}
\]
Notice that, by~\eqref{eq:PQ_pr2},
\[
\left\langle {L}u_1,\cQ({\lambda})^{\langle*\rangle}v_1-\wt\bR {L}v_1\right\rangle_{-,+}
   =\left\langle \cQ({\omega})^{\langle*\rangle}u_1-\wt\bR {L}u_1, {L}v_1\right\rangle_{+,-}=0
\]
and hence the right hand part (RHP) of~\eqref{eq:ExtGreen2} equals to
 \begin{multline}
 RHP=iv^*\wh W({\lambda})J_p\wh W(\omega)^*u
   +\left\langle {L}u_1,-\cQ({\lambda})^{\langle*\rangle}v_1+\wt\bR {L}v_1\right\rangle_{-,+}\\
   -\left\langle -\cQ({\omega})^{\langle*\rangle}u_1+\wt\bR {L}u_1, {L}v_1\right\rangle_{+,-}
   =iv^*\wh W({\lambda})J_p\wh W(\omega)^*u
   =(\ov\omega-\lambda)v_1^*{\mathsf K}^{11}_\omega(\lambda)u_1.
   \end{multline}
This implies the first identity  in~\eqref{eq:KrSaakId1}.
\medskip

(ii) By Lemma \ref{lem:PreresM} (ii), the kernel
    ${\mathsf{N}}^{{\sA}_{\Pi{\sL}}}_\omega({\lambda})$ has $\kappa$ negative
    squares on   $\rho(A,\sL)\cap \rho(A_0)$.
By Theorem \ref{thm:Gres_F1} (iii),  the kernel
  ${\mathsf K}_\omega(\lambda)$ also
  has $\kappa$ negative squares on $\rho(A,\sL)\cap \rho(A_0)$
  and so  $W\in \cW_{\kappa}(J_p)$.
\medskip

(iii) Assume that $\lambda\in\rho_s(A,{\sL})$.
%   Let us show that $\wh W({\lambda})$ is invertible for $z\in\rho_s(A,{\sL})$.
%   First, it follows from Theorem \ref{thm:L_decom} (see \eqref{eq:A+_decom}) that $\ran W({\lambda})^*=\ran\Gamma(\wh{\sL}({\lambda})^*)=\dC^{2p}$. Next we show that   $\ker W({\lambda})^*=\ker \Gamma(\wh{\sL}({\lambda})^*)=\{0\}$. Indeed,   the equality $\Gamma(\wh{\sL}({\lambda})^*)u=0$ implies that
%   $(\wh{\sL}({\lambda})^*)u\in A$, which together with  \eqref{eq:A+_decom} yields $u=0$.
%   Therefore, the operator $W({\lambda})^*$ is invertible, and hence so is $W({\lambda})$
%   for $z\in\rho_s(A,{\sL})$. Inequality \eqref{eq:11.J-contr}
Then~\eqref{eq:Christ_Id}  implies the identities
\begin{equation}\label{eq:11.W*W}
\wh W({\lambda})J_p\wh W(\ov{\lambda})^*=J_p,\quad  \wh W(\ov{\lambda})^*J_p\wh W({\lambda})=J_p, \quad \lambda\in\rho_s(A,{\sL})
\end{equation}
Let us  decompose the matrices $W_{\Pi{\sL}}({\lambda})$ and $\wh W({\lambda})$ into four $p\times p$-blocks:
%  Let us consider the $p\times p$-block representations of the matrices $W_{\Pi{\sL}}({\lambda})$ and $\wh W({\lambda})$
  \[
  W_{\Pi{\sL}}({\lambda})=[w_{ij}({\lambda})]_{i,j=1}^2,\qquad
  \wh W({\lambda})=[\wh w_{ij}({\lambda})]_{i,j=1}^2.
  \]
The equalities in~\eqref{eq:11.W*W} are equivalent to the following conditions on the components of the matrix $ \wh W $:
\begin{equation}\label{LR_39}
\wh w_{21}\wh w_{22}^\#=\wh w_{22}\wh w_{21}^\#,\ \
\wh w_{11}\wh w_{12}^\#=\wh w_{12}\wh w_{11}^\#,\ \
\wh w_{11}\wh w_{22}^\#-\wh w_{12}\wh w_{21}^\#=I_{p}.
\end{equation}
\begin{equation}\label{LR_38}
 \wh w_{12}^\#\wh w_{22}=\wh w_{22}^\#\wh w_{12},\ \
 \wh w_{11}^\#\wh w_{21}=\wh w_{21}^\#\wh w_{11},\ \
\wh w_{11}^\#\wh w_{22}-\wh w_{21}^\#\wh w_{12}=I_p,
\end{equation}

% It follows from \eqref{eq:LresM} and the assumption $z\in\rho_s(A,{\sL})$ that
%  $w_{21}({\lambda})$ and $w_{21}(\ov{\lambda})$ are invertible. Moreover,
By~\eqref{eq:Formula_W}, we have
\begin{equation}\label{eq:11.Wh_w}
  \begin{split}
     \wh w_{11}({\lambda})^* & =-\wh \Gamma_0\wh\cQ({\lambda})^{\langle*\rangle},\qquad
     \wh w_{21}({\lambda})^*  =\wh \Gamma_0\wh\cP(\lambda)^{\langle*\rangle},\\
     \wh w_{12}({\lambda})^* & =-\Gamma_1\wh\cQ({\lambda})^{\langle*\rangle},\qquad
     \wh w_{22}({\lambda})^*  =\Gamma_1\wh\cP(\lambda)^{\langle*\rangle}.
  \end{split}
\end{equation}
Hence, we obtain explicit formulas for the elements of the matrix $\sA_{\Pi{\sL}}({\lambda})=(a_{ij}({\lambda}))_{i,j=1}^2$ by means of $ \wh w_{ij}({\lambda})$. By~\eqref{eq:Lres2A} and \eqref{eq:Formula_W},
\begin{equation} \label{eq:11.a11}
  a_{11}({\lambda})=M({\lambda})=\wh w_{22}(\ov{\lambda})^*\wh w_{21}(\ov{\lambda})^{-*}
  =\wh w_{22}^\#({\lambda})\wh w_{21}^\#({\lambda})^{-1}.
\end{equation}
Next, it follows from \eqref{eq:Lres2A}, \eqref{eq:11.Wh_w} and~\eqref{eq:PQ_pr1} that for all $u\in\dC^p$
\[%begin{equation} \label{eq:11.a21A}
  a_{21}({\lambda})\wh w_{21}(\ov{\lambda})^*u= {L}^{\langle*\rangle}\gamma({\lambda})\Gamma_0\wh\cP(\lambda)^{\langle*\rangle}u
  = {L}^{\langle*\rangle}\cP(\lambda)^{\langle*\rangle}u=u,
\]%end{equation}
whence
\begin{equation} \label{eq:11.a21}
  a_{21}({\lambda})=\wh w_{21}^\#({\lambda})^{-1},\quad
  a_{12}({\lambda})=a_{21}^\#({\lambda})=\wh w_{21}( \lambda)^{-1}.
\end{equation}
To find the expression of $a_{22}({\lambda})= {L}^{\langle*\rangle}\wh R^0_\lambda {L}$ we consider the problem
\[
\wh f=\begin{bmatrix}
  f \\
  {L}u
\end{bmatrix}\in A^{\langle*\rangle}-\lambda I,\quad
\Gamma_0\wh f=0,\quad
\wh f=\begin{bmatrix}
  f \\
  {L}u+\lambda f
\end{bmatrix}\in A^{\langle*\rangle},\quad
\begin{array}{l}
f\in\sH,\\
u\in\dC^p.
\end{array}
\]
It follows from \eqref{eq:PQ*} that this problem has a solution of the form
\[
\wh f=\wh\cQ(\ov{\lambda})^{\langle*\rangle}u-\wh\cP(\ov{\lambda})^{\langle*\rangle}v\quad\text{with}\quad
v\in\dC^p.
\]
In view of \eqref{eq:11.Wh_w},  the equality $\wh\Gamma_0\wh f=0$ implies that
\[
\wh\Gamma_0\wh f=\wh \Gamma_0\wh\cQ(\ov{\lambda})^{\langle*\rangle}u-\Gamma_0\wh\cP(\ov{\lambda})^{\langle*\rangle}v=
-\wh w_{11}^\#({\lambda})u-\wh w_{21}^\#({\lambda})v=0
\]
and hence $v=-\wh w_{21}^\#({\lambda})^{-1}\wh w_{11}^\#({\lambda})u$. Therefore,
\[
f=\cQ(\ov{\lambda})^{\langle*\rangle}u
+\cP(\ov{\lambda})^{\langle*\rangle}\wh w_{21}^\#(\lambda)^{-1}\wh w_{11}(\ov{\lambda})^*u
=\wt R^0_\lambda {L}u,
\]
and, by~\eqref{eq:11.Wh_w}, \eqref{eq:PQ_pr1}, \eqref{eq:PQ_pr2},
\begin{equation}\label{eq:11.a22}
\begin{split}
  a_{22}({\lambda})u&={L}^{\langle*\rangle}(\wt R^0_\lambda-\wt\bR){L}\\
  &={L}^{\langle*\rangle}(\cQ(\ov{\lambda})^{\langle*\rangle}u-\wt\bR {L})u
  +{L}^{\langle*\rangle}\cP(\ov{\lambda})^{\langle*\rangle}
  \wh w_{21}^\#({\lambda})^{-1}\wh w_{11}^\#({\lambda})u\\
  &=\wh w_{21}^\#({\lambda})^{-1}\wh w_{11}^\#({\lambda})u.
\end{split}
\end{equation}
It follows from  \eqref{eq:11.a11}--\eqref{eq:11.a22} that the preresolvent matrix takes the form
\begin{equation}\label{eq:11.Preres_A}
  \sA_{\Pi{\sL}}({\lambda})=\begin{bmatrix}
  \wh w_{22}^\#({\lambda})\wh w_{21}^\#({\lambda})^{-1} & \wh w_{21}(\lambda)^{-1} \\
  \wh w_{21}^\#({\lambda})^{-1}
  & \wh w_{21}^\#({\lambda})^{-1}\wh w_{11}^\#({\lambda})
         \end{bmatrix},\quad \lambda\in\rho_s(A,{\sL}).
\end{equation}
Finally, \eqref{eq:LresM}, %\eqref{eq:11.W*W},
\eqref{eq:11.Preres_A}, \eqref{LR_39} and \eqref{LR_38} imply that for all $\lambda\in\rho_s(A,{\sL})\cap\rho(A_0)$
\[
\begin{split}
   w_{11}({\lambda}) & =a_{22}({\lambda})a_{12}({\lambda})^{-1}=
   \wh w_{21}^\#({\lambda})^{-1}\wh w_{11}^\#({\lambda})\wh w_{21}( \lambda)=\wh w_{11}({\lambda}), \\
   w_{12}({\lambda}) & =a_{22}({\lambda})a_{12}({\lambda})^{-1}a_{11}({\lambda}) -a_{21}({\lambda})\\
      &=\wh w_{11}({\lambda})\wh w_{22}^\#({\lambda})\wh w_{21}^\#({\lambda})^{-1}-\wh w_{21}^\#({\lambda})^{-1}
     =\wh w_{12}({\lambda}), \\
   w_{21}({\lambda}) & =a_{12}({\lambda})^{-1}=\wh w_{21}({\lambda}), \\
   w_{22}({\lambda}) & =a_{12}({\lambda})^{-1}a_{11}({\lambda})=
   \wh w_{21}({\lambda})\wh w_{22}^\#({\lambda})\wh w_{21}^\#({\lambda})^{-1}
   =\wh w_{22}({\lambda}), \\
\end{split}
\]
and hence, $W_{\Pi{\sL}}({\lambda})=\wh W({\lambda})$ for all $\lambda\in\rho_s(A,{\sL})\cap\rho(A_0)$.
\end{proof}

\begin{theorem}\label{prop:Gresolv}
Let $\Pi=  (\dC^p,\Gamma_0,\Gamma_1) $ be a boundary triple for
$A^{*}$ and let ${\sL}\subset\sH_-$ be a gauge for $A$,
% such that $\rho(A,{\sL})\cap\rho({A}_0)\ne\emptyset$,
 let  $W_{\Pi{\sL}}(\lambda)$ be the $\Pi{\sL}$-resolvent matrix of $A$ defined on $\rho(A,{\sL})\cap\rho({A}_0)$ by
\eqref{eq:LresM}, and let the condition \eqref{eq:A_N_k} holds.
 Then  the formula
\begin{equation}\label{eq:Lres}
(r(\lambda)=){L}^{\langle*\rangle}\wh{\mathbf R}_{\lambda}{L}
 =T_{W_{\Pi{\sL}}(\lambda)}[\tau({\lambda})], \quad
 \lambda\in\rho(A,{\sL})\cap\rho(A_0)\cap\rho(\wt A),
%{\CR =(C(\lambda)w_{21}({\lambda})+D(\lambda)w_{22}({\lambda}))^{-1}
% (C(\lambda)w_{11}({\lambda})+D(\lambda)w_{12}({\lambda}))},\quad \lambda\in\dC\setminus\dR,
%%%%%%%%%%%%%%%%%%%%%%%%%%%%%%%%%%%%%%%%%%%%%%%%%%%%%%%%%%%
%(w_{11}({\lambda})\tau({\lambda})+w_{12}({\lambda}))
%(w_{21}({\lambda})\tau({\lambda})+w_{22}({\lambda}))^{-1},
\end{equation}
%where $\lambda\in\rho(A,{\sL})\cap\rho(A_0)\cap\rho(\wt A)$,
establishes a one--to--one correspondence
%${L}^{\langle*\rangle}\wh{\mathbf R}_{\lambda}{L}
$r(\cdot)\longleftrightarrow  \tau(\cdot)$   between the set  %${\sL}\cR_{\wt\kappa}(A)$
of ${\sL}$-regular ${\sL}$-resolvents $r$ of $A$ of index $\wt\kappa$ and the set of
 $\tau\in\wt\cN_{\wt\kappa-\kappa}^{\ptp}$ such that
 $r\in \cN_{\wt\kappa}^{p\times p}$.
  \end{theorem}
\begin{proof} By Lemma~\ref{lem:PreresM}(iii), $a_{12}(\lambda)$ is invertible for $\lambda\in\rho(A,{\sL})\cap\rho(A_0)\cap\rho(\wt A)$.
  It follows from~\eqref{eq:11.Lres1} and~\eqref{eq:LresM} that
%  \eqref{gres0},  \eqref{eq:Lres2A} and~\eqref{eq:LresM} that
\begin{equation}\label{eq:Lres.2}
  \begin{split}
{L}^{\langle*\rangle}\wh{\mathbf R}_{\lambda}{L}
&=a_{22}({\lambda})-a_{21}({\lambda})(\tau({\lambda})+a_{11}({\lambda}))^{-1}
a_{12}({\lambda})\\
&=a_{22}({\lambda})-a_{21}({\lambda})(a_{12}({\lambda})^{-1}\tau({\lambda})+
a_{12}({\lambda})^{-1}a_{11}({\lambda}))^{-1}\\
&=(w_{11}({\lambda})\tau({\lambda})+w_{12}({\lambda}))
(w_{21}({\lambda})\tau({\lambda})+w_{22}({\lambda}))^{-1}.
\end{split}
\end{equation}
Now the statement follows from~Lemma~\ref{lem:Gres_F1}.
%By~\eqref{eq:LresM}, the latter formula
%By Theorem~\ref{krein}, Lemma~\ref{lem:PreresM} and~\eqref{eq:LresM}, for every $\tau\in\wt\cR(\dC^p)$ the
%operator function $W_{\Pi{\sL}}[\tau({z})]$ is an ${\sL}$-resolvent of $A$ and, moreover, every  ${\sL}$-resolvent of $A$ admits such a representation.
%
%Since, by  Theorem~\ref{thm:ResM},  $W_{\Pi{\sL}}$ is invertible for $z\in\rho_{\rm s}(A,{\sL})$, the mapping  $W_{\Pi{\sL}}:\wt\cR(\dC^p)\ni\tau\mapsto P_{{\sL}}(\widetilde A-{z})^{-1}\!\upharpoonright\!{\sL}\in\cR(\dC^p)$
%%\tau$
%is injective.
\end{proof}
\begin{remark}
It may happen that for $W\in\cW_{\kappa_1}(J_p)$ and for some $\tau\in \wt\cN_{\kappa_2}^{\ptp}$
the set $\Lambda_{\varphi,\psi}$ is empty and then the linear fractional transform
$r=T_W[\tau]\in \wt\cN_{\kappa'}^{\ptp}$ is a family of linear relations
with non-trivial multivalued parts.
In this case $r$ is not an $\sL$-resolvent of $A$.
Moreover, if for some $\tau\in \wt\cN_{\kappa_2}^{\ptp}$
the set $\Lambda_{\varphi,\psi}$ is not empty but the index $\kappa'$ is less than
 $\kappa_1+\kappa_2$, then the $\sL$-resolvent $r$ is not $\sL$-regular.
These effects will never occur if $A$ is a symmetric linear relation in a Hilbert space.
\end{remark}
\begin{corollary}\label{cor:Gres_phipsi}
Let in the assumptions of Theorem~\ref{prop:Gresolv} $A$ be a closed symmetric linear relation in a Hilbert space and let  $W_{\Pi{\sL}}(\lambda)$ be the $\Pi{\sL}$-resolvent matrix of $A$ defined by~\eqref{eq:LresM}.
 Then  the formula %\eqref{eq:2.2}
 \begin{equation}\label{eq:LresH}
   {L}^{\langle*\rangle}\wh{\mathbf R}_{z}{L}=T_{W_{\Pi{\sL}}}[\tau({z})]
 \end{equation}
 establishes a one--to--one correspondence  between the set
of all ${\sL}$-resolvents of $A$ of index $0$ and the set of all
Nevanlinna families $\tau\in \wt\cR^{\ptp}=\wt\cN_{0}^{\ptp}$.
\end{corollary}
\begin{proof}
  ADD a proof that $\Lambda_{\varphi,\psi}$ is not empty and
  $T_{W_{\Pi{\sL}}}[\tau({z})]$ is well defined.
\end{proof}
%%%%%%%%%%%%%%%%%%%%%%%%%%%%%%%%%%%%%%%%%%%%%%%%%%%%%%%%%%%%%%%%%%%%%%%%%%%%%%%%%%%%%

\subsection{Left ${\sL}$-resolvent matrix}
%%%%%%%%%%%%%%%%%%%%%%%%%%%%%%%%%%%%%%%%%%%%%%%%%%%%%%%%%%%%%%%%%%%%%%%%%%%%%%%%%%%%%%
For a right $\Pi{\sL}$-resolvent matrix  $W(\lambda):=W_{\Pi{\sL}}({\lambda})=\left[w_{i,j}(\lambda)\right]_{i,j=1}^2$ of $A$
%For a $2\times 2$ block matrix function $W(\lambda)=\left[w_{i,j}(\lambda)\right]_{i,j=1}^2\in \cP(\cJ_d)$
and  $\tau\in \wt\cR(\cH)$ we define the left $\Pi{\sL}$-resolvent matrix of $A$
\begin{equation}\label{eq:LeftResMat}
  W^\ell_{\Pi{\sL}}(\lambda)=[w^\ell_{ij}(\lambda)]_{i,j=1}^2:=W^\#(\lambda) %=W(\overline{\lambda})^*
\end{equation}
and the left linear-fractional transformation
\begin{equation}\label{eq:2.2}
  T_W^{\ell}[\tau]:=(\tau(\lambda)w_{21}^{\ell}(\lambda)+w_{22}^{\ell}(\lambda))^{-1}
  (\tau(\lambda)w_{11}^{\ell}(\lambda)+w_{12}^{\ell}(\lambda)).
\end{equation}
%The matrix function $W^\ell_{\Pi{\sL}}(\lambda):=W^\#(\lambda)=W(\overline{\lambda})^*$ will be called the left $\Pi{\sL}$-resolvent matrix of $A$.
It follows from~\eqref{eq:LresM} and the formulas
\[
{\mathfrak a}_{11}^\#({\lambda})={\mathfrak a}_{11}({\lambda}),\quad
{\mathfrak a}_{12}^\#({\lambda})={\mathfrak a}_{21}({\lambda}),\quad
{\mathfrak a}_{22}^\#({\lambda})={\mathfrak a}_{22}({\lambda}),\quad
\lambda\in\rho_s(A,{\sL}),
\]
that the left $\Pi{\sL}$-resolvent matrix of $A$ takes the form
\begin{equation}\label{eq:LresM_Left}
    W_{\Pi{\sL}}^\ell({\lambda})=\begin{bmatrix}
        {\mathfrak a}_{21}({\lambda})^{-1}{\mathfrak a}_{22}({\lambda}) & {\mathfrak a}_{21}({\lambda})^{-1}\\
    {\mathfrak a}_{11}(\lambda){\mathfrak a}_{21}({\lambda})^{-1}{\mathfrak a}_{22}({\lambda})-{\mathfrak a}_{12}({\lambda})
      &
    {\mathfrak a}_{11}({\lambda}){\mathfrak a}_{21}({\lambda})^{-1} \\
    \end{bmatrix},\quad \lambda\in\rho_s(A,{\sL}).
\end{equation}
Moreover, \eqref{eq:Formula_W} yields another formula for  the left $\Pi{\sL}$-resolvent matrix
\begin{equation}\label{eq:Left_WPQ}
   W_{\Pi{\sL}}^\ell({\lambda}):=\wh\Gamma(\wh{\sL}(\ov\lambda))^{\langle*\rangle})
 =\begin{bmatrix}
   -\wh\Gamma_0\wh\cQ(\ov\lambda)^{\langle*\rangle}  &  \Gamma_0\wh\cP(\ov\lambda)^{\langle*\rangle}\\
   -\wh\Gamma_1\wh\cQ(\ov\lambda)^{\langle*\rangle} &   \Gamma_1\wh\cP(\ov\lambda)^{\langle*\rangle}
       \end{bmatrix},\quad
 {\lambda}\in\rho_s(A,{\sL}).%\quad\textup{where}\quad
\end{equation}

By Theorem~\ref{thm:Gres_F1}, for every $\tau\in \wt\cN_{\kappa_2}^{\ptp}$
 the family $\wt\tau=T_W[\tau]$
belongs to $\wt\cN_{\kappa'}^{\ptp}$ for some $\kappa'\in\dN$.
Since  $\tau^\#(\lambda)=\tau(\lambda)$ and $\wt\tau^\#(\lambda)=\wt\tau(\lambda)$,
the left and right linear-fractional transformations coincide:
\begin{equation}\label{eq:LeftLFT2}
  (T_W^{\ell}[\tau])(\lambda)=(T_W[\tau])^\#(\lambda)=(T_W[\tau])(\lambda),\quad \lambda\in\gh_W\setminus \dR.
\end{equation}
 If $\tau\in \wt\cN_{\kappa_2}^{\ptp}\setminus \cN_{\kappa_2}^{\ptp}$ it is convenient to consider the following  kernel representation of $\tau$,
see~\eqref{LR_50},
\[
\tau(\lambda)=\ker{\begin{bmatrix} {C(z)} & -{D}(z)\end{bmatrix}}
\]
where $\begin{bmatrix}
    C(\lambda)& D(\lambda)
    \end{bmatrix}$ is an $\cN_{\kappa_2}^{\ptp}$-pair connected with
    $\cN_{\kappa_2}^{\ptp}$-family $\tau$
by \eqref{eq:Nkpairs_fam}.

\begin{theorem}\label{thm:Gres_CD}
Let assumptions of Theorem~\ref{prop:Gresolv} hold and let $W_{\Pi{\sL}}^\ell({\lambda})$
be the left $\Pi{\sL}$-resolvent matrix of $A$ defined by~\eqref{eq:LresM_Left}. Then the formula
\begin{equation}\label{eq:2.2L}
  {L}^{\langle*\rangle}\wh {\mathbf R}_{\lambda}{L}=(C(\lambda)w_{12}^\ell(\lambda)+D(\lambda)w_{22}^\ell(\lambda))^{-1}
  (C(\lambda)w_{11}^\ell(\lambda)+D(\lambda)w_{21}^\ell(\lambda)),
\end{equation}
where $\lambda\in\rho(A,{\sL})\setminus\dR$, establishes a one--to--one correspondence  between the set
of ${\sL}$-regular ${\sL}$-resolvents of $A$ of index $\wt\kappa$ and the set of
$\cN_{\wt\kappa-\kappa}^{\ptp}$-pairs $\begin{bmatrix}
    C(\lambda)& D(\lambda)
    \end{bmatrix}$
% $\tau\in\wt\cN_{\wt\kappa-\kappa}^{\ptp}$
such that
 ${L}^{\langle*\rangle}\wh{\mathbf R}_{\lambda}{L}\in \cN_{\wt\kappa}^{p\times p}$.
\end{theorem}
\begin{proof} Using the
connection $\tau(\lambda)=\ker{\begin{bmatrix} {C(\lambda)} & -{D}(\lambda)\end{bmatrix}}$ between all $\cN_{\wt\kappa-\kappa}^{\ptp}$-families
$\tau(\lambda)$
and $\cN_{\wt\kappa-\kappa}^{\ptp}$-pairs $\begin{bmatrix} {C(\lambda)} & {D}(\lambda)\end{bmatrix}$,
see Lemma~\ref{lem:Nkpairs_fam},
and the formula
\[
(\tau({\lambda})+{\mathfrak a}_{11}({\lambda}))^{-1}=(C(\lambda)+D(\lambda)a_{11}(\lambda))^{-1}D(\lambda)
\]
we obtain from
\eqref{eq:11.Lres1} that for $ \lambda\in\rho(A,{\sL})\cap\rho(A_0)\cap\rho(\wt A)$
\begin{equation}\label{eq:Lres.2A}
  \begin{split}
{L}^{\langle*\rangle}\wh{\mathbf R}_{\lambda} {L}
&=a_{22}(\lambda)-a_{21}(\lambda)(C(\lambda)+D(\lambda)a_{11}(\lambda))^{-1}
D(\lambda)a_{12}(\lambda)\\
&=a_{22}(\lambda)-
(C(\lambda)a_{21}(\lambda)^{-1}+D(\lambda)a_{11}(\lambda)a_{21}(\lambda)^{-1})^{-1}D(\lambda)a_{12}(\lambda)\\
&=\left(Ca_{21}^{-1}+Da_{11}a_{21}^{-1}\right)^{-1}
\left(Ca_{21}^{-1}a_{22}+ D\left\{a_{11}a_{21}^{-1}a_{22}-a_{12}\right\}\right)
\end{split}
\end{equation}
which, by~\eqref{eq:LresM_Left}, yields~\eqref{eq:2.2}.
\end{proof}
%%%%%%%%%%%%%%%%%%%%%%%%%%%%%%%%%%%%%%%%%%%%%%%%%%%%%%%%%%%%%%%
\begin{corollary}\label{cor:Gres_CD}
Let in the assumptions of Theorem~\ref{prop:Gresolv} $A$ be a closed symmetric linear relation in a Hilbert space $\sH$ and let  $W_{\Pi{\sL}}(\lambda)$ be the $\Pi{\sL}$-resolvent matrix of $A$ defined by~\eqref{eq:LresM}.
 Then  the formula \eqref{eq:2.2}
 establishes a one--to--one correspondence  between the set
of all ${\sL}$-resolvents of $A$ of index $0$ and the set of all
$\cN_{0}^{\ptp}$-pairs $\begin{bmatrix}
    C(\lambda)& D(\lambda)
    \end{bmatrix}$.
\end{corollary}
%%%%%%%%%%%%%%%%%%%%%%%%%%%%%%%%%%%%%%%%%%%%%%%%%%%%%%%%%%%%%%%%%%%%%%%%%%%%%%%%%%%%%
\section{Canonical systems of differential equations}
Let  $\cJ$ be a ${p\times p}$-matrix such that $\cJ^{*}=\cJ^{-1}=-\cJ$.
Consider the canonical differential equation  %%expression
   \begin{equation}\label{eq:can,eq-n}
\cJ f'(t)+\cF(t)f(t) = \lambda\mathcal H(t)f(t), \quad  t\in(0,l),\quad \lambda \in {\dC},
  \end{equation}
  where $f(\cdot)$  is a ${\dC}^p$-vector function and ${\ptp}$-matrix-functions $\cF(t)$ and $\cH(t)$ satisfy the assumptions
\begin{enumerate}
  \item [(A1)] $\cF(t)$ and $\cH(t)$ are real Hermitian matrix-functions
with entries from $L^1(0,l)$ and  $\cH(t)\ge 0$ for a.e. $t\in (0,l)$
  \item [(A2)]  For each absolutely continuous $f$ such that $\cJ f'(t)+\cF(t)f(t)=0$ the following implication holds
  \[
  \cH(t)f(t)=0\quad \text{a.e. }\Longrightarrow f\equiv 0 \quad \text{on } (0,l).
  \]
\end{enumerate}
% $\cF(t)$ and $\cH(t)$  be real Hermitian %non-negative
%${\ptp}$-matrix-function
%with entries integrable on the interval $I=(0,l)$, such that $\cH(t)\ge 0$ for a.e. $t\in I$.

 % Suppose, in addition, that
%  the system~\eqref{eq:can,eq-n} is definite, i.e.,
%  for each absolutely continuous $f$ such that $\cJ\frac{df}{dt}+\cF(t)f(t)=0$ the following implication holds
%  \[
%  H(t)f(t)=0\quad \text{a.e. }\Longrightarrow f\equiv 0 \quad \text{on } (0,l).
%  \]

Let $\cL^2_\cH(I)$ be the semi-Hilbert space of measurable $\dC^p$-valued functions $f$, such that $\langle f,f\rangle_\cH :=\int_0^l f(t)^*\cH(t)f(t)dt<\infty$.
The semi-definite inner product in $\cL^2_\cH(0,l)$  corresponding to the semi-norm $\|f\|_\cH :=\langle f,f\rangle_\cH^{1/2}$ is defined by
\begin{equation}\label{eq:InnerPr}
  \langle f,g\rangle_\cH:=\int_0^l g(t)^*\cH(t)f(t)dt.
\end{equation}
% As is known  (see~\cite{Kac50}, \cite{BeHaSn20}),
% $\cL^2_\cH(I)$ is complete with respect to the semi-norm
%$\|f\|_\cH = \langle f,f\rangle_\cH^{1/2}$.
Let $L^2_\cH(0,l)$ be the factor-space $L^2_\cH(0,l)=\cL^2_\cH(0,l)/\{f\in \cL^2_\cH(0,l):\langle f,f\rangle_\cH=0\}$.
 For a function $f\in \cL^2_\cH(0,l)$ we denote by $\wt f$ the corresponding class in $L^2_\cH(0,l)$.
  Clearly,  $L^2_\cH(0,l)$ is a Hilbert space with respect to the inner product
  $\langle \wt f, \wt g\rangle_\cH = \langle f,g\rangle_\cH$.

  Define  the maximal relation $A_{\rm max}$ in  $L^2_\cH(0,l)$ by
  \[
 A_{\rm max}=\left\{\begin{bmatrix}
                     \wt  f & \wt g
                    \end{bmatrix}^T\in L^2_\cH(0,l)\times L^2_\cH(0,l):\,
                    \cJ f'+\cF f = \cH g\right\}
  \]
  where $\wt  f\in AC[0,l]$, $ \wt g\in \cL^2_\cH(I)$ are representatives of $f$ and $g$.
Let  the preminimal relation $A'$ is defined as the
 restriction of the maximal relation $A_{\rm max}$ to the elements
 $\begin{bmatrix}
                      f & g
                    \end{bmatrix}^T$ such that $f$ has compact support on $(0,l)$
and let the
 minimal relation $A_{\rm min}$ is defined as  $A_{\rm min}=\overline{A'}$.
 As is known, see~\cite{{LaTe82}},
 $A:=A_{\rm min}$ is a symmetric relation with defect numbers $n_\pm(A)=p$,
 %and  the maximal relation $A_{\rm max}$ is closed,
 $A_{\rm max}=A_{\rm min}^*$ and
  \[
 A_{\rm min}=\left\{\begin{bmatrix}
                     \wt  f & \wt g
                    \end{bmatrix}^T\in A_{\rm max}:\,
                     f(0)= f(l) = 0\right\}.
  \]

 Recall \cite{BeHaSn20}, that for every $g\in\sH$ and $\lambda\in\dC$ the system $\cJ f'+\cF f = {\lambda}\cH f+ \cH g$
has a unique solution $f\in AC[0,l]$.
 Next we denote by $U(\cdot,\lambda)$ the fundamental ${\ptp}$ matrix solution of the initial problem
\begin{equation}\label{Intro_canon_system_second}
   \cJ\frac{dU(t,\lambda)}{dt}+\cF(t)U(t)= {\lambda}\cH U(t,\lambda),\quad \text{a.e. on}\quad(0,l),
   \quad  U(0,\lambda)= I_n.
   %% \quad t\in[0,l].
\end{equation}
The matrix function $U(\lambda):=U(l,\lambda)$ is called the monodromy matrix of the system~\eqref{eq:can,eq-n}.
\begin{proposition}
 \label{prop:14.51}
 Let the assumptions (A1)-(A2) hold and let $A=A_{\rm min}$ be the minimal relation associated with the canonical system~\eqref{eq:can,eq-n}. Then
%Let $\cH(\cdot)\in L^1((0,l); {\dC}^{{\ptp}})$ and $\cH(x)\ge 0$ for a.e. $x\in (0,l)$.
%Then
% $S$ is a simple symmetric linear relation in $L^2_\cH(0,l)$ with defect numbers $(r,r)$, where
%      $r=\text{rank }\Phi(0)$. If, in addition, the system~\eqref{eq:can,eq-n} is definite, then:
%%%see~\cite[Proposition 2.10]{LeschMMM03}
\begin{enumerate}
  \item [\rm(i)]
As a boundary triple
$\Pi = \{{\dC}^p,\Gamma_0,\Gamma_1\}$ for $A_{\rm max}$ one can take
\begin{equation}\label{eq:BT_Can}
\Gamma_0 \begin{pmatrix}
              f \\
               g
             \end{pmatrix}=\frac{1}{\sqrt{2}}\bigl(f(0) + f(l)\bigr),\quad
\Gamma_1 \begin{pmatrix}
              f \\
               g
             \end{pmatrix}= -\frac{1}{\sqrt{2}}\cJ\bigl(f(0) - f(l)\bigr), \quad
\begin{pmatrix}
              f \\
               g
             \end{pmatrix}\in A_{\rm max}.
\end{equation}
  \item [\rm(ii)] The corresponding Weyl function is given by
    \begin{equation}
M(\lambda) = -\cJ(I_n - U(\lambda))(I_n+U(\lambda))^{-1}.
 \end{equation}
 \item [\rm(iii)] The $\gamma$-field is
     \begin{equation}\label{eq:gamma_lambda}
\gamma(\lambda) = \sqrt{2}U(\cdot,\lambda)(I_n+U(\lambda))^{-1}.
 \end{equation}
  \item [\rm(iv)] The adjoint to the $\gamma$-field is
     \begin{equation}\label{eq:gamma_*}
\gamma(\ov\lambda)^*f = \sqrt{2}\int_0^l(I_n+U^\#(\lambda))^{-1}U^\#(s,\lambda)\cH(s)f(s)ds.
 \end{equation}
   \item [\rm(v)] The resolvent $R^0_\lambda=(A_0-\lambda I_\sH)^{-1}$ of the linear relation $A_0:=\ker \Gamma_0$ takes the form
     \begin{equation}\label{eq:R0lambda}
(R^0_\lambda f)(t) = \frac12 U(t,\lambda)\int_0^l\left\{\sgn(s-t) \cJ-\cJ M(\lambda)\cJ\right\}U^\#(s,\lambda)\cH(s)f(s)ds.
 \end{equation}
\end{enumerate}
\end{proposition}

Let $\sH_+=\dom A^*$ be the Hilbert space with the norm~\eqref{eq:S+Norm}. Since  $\Gamma_0$ and $\Gamma_1$ are bounded as operators from $A^*$ to $\dC^p$, see~\cite{DM95}, for every $u\in\dC^p$ the functional
\[
\langle \delta\otimes u,f\rangle_{-,+}=f(0)^*u,\quad f\in \sH_+,
\]
is bounded on $\sH_+$, see also~\cite[Lemma II.4.1]{{LaTe82}}. The subspace
     \begin{equation}\label{eq:G_frak}
{\sL}:=\left\{\delta\otimes u:\, u\in\dC^p\right\}
 \end{equation}
of $\sH_-$ is disjoint with $\ran(\bA-\lambda I_\sH)$ since otherwise there is $u\in\dC^p$
such that $\delta\otimes u\in \ran(\bA-\lambda I_\sH)$ and, by Lemma (vi), we get
\[
0=\langle \delta\otimes u,f_{\ov\lambda}\rangle_{-,+}=f_{\ov\lambda}(0)^*u,\quad\text{for all }\quad
f_{\ov\lambda}\in \sN_{\ov\lambda}.
\]
This implies $u=0$, and therefore the subspaces
${\sL}$ and  $\ran(\bA-\lambda I_\sH)$ are disjoint.
Hence $\rho(A,{\sL})=\dC$ and ${\sL}$ is a gauge for $A$.

\begin{proposition}
 \label{prop:PreresolvMatrix}
 Let  the assumptions of Proposition~\ref{prop:14.51} hold, let
 ${\sL}$ is given by \eqref{eq:G_frak} and let the operator ${L}:\dC^p\to{\sL}$
be defined by ${L}u=\sqrt{2}\delta\otimes u$, $u\in\dC^p$. Then
the $\Pi{\sL}$-preresolvent matrix $\sA_{\Pi{\sL}}(\lambda)$ takes the form
    \begin{equation}\label{eq:PiG_Res_M}
\sA_{\Pi{\sL}}(\lambda) =\begin{bmatrix}
                         M(\lambda)         & 2(I_n+U^\#(\lambda))^{-1} \\
                         2(I_n+U(\lambda))^{-1} & \cJ (-M(\lambda)+\Re  M(i))\cJ
                       \end{bmatrix}, \quad \lambda\in\rho(A_0),
 \end{equation}
 where $M(\lambda) = -\cJ(I_n - U(\lambda))(I_n+U(\lambda))^{-1}$.
\end{proposition}
\begin{proof}
It follows from the equality
\[
u^*\left({L}^{\langle *\rangle}f\right)=\langle {L}f,u\rangle_{-,+}=\sqrt{2} u^* f(0), \quad f\in\sH_+,\quad u\in\dC^p,
\]
that ${L}^{\langle *\rangle}f=\sqrt{2} f(0)$.
  This and the equality \eqref{eq:gamma_lambda} yield the formulas for $a_{21}(\lambda)$
  and $a_{12}(\lambda)$
    \begin{equation}\label{eq:a21}
  a_{21}(\lambda)={L}^{\langle *\rangle}\gamma(\lambda)=2(I_n+U(\lambda))^{-1}, \quad \lambda\in\rho(A_0),
\end{equation}
    \begin{equation}\label{eq:a12}
  a_{12}(\lambda)=a_{21}^\#(\lambda)=2(I_n+U^\#(\lambda))^{-1}, \quad \lambda\in\rho(A_0).
\end{equation}

Next, by~\eqref{eq:R0lambda} and~\eqref{eq:cR},  we get
for $u\in\dC^p$ and $\lambda\in\rho(A_0)$
%\begin{equation}\label{eq:Rlambda_delta}
%  \wt R^0_\lambda(\delta\otimes u)
%=\frac12 U(t,\lambda)(-\cJ-\cJ M(\lambda)\cJ)u
%\end{equation}
%\begin{equation}\label{eq:Rlambda_delta2}
%    \begin{split}
%(\wt R^0_\lambda-\wt\cR)(\delta\otimes u)
%=& \frac12 U(t,\lambda)(-\cJ-\cJ M(\lambda)\cJ)u\\
%&+
%\frac14\left\{ U(t,i)(\cJ+\cJ M(i)\cJ)+U(t,-i)(\cJ+\cJ M(-i)\cJ)\right\}u
%\end{split}
%\end{equation}
and hence, by~\eqref{eq:Lres2A} and~\eqref{eq:RegExtRes},
    \begin{equation}\label{eq:a22}
%    \begin{split}
  a_{22}(\lambda)={L}^{\langle *\rangle}\wh R^0_\lambda {L}u
  ={L}^{\langle *\rangle} (\wt R^0_\lambda-\wt\cR){L}=-\cJ M(\lambda)\cJ+\cJ \text{Re }M(i)\cJ.
%  \end{split}
\end{equation}
\end{proof}

\begin{theorem}
 \label{thm:ResolvMatrix}
 Let  the assumptions of Proposition~\ref{prop:14.51} hold, let
 ${\sL}$ be given by \eqref{eq:G_frak} and let the operator ${L}:\dC^p\to{\sL}$
be defined by ${L}u=\sqrt{2}\delta\otimes u$, $u\in\dC^p$. Then
\begin{enumerate}
\item[(i)]
The left $\Pi{\sL}$-resolvent matrix $W_{\Pi{\sL}}^{\ell}(\lambda)$ takes the form
    \begin{equation}\label{eq:PiG_Res_M2}
W_{\Pi{\sL}}^{\ell}(\lambda) =\frac12\begin{bmatrix}
(U(\lambda)-I_n)\cJ +(U(\lambda)+I_n)K  & (U(\lambda)+I_n)\\
\cJ(U(\lambda)+I_n)\cJ +\cJ(U(\lambda)-I_n)K  & \cJ(U(\lambda)-I_n)
                       \end{bmatrix}. %, \quad \lambda\in\rho(A_0),
 \end{equation}
 where
 \begin{equation}\label{eq:KReM}
   K:=\cJ \Re M(i)\cJ.
 \end{equation}
 \item[(ii)] The formula \eqref{eq:2.2L}
 establishes a one--to--one correspondence  between the set
of all ${\sL}$-resolvents of $A$ of index $0$ and the set of all
$\cN_{0}^{\ptp}$-pairs $\begin{bmatrix}
    C(\lambda)& D(\lambda)
    \end{bmatrix}$.
\medskip
 \item[(iii)] The formula
\begin{equation}\label{eq:GRes_AB}
  {L}^{\langle*\rangle}\wh {\mathbf R}_{\lambda}{L}=(A(\lambda)U(\lambda)+B(\lambda))^{-1}
  (A(\lambda)U(\lambda)-B(\lambda))\cJ+K,
\end{equation}
 establishes a one--to--one correspondence  between the set
of all ${\sL}$-resolvents of $A$ of index $0$ and the set of all
pairs $\begin{bmatrix}
    A(\lambda)& B(\lambda)
    \end{bmatrix}$ of $\ptp$ matrix functions such that
\begin{enumerate}
  \item [(a)] $-i(\dsp A(\lambda)\cJ A(\lambda)^*-B(\lambda)\cJ B(\lambda)^*)\ge 0$ for all $\lambda\in\dC_+$;
\smallskip
%  the kernel
%  \[
%  {\mathsf
%N}_\omega^{CD}(\lambda)
%:=\left\{\begin{array}{ll}
%\frac{\dsp A(\lambda)\cJ A(\omega)^*-B(\lambda)\cJ B(\omega)^*}
%{\dsp\lambda-\ov{\omega}}
%&\quad \textrm{for }\lambda,\,\omega\in\dC_+\cup\dC_-, \ \lambda\ne\ov{\omega}\\
%A'(\lambda)A(\omega)^*-B'(\lambda)B(\omega)^*
%&\quad \textrm{for }\lambda,\,\omega\in\dC_+\cup\dC_-, \ \lambda=\ov{\omega}
%\end{array}\right.
%  \]
%  has $\kappa$ negative squares on $\dC_+\cup\dC_-$;
  \item [(b)] $A(\lambda)\cJ A^\#(\lambda)-B(\lambda)\cJ B^\#(\lambda)=0$ for all $\lambda\in\dC_+\cup\dC_-$;
  \smallskip
  \item[(c)] $\rank \begin{bmatrix}
    A(\lambda)& B(\lambda)
    \end{bmatrix}=p$
  for all $\lambda\in\dC_+\cup\dC_-$.
\end{enumerate}

    \end{enumerate}
 \end{theorem}
\begin{proof}
(i)   By~\eqref{eq:LresM_Left} and~\eqref{eq:PiG_Res_M},
we get
\begin{equation}\label{eq:w11}
\begin{split}
  w_{11}^{\ell}(\lambda)&=a_{21}({\lambda})^{-1}a_{22}({\lambda})\\
  &=\frac12(I_n+U(\lambda))
  \left\{(U(\lambda)-I_n)(I_n+U(\lambda))^{-1}\cJ+K\right\}\\
&=\frac12(U(\lambda)-I_n)\cJ +\frac12(U(\lambda)+I_n)K.
\end{split}
\end{equation}
Since $U^\#(\lambda)=-\cJ U(\lambda)^{-1}\cJ$, we have
\[
(I_n+U^\#(\lambda))^{-1}=-\cJ(I_n+U(\lambda)^{-1})^{-1}\cJ
=-\cJ U(\lambda)(I_n+U(\lambda))^{-1}\cJ.
\]
Hence, by~\eqref{eq:LresM_Left} and~\eqref{eq:PiG_Res_M},
\begin{equation}\label{eq:w21}
\begin{split}
  w_{21}^{\ell}(\lambda)
  &=a_{11}({\lambda})a_{21}({\lambda})^{-1}a_{22}({\lambda})-a_{12}({\lambda})\\
  &=\frac12 \cJ
  \left\{\left[(U(\lambda)-I_n)^2+4U(\lambda)\right](U(\lambda)+I_n)^{-1}\cJ
  +(U(\lambda)-I_n)K\right\}\\
&=\frac12\left\{\cJ(U(\lambda)+I_n)\cJ+\cJ(U(\lambda)-I_n)K\right\}.
\end{split}
\end{equation}
Similarly, we get from~\eqref{eq:LresM_Left} and~\eqref{eq:PiG_Res_M}
\begin{equation}\label{eq:w12}
\begin{split}
  w_{12}^{\ell}(\lambda)= a_{21}({\lambda})^{-1}=\frac12(I_n+U(\lambda))
\end{split}
\end{equation}
and
\begin{equation}\label{eq:w22}
\begin{split}
  w_{22}^{\ell}(\lambda)= a_{11}({\lambda})a_{21}({\lambda})^{-1}=\frac12\cJ(U(\lambda)-I_n).
\end{split}
\end{equation}
Now \eqref{eq:PiG_Res_M2} follows from~\eqref{eq:w11}--\eqref{eq:w22}.

\medskip
(ii) follows from Corollary~\ref{cor:Gres_CD}.
\end{proof}

%%%%%%%%%%%%%%%%%%%%%%%%%%%%%%%%%%%%%%%%%%%%%%%%%%%%%%%%%%%%%%%%%%%%%%%%%%%%%%%%%%%%%%%%%
In the next proposition we present another proof for the formula~\eqref{eq:PiG_Res_M2}
based on the formula~\eqref{eq:Formula_W}.

\begin{proposition}
 \label{prop:PQ_CanSyst}
 Let the assumptions (A1)-(A2) hold, let $A=A_{\rm min}$ be the minimal relation associated with the canonical system~\eqref{eq:can,eq-n}, let the operator ${L}:\dC^p\to{\sL}$
be defined by ${L}u=\sqrt{2}\delta\otimes u$, $u\in\dC^p$, and let the operator functions
$\cP$ and $\cQ$ be defined by~\eqref{eq:P_lambda} and~\eqref{eq:Q_lambda}. Then
\begin{enumerate}
  \item [\rm(i)] The operator functions
$\cP(\lambda)^*$ and $\cQ(\lambda)^*$ take the form
\begin{equation}\label{eq:P*}
  \cP(\lambda)^*=\frac{1}{\sqrt{2}}U(\cdot,\ov\lambda),\quad \lambda\in\dC,
\end{equation}
\begin{equation}\label{eq:Q*}
  \cQ(\lambda)^*=-\frac{1}{\sqrt{2}}U(\cdot,\ov\lambda)(\cJ+K),\quad \lambda\in\dC.
\end{equation}
\item[(ii)]
The left $\Pi{\sL}$-resolvent matrix $W_{\Pi{\sL}}^{\ell}(\lambda)$ takes the form
\eqref{eq:PiG_Res_M2}.
\end{enumerate}
\end{proposition}
\begin{proof}
(i) Since $\cP(\lambda)^{*}u\in\sN_{\ov \lambda}=\{U(\cdot,\ov\lambda)v:\,v\in\dC^p\}$ for all $u\in\dC^p$,
there exists $v\in\dC^p$ such that $\cP(\lambda)^*u=U(\cdot,\ov\lambda)v$.
By the equality $L^{\langle *\rangle}\cP(\lambda)^{*}u=u$
we get $\sqrt{2}v=u$ and hence \eqref{eq:P*} holds.

Next, for $v\in\dC^p$ and $f\in \sH=L^2_\cH$ we obtain
\begin{equation}\label{eq:Q^*}
  (f,\cQ(\lambda)^*v)_{\sH}=v^*(\cQ(\lambda)f)
  =\sqrt{2}v^*\left(\{R_{\lambda}^0-\wh R_{\lambda}^0\Pi_\sL^\lambda\}f\right)(0).
\end{equation}
Since $\gamma(\ov\lambda)^{\langle *\rangle}L= 2(I+U^\#(\lambda))^{-1}$,
we get from~\eqref{eq:gamma_*}
\begin{equation}\label{eq:Pi_lambda}
\begin{split}
\Pi_\sL^\lambda f&=L(\gamma(\ov\lambda)^{\langle *\rangle}L)^{-1}\gamma(\ov\lambda)^{\langle *\rangle} f=\frac{1}{\sqrt{2}}L\int_0^l U^\#(s,\lambda)\cH(s)f(s)ds\\
&=\delta\otimes \int_0^l U^\#(s,\lambda)\cH(s)f(s)ds.
\end{split}
\end{equation}
By~\eqref{eq:R0lambda} and~\eqref{eq:cR},  we get
for $u\in\dC^p$ and $\lambda\in\rho(A_0)$
\begin{equation}\label{eq:Rlambda_delta}
  \wt R^0_\lambda(\delta\otimes u)
=\frac12 U(t,\lambda)(-\cJ-\cJ M(\lambda)\cJ)u,
\end{equation}
\begin{multline}\label{eq:Rlambda_delta2}
(\wt R^0_\lambda-\wt\cR)(\delta\otimes u)
= \frac12 U(t,\lambda)(-\cJ-\cJ M(\lambda)\cJ)u\\
+
\frac14\left\{ U(t,i)(\cJ+\cJ M(i)\cJ)+U(t,-i)(\cJ+\cJ M(-i)\cJ)\right\}u.
\end{multline}
By \eqref{eq:Q^*}, \eqref{eq:Pi_lambda}, \eqref{eq:Rlambda_delta} and \eqref{eq:Rlambda_delta2}, we obtain the equality
\begin{equation}\label{eq:Q_lambda2}
\begin{split}
(f,\cQ(\lambda)^*v)_{\sH}=&
\frac{1}{\sqrt{2}} v^*\int_0^l\left\{\cJ-\cJ M(\lambda)\cJ\right\}U^\#(s,\lambda)\cH(s)f(s)ds\\
&-v^*L^{\langle *\rangle}\wh R_{\ov\lambda}^0
\left\{\delta\otimes \int_0^l U^\#(s,\lambda)\cH(s)f(s)ds\right\}\\
=&\frac{1}{\sqrt{2}} v^*
\int_0^l\left\{\cJ-K\right\} U^\#(s,\lambda)\cH(s)f(s)ds,
\end{split}
\end{equation}
which proves \eqref{eq:Q*}.
\medskip

(ii) By~\eqref{eq:Left_WPQ}, the left $\Pi{\sL}$-resolvent matrix $W_{\Pi{\sL}}^{\ell}(\lambda)=[w_{ij}(\lambda)]_{i,j=1}^2$ is given by
\begin{equation}\label{eq:Left_WPQ2A}
   W_{\Pi{\sL}}^\ell({\lambda})
   =\begin{bmatrix}
   w_{11}(\lambda)  &  w_{12}(\lambda)\\
   w_{21}(\lambda)  &  w_{22}(\lambda)
       \end{bmatrix}
 =\begin{bmatrix}
   -\wh\Gamma_0\wh\cQ(\ov\lambda)^*  &  \Gamma_0\wh\cP(\ov\lambda)^*\\
   -\wh\Gamma_1\wh\cQ(\ov\lambda)^* &   \Gamma_1\wh\cP(\ov\lambda)^*
       \end{bmatrix},\quad
 {\lambda}\in\rho_s(A,{\sL}).%\quad\textup{where}\quad
\end{equation}
%
%Let us find the components of the left $\Pi{\sL}$-resolvent matrix $W_{\Pi{\sL}}^{\ell}(\lambda)=[w_{ij}(\lambda)]_{i,j=1}^2$, see~\eqref{eq:Left_WPQ}.
By \eqref{eq:BT_Can} and \eqref{eq:P*}, we get
\begin{equation}\label{eq:Gamma0P}
\begin{split}
 w_{12}(\lambda)&= \Gamma_0\wh\cP(\ov\lambda)^*v=\frac12 (I_p+U(\lambda))v,\\
  w_{22}(\lambda)&= \Gamma_1\wh\cP(\ov\lambda)^*v=-\frac12 (I_p-U(\lambda))v.
  \end{split}
\end{equation}
Since, by~\eqref{eq:Rlambda_delta}, for $v\in\dC^p$
\begin{equation}\label{eq:wh_Rlambda_delta}
\begin{bmatrix}
  \wt R^0_{\lambda}(\delta\otimes v)\\
  \lambda \wt R^0_{\lambda}(\delta\otimes v)+\delta\otimes v
\end{bmatrix}=
\begin{bmatrix}
    -\frac12 U(\cdot,\lambda)(\cJ+\cJ M(\lambda)\cJ)v \\
 -\frac12 \lambda U(\cdot,\lambda)(\cJ+\cJ M(\lambda)\cJ)v  +\delta\otimes v
\end{bmatrix}.
\end{equation}
It follows from~\eqref{eq:PQ*}, \eqref{eq:Q*}, and~\eqref{eq:Rlambda_delta} that
%and \eqref{eq:wh_Rlambda_delta} that
%$\wh\cQ(\ov\lambda)^*v$, $v\in\dC^p$, admits the representation
\[%begin{equation}\label{eq:Gamma0P}
\wh\cQ(\ov\lambda)^*v=
\begin{bmatrix}
    -\frac{1}{\sqrt{2}} U(\cdot,\lambda)(\cJ+\cJ \text{\rm Re }M(i)\cJ)v \\
 -\frac12 \lambda U(\cdot,\lambda)(\cJ+\cJ \text{\rm Re }M(i)\cJ)v  +\delta\otimes v
\end{bmatrix}=
\begin{bmatrix}
  \wt R^0_{\lambda}(\delta\otimes \wt v)\\
  \lambda \wt R^0_{\lambda}(\delta\otimes \wt v)+\delta\otimes v
\end{bmatrix}
\]%end{equation}
where
\[
\wt v=(\cJ+\cJ M(\lambda)\cJ)^{-1}(\cJ+\cJ \text{\rm Re }M(i)\cJ)v,
\]
and
\begin{equation}\label{eq:wt v-v}
  \wt v-v=(\cJ+\cJ M(\lambda)\cJ)^{-1}\cJ (\text{\rm Re }M(i)- M(\lambda))\cJ v.
\end{equation}
Therefore,
\begin{equation}\label{eq:wh_Q*}
\wh\cQ(\ov\lambda)^*v=\wh f +\wh g,
\end{equation}
where
\begin{equation}\label{eq:wh_f_lambda}
 \wh f=\sqrt{2}
\begin{bmatrix}
  \wt R^0_{\lambda}(\delta\otimes \wt v)\\
  \lambda \wt R^0_{\lambda}(\delta\otimes \wt v)+\delta\otimes v
\end{bmatrix}\in \bA_0,\quad %\text{and}\quad
\wh g=\sqrt{2}\begin{bmatrix}
  \wt R^0_{\lambda}(\delta\otimes (\wt v-v))\\
  \lambda \wt R^0_{\lambda}(\delta\otimes(\wt v-v))
\end{bmatrix}\in\sN_\lambda.
\end{equation}

Now, using the formulas~\eqref{eq:wh_Gamma} for the  extended boundary triple $(\dC^p,\wh\Gamma_0,\wh\Gamma_1)$, \eqref{eq:wh_Rlambda_delta} and~\eqref{eq:wt v-v}, we obtain
\begin{equation}\label{eq:wh_Gamma f}
\wh\Gamma_0\wh f=0,\quad
\wh\Gamma_1\wh f=\sqrt{2}\gamma(\ov\lambda)^{\langle *\rangle}(\delta\otimes v)
=2(I_p+U^\#(\lambda))^{-1}v,
\end{equation}
\begin{equation}\label{eq:wh_Gamma0 g}
\begin{split}
\wh\Gamma_0\wh g&=-\frac{1}{\sqrt{2}}\Gamma_0U(\cdot,\lambda)
(\cJ+\cJ M(\lambda)\cJ)(\wt v-v)\\
&=-\frac{1}{2}(I_p+U(\lambda))
\cJ (\text{\rm Re }M(i)- M(\lambda))\cJ v\\
&=\frac{1}{2}\left\{(I_p-U(\lambda))\cJ -
(I_p+U(\lambda))K\right\} v
\end{split}
\end{equation}
\begin{equation}\label{eq:wh_Gamma1 g}
\begin{split}
\wh\Gamma_1\wh g&=\frac{1}{\sqrt{2}}\Gamma_1U(\cdot,\lambda)
(\cJ+\cJ M(\lambda)\cJ)(\wt v-v)\\
&=\frac{1}{2}\cJ(I_p-U(\lambda))
\cJ (\text{\rm Re }M(i)- M(\lambda))\cJ v\\
&=-\frac{1}{2}(I_p-U(\lambda))^2(I_p+U(\lambda))^{-1}\cJ v +
\frac{1}{2}\cJ(I_p-U(\lambda))K v
\end{split}
\end{equation}
By~\eqref{eq:wh_Gamma f}, \eqref{eq:wh_Gamma0 g} and~\eqref{eq:wh_Gamma1 g}, we get
\begin{equation}\label{eq:Gamma0Q}
   \Gamma_0\wh\cQ(\ov\lambda)^*v=\frac{1}{2}\left\{(I_p-U(\lambda))\cJ -
(I_p+U(\lambda))K\right\} v,
\end{equation}
\begin{equation}\label{eq:Gamma1Q}
\begin{split}
   &\Gamma_1\wh\cQ(\ov\lambda)^*v=\frac{1}{2}\left\{4(I_p+U^\#)^{-1}
   -\cJ(I_p-U)^2(I_p+U)^{-1}\cJ +
\cJ(I_p-U)K\right\} v\\
&=-\frac{1}{2}\cJ\left\{4 U(I_p+U)^{-1}
   +(I_p-U)^2(I_p+U)^{-1}\right\}\cJ v +
\frac{1}{2}\cJ(I_p-U)K v\\
&=-\frac{1}{2}\cJ(I_p+U)\cJ v +
\frac{1}{2}\cJ(I_p-U)K v.
\end{split}
\end{equation}
Now the formula~\eqref{eq:PiG_Res_M2} follows from~\eqref{eq:Gamma0P}, \eqref{eq:wh_Gamma0 g} and \eqref{eq:Gamma1Q}.
\end{proof}


\begin{thebibliography}{99}
\bibitem{Arens}
	R.~Arens, Operational calculus of linear relations, Pacific J.
	Math., \textbf{11} (1961), 9--23.

%\bibitem{ArDy12}
%	D.\,Z.~Arov,  H.~Dym, \emph{Bitangential Direct and Inverse Problems for Systems of Integral and Differential Equations}, Cambridge Univ. Press, Cambridge, 2012.
%
%\bibitem{Atk64}
%	F.~Atkinson, \emph{Discrete and continuous boundary problems},
%	{Mathematics in Science and Engineering, Vol. 8},
%	{Academic Press, New York-London}, {1964}, {xiv+570 p}.
\bibitem{AI86}
T.Y. Azizov, I.S. Iokhvidov, \emph{Linear operators in spaces with indefinite
  metric}.
\newblock John Wiley and Sons, New York (1989)

\bibitem{BeHaSn20}
J. Behrndt, S. Hassi, H. de~Snoo, Boundary value problems, Weyl functions,
  and differential operators.
\newblock Birkh\"auser, Cham (2020)

\bibitem{Ben72}
    C.~Bennewitz, {Symmetric relations on a Hilbert space}, Lect. Notes
    Math., \textbf{280} (1972), 212--218.

\bibitem{BDHS11}
J. Behrndt, V. Derkach, S. Hassi,  H. de~Snoo,
A realization theorem for generalized Nevanlinna families,
Operators and Matrices, Volume 5 (2011), no. 4, 679–706
%\bibitem{Ben89}
%    C.~Bennewitz, {Spectral asymptotics for Sturm--Liouville equations},
%    Proc. Lond. Math. Soc., \textbf{59} (1989), 294--338.
\bibitem{Ber65}
        Yu. M. Berezanskii, \emph{Expansions in eigenfunctions of selfadjoint operators}, ``Naukova
        Dumka'',  Kiev, 1965; English transl. Amer. Math. Soc.  Providence,  RI,1968.

%\bibitem{BerUsShe}
%	Y.\,M. Berezansky, Z.\,G. Sheftel, G.\,F. Us, \emph{Functional analysis}, Vol. I. Operator Theory: Adv. and Appl., \textbf{85}, Birkhäuser Verlag, Basel, 1996, xx+423 pp.
%
%\bibitem{dB2}
%	L.~de~Branges, \emph{Some Hilbert spaces of entire functions II},
%	Trans. Amer. Math. Soc., \textbf{99} (1961), 118--152.
%
%
%\bibitem{dB4}
%	L.~de~Branges, \emph{Some Hilbert spaces of entire functions IV},
%	Trans. Amer. Math. Soc., \textbf{105} (1962), 43--83.
%
%\bibitem{dB5}
%	L.~de Branges, \emph{Hilbert spaces of entire functions},
%	Prentice Hall, Englewood Cliffs, N.J., 1968.
\bibitem{Br76}
V.M. Bruk,  A certain class of boundary value problems with a spectral
  parameter in the boundary condition.
\newblock Mat. Sbornik \textbf{100(142):2}, 210--216 (1976)

\bibitem{Bog68}
J. Bognar,  \emph{Indefinite inner product spaces, Ergebnisse der Mathematik
  und ihrer Grenzgebiete}, vol.~78.
\newblock Springer-Verlag, New York-Heidelberg (1974)
\bibitem{Cal39}
J.W. Calkin,  Abstract symmetric boundary conditions.
\newblock Trans. Amer. Math. Soc. \textbf{45}(3), 369--442 (1939)

\bibitem{Der90}
V. Derkach,  Extensions of a Hermitian operator that is not densely defined in a {K}rein space,
Dokl. Akad. Nauk Ukr. SSR, Ser. A, {\bf 23} (1990), 15--19


\bibitem{D97}
V.\,Derkach, On indefinite moment problem and resolvent matrices of
  {H}ermitian operators in {K}rein spaces.
\newblock Math. Nachr. \textbf{184}(5), 135--166 (1997)

\bibitem{D99}
V. Derkach,   On generalized resolvents of Hermitian relations
in Kre\u{\i}n spaces. Functional analysis, 5. J. Math. Sci. (New York)
{\bf 97} (1999), no. 5: 4420-4460

\bibitem{DD19} V. Derkach and  H. Dym,
        Rigged de Branges-Pontryagin spaces and their application to  extensions and  embedding,
        J. Funct. Anal., {\bf 277} (2019), 31-110.

\bibitem{DD21} V. Derkach and  H. Dym,
        Functional models for entire symmetric operators in rigged de Branges-Pontryagin spaces,
        J. Funct. Anal., {\bf 280} (2021), 108776.
\bibitem{DD24} V. Derkach and  H. Dym,
        Entire  symmetric operators in  de Branges-Pontryagin  spaces and a truncated matrix moment problem,
        Complex Anal. Oper. Theory, {\bf 18}, 153 (2024).

\bibitem{DM91}
	V.\,Derkach, M.\,Malamud, {Generalized resolvents and the boundary value problems for hermitian operators with gaps}, J. Funct. Anal. \textbf{95} (1991), 1--95.

\bibitem{DM95}
	V.\,Derkach, M.\,Malamud, {The extension theory of hermitian operators and the moment problem}, J. Math. Sciences, \textbf{73} (1995), 141--242.

%\bibitem{EckTes13}
%	J.\,Eckhardt, G.\,Teschl,
%	\emph{Sturm--Liouville operators with measure-valued coefficients},
%	J. Anal. Math. \textbf{120} (2013), 151--224.
%
%\bibitem{EGNT13}
%	J.\,Eckhardt, F.\,Gesztesy, R.\,Nichols, and G.\,Teschl,
%	\emph{Weyl--Titchmarsh theory for Sturm--Liouville operators with distributional potentials}, Opuscula Math., \textbf{33} (2013), 467--563.
%
%\bibitem{Ever66}
%    W.\,N.~Everitt,
%    \emph{On the limit point classification of second-order differential operators}, J. Lond. Math. Soc. \textbf{41} (1966), 531--544.
%
%\bibitem{Ever73}
%	W.\,N.~Everitt and M.~Giertz,
%	\emph{A Dirichlet type result for ordinary differential operators},
%	Math. Ann. \textbf{203} (1973), 119--128.
%
%\bibitem{Ever76}
%	W.\,N.~Everitt, \emph{A note on the Dirichlet condition for second-order differential expressions},
%	Canadian J. Math.  \textbf{28} (1976), 312--320.
%
%\bibitem{EvEver}
%	W.\,D.~Evans and W.\,N.~Everitt, A\emph{ return to the Hardy--Littlewood integral inequality},
%	Proc. Roy. Soc. London Ser A, \textbf{380} (1982), 447--486.
%
%\bibitem{Fel57}
%	W.~Feller, \emph{Generalized second order differential operators and their lateral conditions},
%	Illinois J. Math., \textbf{1} (1957), 459--504.
%
%\bibitem{Fel59}
%    W.~Feller, \emph{The birth and death processes as diffusion processes}, J. Math. Pures Appl., \textbf{38} (1959), 301--345.
%
%\bibitem{GhWe20}
%	A. Ghatasheh and R. Weikard,
%	\emph{Spectral theory for systems of ordinary differential equations with distributional coefficients},
%	J. Differential Equations, \textbf{268} (2020), no. 6, 2752--2801.

\bibitem{GG91}
    V.\,I.~Gorbachuk,  M.\,L.~Gorbachuk, \emph{Boundary value problems for operator differential equations}, Kluwer Academic Publishers Group, 1991.

\bibitem{IKL} I. S. Iohvidov, M. G. Kre\u{\i}n, H. Langer,  Introduction to the spectral theory of operators
        in spaces with an indefinite metric. Mathematical Research, {\bf 9}. Akademie-Verlag, Berlin, 1982. 120 pp.

%\bibitem{IMcK65}
%    K.~Ito and H.\,P.~McKean, \emph{Diffusion processes and their sample paths}, Springer, Berlin-Heidelberg-New York, 1965.
\bibitem{KW98}
        M. Kaltenb\"ack, H. Woracek, Generalized resolvent matrices and spaces of analytic functions,
        Integ. Eq. Oper. Th., 32 (1998), 282-318.

%\bibitem{Ka1}
%I.\,S.~Kac, \emph{Linear relations generated by a canonical differential equation of phase dimension 2 and decomposability in eigenfunctions},
%Algebra i Analiz, \textbf{14} (2002), no. 3, 86--120.
%
%
%\bibitem{KacKr58}
%	I.\,S.~Kac, M.\,G.~Kre\u{\i}n, \emph{Criteria for the discreteness of the spectrum of a singular string}, (Russian) Izv.\ Vysš.\ Učebn.\ Zaved.\ Matematika, \textbf{2} (1958), 136--153.

\bibitem{KaKr74}
    I.\,S.~Kac, M.\,G.~Kre\u{\i}n, {$R$-functions--analytic functions mapping the upper halfplane into itself},
    Supplement I to the Russian edition of F.\,V.~Atkinson,
    \emph{Discrete and continuous boundary problems}, Mir, Moscow, 1968 (Russian).
    English translation: Amer. Math. Soc. Transl. Ser. (2)\textbf{103} (1974), 1--18.

%\bibitem{KacK68}
%    I.\,S. Kac and M.\,G. Kre\u{\i}n, \emph{On the spectral functions of the string}, Supplement II to the Russian edition of F.\,V.~Atkinson,
%    \emph{Discrete and continuous boundary problems}, Mir, Moscow, 1968 (Russian).
%    English translation: Amer. Math. Soc. Transl., (2) \textbf{103} (1974), 19--102.
%
%\bibitem{Kalf74}
%	H.~Kalf, \emph{Remarks on some Dirichlet type results for semibounded Sturm--Liouville operators},
%	Math. Ann. \textbf{210} (1974), 197--205.
%
%\bibitem{KWW06}
%	H.~Kaltenback, H.~Winker and H.~Woracek, \emph{Symmetric relations of finite negativity. Operator theory in Kre\u{\i}n spaces and nonlinear eigenvalue problems}, 191--210, Oper. Theory Adv. Appl., \textbf{162}, Birkhauser, Basel, 2006.
%
%\bibitem{KWW07}
%	M.~Kaltenb\"ack, H.~Winkler and H.~Woracek,
%	\emph{Strings, dual strings, and related canonical systems}, Math. Nachr. \textbf{280} (2007), no. 13--14, 1518--1536.
%
%\bibitem{Kas75}
%	Y.~Kasahara, \emph{Spectral theory of generalized second order differential operators and its applications to Markov processes}, Japan. J. Math. (N.S.) \textbf{1} (1975/76), no.1, 67--84.
%
\bibitem{Koc75}
	A.\,N.~Kochubei, On extentions of symmetric operators and symmetric binary relations, Matem. Zametki, \textbf{17}, no. 1 (1975), 41--48.

%\bibitem{Kost13}
%    A.~Kostenko, \emph{The similarity problem for indefinite Sturm--Liouville operators and the HELP inequality},
%    Adv. Math. \textbf{246} (2013), 368--413.
\bibitem{Kr44}
M.G. {K}rein,  On {H}ermitian operators whose deficiency indices are $1$.
\newblock Dokl. Akad. Nauk SSSR \textbf{43}, 339--342 (1944)

\bibitem{Kr46}
M.G. {K}rein, On resolvents of {H}ermitian operator with deficiency index
  $(m,m)$.
\newblock Dokl. Akad. Nauk SSSR \textbf{52}, 657--660 (1946)

\bibitem{Kr49}
M.G. {K}rein,  The fundamental propositions of the theory of representations of
  {H}ermitian operators with deficiency index $(m,m)$.
\newblock Ukrain. Math. Zh. \textbf{1}, 3--66 (1949).
\newblock (Russian) English translation: Amer. Math. Soc. Transl., (2) 97
  (1970), 75-143


\bibitem{KL71}
        M.G.~Kre\u{\i}n and H.~Langer, On defect subspaces and generalized
        resolvents of Hermitian operator in Pontryagin space,  Funct. Anal. Appl., 5 (1971),
        136--146; ibid. 5 (1971), 217--228.
\bibitem{KS70}
M.G. {K}rein, S.N. Saakyan,  The resolvent matrix of a hermitian operator and
  the characteristic functions connected with it.
\newblock Funct, Analysis and Appl. \textbf{4}(3), 103--104 (1970)

%\bibitem{LaSc90}
%	H.~Langer, W.~Schenk,
%	\emph{Generalized Second-Order Differential Operators, Corresponding Gap Diffusions and Superharmonic Transformations},
%	Math. Nachr. \textbf{148} (1990), 746.
\bibitem{LaTe82}
H. Langer and B. Textorius, $L$-Resolvent matrices of symmetric linear relations with equal defect
 numbers; applications to canonical differential relations. Integral Equations Operator Theory \textbf{5}
 (1982), 208-243.

\bibitem{LaTe84}
H. Langer and B. Textorius, Spectral functions of a symmetric linear relation with a directing mapping. I
Proc. Roy. Soc. Edinburgh Sect. A \textbf{97} (1984),  165-176.

\bibitem{LaTe85}
H. Langer and B. Textorius, Spectral functions of a symmetric linear relation with a directing mapping. II
Proc. Roy. Soc. Edinburgh Sect. A \textbf{101} (1985), no. 1-2, 111–124.

%\bibitem{LM03}
%	M.~Lesch, M.~Malamud,
%	\emph{On the deficiency indices and self-adjointness of symmetric Hamiltonian systems},
%	J. Differ. Equ., \textbf{189} (2003), no. 2, 556--615.

\bibitem{M92}
	M.~Malamud,	On the formula of generalized resolvents of a nondensely defined Hermitian operator,
	{Ukr. Mat. Zh.}, \textbf{44} (1992), no.~12, 1658--1688.

%\bibitem{Man68}
%	{P.~Mandl},	\emph{Analytical Treatment of One-dimensional Markov Processes},
%	Academia, Springer, 1968.
%
%
%\bibitem{Mog15}
%	V.~Mogilevskii, \emph{Spectral and pseudospectral functions of Hamiltonian systems: development of the results by Arov-Dym and Sakhnovich},
%	Methods of Funct. Anal. Topology, \textbf{21} (2015), no.~4, 370--402.

\bibitem{Orc67}
    B. Orcutt, Canonical differential equations, Univ Virginia, Ph. D. Thesis, 1969

\bibitem{RB69}
    F.S. Rofe-Beketov, On selfadjoint extensions of differential operators in a space of vector-functions,
    Teor. Funkts., Funkts. Anal. i Prilozhen, 8 (1969), 324.
\bibitem{Sa65}
S.N. Saakyan,  Theory of resolvents of a symmetric operator with infinite
  defect numbers.
\newblock Akad. Nauk Armjan. SSR Dokl. \textbf{41}, 193--198 (1965).
\newblock (Russian)

\bibitem{Sh71}
 Yu.L. Shmul'yan, Representation of Hermitian operators with an ideal reference subspace
Mat. Sb. (N.S.)
\newblock Uspehi Mat. Nauk \textbf{85} (1971), 553--562.

\bibitem{Sh80}
Yu.L. Shmul'yan,  Transformers of linear relations in ${J}$-spaces.
\newblock Funkts. Analiz i Prilogh. \textbf{14}, 39--43 (1980)

\bibitem{ShTs77}
 Yu.L. Shmul'yan,  E.R. Tsekanovski\u{\i}, The theory of biextensions of
  operators in rigged {H}ilbert spaces. unbounded operator colligations and
  characteristic functions.
\newblock Uspehi Mat. Nauk \textbf{32}(5), 69--124 (1977).
\newblock (Russian)

\end{thebibliography}
\end{document}